\documentclass[11pt]{article}

\usepackage{fullpage}
\usepackage{amsfonts}
\usepackage{graphicx}
\usepackage{amsmath,amssymb,amsthm}
\usepackage[ruled]{algorithm2e}
\usepackage{algorithmic}
\usepackage{epstopdf}
\usepackage{psfrag}

\newcommand{\minimize}{\mathop{\mathrm{minimize}{}}}
\newcommand{\maximize}{\mathop{\mathrm{maximize}{}}}
\newcommand{\dom}{\mathrm{dom}}
\newcommand{\eqdef}{~\stackrel{\mathrm{def}}{=}~}
\newcommand{\barR}{\bar{R}}
\newcommand{\order}{\mathcal{O}}
\newcommand{\field}{\mathcal{F}}
\newcommand{\E}{\mathbb E}
\newcommand{\R}{{\mathbb R}}
\newcommand{\norms}[1]{\|{#1}\|}
\newcommand{\ltwo}[1]{\left\|{#1}\right\|_2}
\newcommand{\ltwos}[1]{\norms{#1}_2}
\newcommand{\lones}[1]{\norms{#1}_1}
\newcommand{\yvec}{\ensuremath{y}}
\newcommand{\uvec}{\ensuremath{u}}
\newcommand{\xhat}{x^\star}
\newcommand{\xbar}{\overline x}
\newcommand{\supt}{{(t)}}
\newcommand{\suptp}{{(t+1)}}
\newcommand{\suptm}{{(t-1)}}
\newcommand{\yhat}{y^\star}
\newcommand{\ytilde}{\widetilde y}
\newcommand{\phistar}{\phi^*}
\newcommand{\uhat}{\ensuremath{\uvec^\star}}
\newcommand{\spf}{f}

\newtheorem{theorem}{Theorem}
\newtheorem{lemma}{Lemma}
\newtheorem{corollary}{Corollary}

\newtheorem{assumption}{Assumption}

\title{\vspace{-1ex}Stochastic Primal-Dual Coordinate Method for Regularized Empirical 
  Risk Minimization\thanks{An extended abstract (9 pages) of an early version of this manuscript (arXiv:1409.3257) appeared in the Proceedings of The 32nd International Conference on Machine Learning (ICML), Lille, France, July 2015.}}

\author{
Yuchen Zhang\thanks{
    Department of Electrical Engineering and Computer Science, 
    University of California, Berkekey, CA 94720, USA.
    Email: \texttt{yuczhang@eecs.berkeley.edu}.
    (This work was performed during an internship at Microsoft Research.)} \and
Lin Xiao\thanks{
    Machine Learning Groups, Microsoft Research, Redmond, WA 98053, USA.
    Email: \texttt{lin.xiao@microsoft.com}.}
}

\date{September, 2015}

\begin{document}
\maketitle

\vspace{-3ex}
\begin{abstract}
    We consider a generic convex optimization problem associated with 
    regularized empirical risk minimization of linear predictors. 
    The problem structure allows us to reformulate it as a convex-concave 
    saddle point problem. 
    We propose a stochastic primal-dual coordinate (SPDC) method,
    which alternates between maximizing over a randomly chosen dual variable
    and minimizing over the primal variable.
    An extrapolation step on the primal variable is performed to obtain 
    accelerated convergence rate.
    We also develop a mini-batch version of the SPDC method which facilitates
    parallel computing, and an extension with weighted sampling probabilities 
    on the dual variables, 
    which has a better complexity than uniform sampling on unnormalized data.
    Both theoretically and empirically, we show that the SPDC method has
    comparable or better performance than several state-of-the-art
    optimization methods.
\end{abstract}

\section{Introduction}

We consider a generic convex optimization problem that arises often in 
machine learning: regularized empirical risk minimization (ERM)
of linear predictors.
More specifically,
let $a_1, \ldots, a_n\in\R^d$ be the feature vectors of~$n$ data samples,
$\phi_i:\R\to\R$ be a convex loss function associated with 
the linear prediction $a_i^T x$, for $i=1,\ldots,n$,
and~$g:\R^d\to\R$ be a convex regularization function 
for the predictor~$x\in\R^d$.
Our goal is to solve the following optimization problem:
\begin{align}\label{eqn:min-primal}
    \minimize_{x\in\R^d} \quad \left\{ P(x) \eqdef \frac{1}{n} \sum_{i=1}^n \phi_i(a_i^{T}x)+ g(x) \right\} .
\end{align}

Examples of the above formulation 
include many well-known classification and regression problems.
For binary classification, 
each feature vector $a_i$ is associated with a label $b_i\in\{\pm 1\}$.
We obtain the linear SVM (support vector machine)
by setting $\phi_i(z)=\max\{0, 1-b_i z\}$ (the hinge loss) 
and $g(x) = (\lambda/2)\|x\|_2^2$,
where $\lambda>0$ is a regularization parameter.
Regularized logistic regression is obtained by setting 
$\phi_i(z)=\log(1+\exp(-b_i z))$. 
For linear regression problems, each feature vector $a_i$ is associated with 
a dependent variable $b_i\in\R$, and $\phi_i(z)=(1/2)(z-b_i)^2$.
Then we get ridge regression with $g(x) = (\lambda/2)\|x\|_2^2$,
and the Lasso with $g(x)=\lambda \|x\|_1$.
Further backgrounds on regularized ERM in machine learning and statistics 
can be found, e.g., in the book \cite{HTFbook}.

We are especially interested in developing efficient algorithms for solving
problem~\eqref{eqn:min-primal} when the number of samples~$n$ is very large. 
In this case, evaluating the full gradient or subgradient of the function $P(x)$
is very expensive, thus incremental methods that operate on a single component 
function $\phi_i$ at each iteration can be very attractive.
There have been extensive research on incremental (sub)gradient methods 
(e.g. \cite{Tseng98incremental,BlattHeroGauchman07,NedicBertsekas2001,
Bertsekas2011,Bertsekas12Survey})
as well as variants of the stochastic gradient method
(e.g., \cite{Zhang04,Bottou2010,DuS:09,LLZ:09,Xiao2010RDA}).
While the computational cost per iteration of these methods is only a small
fraction, say $1/n$, of that of the batch gradient methods,
their iteration complexities are much higher
(it takes many more iterations for them to reach the same precision).
In order to better quantify the complexities of various algorithms and 
position our contributions, we need to make some concrete assumptions
and introduce the notion of condition number and batch complexity.

\subsection{Condition number and batch complexity}
\label{sec:cond-number}

Let $\gamma$ and $\lambda$ be two positive real parameters.
We make the following assumption:
\begin{assumption} \label{asmp:smooth-convex}
Each $\phi_i$ is convex and differentiable, and its derivative is 
$(1/\gamma)$-Lipschitz continuous
(same as $\phi_i$ being $(1/\gamma)$-smooth), i.e., 
\[
    |\phi'_i(\alpha) -\phi'_i(\beta)| \leq (1/\gamma)|\alpha-\beta|,
    \quad \forall\, \alpha,\beta\in\R, \quad i=1,\ldots,n.
\]
In addition, the regularization function~$g$ is $\lambda$-strongly convex, 
i.e., 
\[
    g(x) \geq g(y) + g'(y)^T(x-y) + \frac{\lambda}{2}\|x-y\|_2^2,
    \quad \forall\,g'(y)\in\partial g(y), \quad x, y\in\R^n.
\]
\end{assumption}

For example, the logistic loss $\phi_i(z)=\log(1+\exp(-b_i z))$
is $(1/4)$-smooth, 
the squared error $\phi_i(z)=(1/2)(z-b_i)^2$ is $1$-smooth,
and the squared $\ell_2$-norm $g(x)=(\lambda/2)\|x\|_2^2$
is $\lambda$-strongly convex.
The hinge loss $\phi_i(z)=\max\{0, 1-b_i z\}$ and the $\ell_1$-regularization
$g(x)=\lambda\|x\|_1$ do not satisfy Assumption~\ref{asmp:smooth-convex}.
Nevertheless, 
we can treat them using smoothing and strongly convex perturbations,
respectively, so that our algorithm and theoretical framework still apply
(see Section~\ref{sec:non-smooth}).

Under Assumption~\ref{asmp:smooth-convex}, 
the gradient of each component function, 
$\nabla \phi_i(a_i^T x)$, is also Lipschitz continuous, with Lipschitz constant
$L_i=\|a_i\|_2^2/\gamma\leq R^2/\gamma$, where $R = \max_i \|a_i\|_2$.
In other words, each $\phi_i(a_i^T x)$ is $(R^2/\gamma)$-smooth.
We define a \emph{condition number} 
\begin{equation} \label{eqn:condition-number}
  \kappa=R^2/(\lambda\gamma),
\end{equation}
and focus on ill-conditioned problems where $\kappa\gg 1$.
In the statistical learning context, the regularization parameter~$\lambda$ 
is usually on the order of $1/\sqrt{n}$ or $1/n$ 
(e.g., \cite{BousquetElisseeff02}),
thus $\kappa$ is on the order of $\sqrt{n}$ or~$n$. 
It can be even larger if the strong convexity in~$g$ is added purely for
numerical regularization purposes (see Section~\ref{sec:non-smooth}).
We note that the actual conditioning of problem~\eqref{eqn:min-primal}
may be better than~$\kappa$, if the empirical loss function 
$(1/n)\sum_{i=1}^n \phi_i(a_i^T x)$ by itself is strongly convex.
In those cases, our complexity estimates in terms of~$\kappa$ 
can be loose (upper bounds), but they are still useful in comparing 
different algorithms for solving the same given problem.

Let $P^\star$ be the optimal value of problem~\eqref{eqn:min-primal}, i.e., 
$P^\star=\min_{x\in\R^d} P(x)$.
In order to find an approximate solution~$\hat{x}$
satisfying $P(\hat{x})-P^\star\leq\epsilon$,
the classical full gradient method and its proximal variants 
require $\order((1+\kappa)\log(1/\epsilon))$ iterations 
(e.g., \cite{Nesterov04book,Nesterov13composite}).
Accelerated full gradient (AFG) methods 
\cite{Nesterov04book,Tseng08, BeckTeboulle09,Nesterov13composite} 
enjoy the improved iteration complexity 
$\order((1+\sqrt{\kappa})\log(1/\epsilon))$.\footnote{
    For the analysis of full gradient methods, we should use
    $(R^2/\gamma + \lambda)/\lambda=1+\kappa$ as the condition number of
    problem~\eqref{eqn:min-primal}; 
    see \cite[Section~5.1]{Nesterov13composite}.
    Here we used the upper bound $\sqrt{1+\kappa} < 1 + \sqrt{\kappa}$
    for easy comparison.
    When $\kappa\gg 1$, the additive constant~$1$ can be dropped.
}
However, each iteration of these batch methods requires a full pass over 
the dataset, computing the gradient of each component function
and forming their average, which cost $\order(nd)$ operations
(assuming the features vectors $a_i\in\R^d$ are dense).
In contrast, the stochastic gradient method and its proximal variants 
operate on one single component $\phi_i(a_i^T x)$ (chosen randomly)
at each iteration, which only costs $\order(d)$.
But their iteration complexities are far worse.
Under Assumption~\ref{asmp:smooth-convex}, 
it takes them $\order(\kappa/\epsilon)$ iterations to find
an~$\hat{x}$ such that $\E[P(\hat{x})-P^\star]\leq\epsilon$,
where the expectation is with respect to the random choices 
made at all the iterations
(see, e.g., \cite{PolyakJuditsky92,NemirovskiJLS09,DuS:09,LLZ:09,Xiao2010RDA}).

To make fair comparisons with batch methods, we measure the complexity
of stochastic or incremental gradient methods 
in terms of the number of equivalent passes over the dataset required to reach 
an expected precision~$\epsilon$.
We call this measure the \emph{batch complexity},
which are usually obtained by dividing their iteration complexities by~$n$.
For example, the batch complexity of the stochastic gradient method is
$\order(\kappa/(n\epsilon))$.
The batch complexities of full gradient methods are the same as their
iteration complexities.

By carefully exploiting the finite average structure in~\eqref{eqn:min-primal}
and other similar problems, several recent work 
\cite{LeRouxSchmidtBach12,SSZhang13SDCA,JohnsonZhang13,XiaoZhang14ProxSVRG,
NIPS2014SAGA} 
proposed new variants of the stochastic gradient or
dual coordinate ascent methods 
and obtained the iteration complexity $\order((n+\kappa)\log(1/\epsilon))$.
Since their computational cost per iteration is $\order(d)$, 
the equivalent batch complexity is $1/n$ of their iteration complexity, i.e.,
$\order((1+\kappa/n)\log(1/\epsilon))$.
This complexity has much weaker dependence on~$n$ than the full gradient 
methods, and also much weaker dependence on~$\epsilon$ than the 
stochastic gradient methods.

In this paper, we propose a stochastic primal-dual coordinate (SPDC) method,
which has the iteration complexity
\[
    \order\bigl((n+\sqrt{\kappa n})\log(1/\epsilon)\bigr),
\]
or equivalently, the batch complexity
\begin{equation}\label{eqn:accl-complexity}
    \order\bigl((1+\sqrt{\kappa/n})\log(1/\epsilon)\bigr).
\end{equation}
When $\kappa>n$,
this is lower than the $\order((1+\kappa/n)\log(1/\epsilon))$ batch complexity
mentioned above.
Indeed, it is very close to a lower bound for minimizing finite sums recently
established in \cite{AgarwalBottou15}.

\subsection{Outline of the paper}

Our approach is based on reformulating problem~\eqref{eqn:min-primal} 
as a convex-concave saddle point problem, 
and then devising a primal-dual algorithm to approximate the saddle point.
More specifically, we replace each component function $\phi_i(a_i^Tx)$ 
through convex conjugation, i.e., 
\[
\phi_i(a_i^T x) 
= \sup_{y_i\in\R} \left\{ y_i\langle a_i, x \rangle - \phi_i^*(y_i) \right\},
\]
where $\phi_i^*(y_i) = \sup_{\alpha\in\R}\{\alpha y_i - \phi_i(\alpha)\}$,
and $\langle a_i, x\rangle$ denotes the inner product of $a_i$ and~$x$
(which is the same as $a_i^T x$, but is more convenient for later presentation).
This leads to a convex-concave saddle point problem
\begin{align}\label{eqn:min-max-saddle}
\min_{x\in \R^d} ~\max_{\yvec\in\R^n} 
~\left\{ f(x,y) \eqdef \frac{1}{n}\sum_{i=1}^n 
\bigl( \yvec_i \langle a_i, x \rangle - \phistar_i(\yvec_i) \bigr)
+ g(x) \right\}.
\end{align}
Under Assumption~\ref{asmp:smooth-convex}, 
each $\phistar_i$ is $\gamma$-strongly convex (since $\phi_i$ is 
$(1/\gamma)$-smooth; see, e.g., \cite[Theorem~4.2.2]{HUL01book})
and~$g$ is $\lambda$-strongly convex.
As a consequence, the saddle point problem~\eqref{eqn:min-max-saddle}
has a unique solution, which we denote by $(x^\star, y^\star)$. 

In Section~\ref{sec:spdc-method},
we present the SPDC method as well as its convergence analysis.
It alternates between maximizing~$f$ over a randomly chosen dual 
coordinate $y_i$ and minimizing~$f$ over the primal variable~$x$.
In order to accelerate the convergence,
an extrapolation step is applied in updating the primal variable~$x$.
We also give a mini-batch SPDC algorithm 
which is well suited for parallel computing.

In Section~\ref{sec:non-smooth} and Section~\ref{sec:non-uniform}, 
we present two extensions of the SPDC method.
We first explain how to solve problem~\eqref{eqn:min-primal} 
when Assumption~\ref{asmp:smooth-convex} does not hold.
The idea is to apply small regularizations to the saddle point function
so that SPDC can still be applied, 
which results in accelerated sublinear rates.
The second extension is a SPDC method with non-uniform sampling.
The batch complexity of this algorithm has the same form 
as~\eqref{eqn:accl-complexity}, 
but with $\kappa=\bar{R}/(\lambda\gamma)$,
where $\bar{R}=\frac{1}{n}\sum_{i=1}^n\|a_i\|$,
which can be much smaller than $R=\max_i\|a_i\|$
if there is considerable variation in the norms $\|a_i\|$.

In Section~\ref{sec:related-work}, we discuss related work.
In particular, the SPDC method can be viewed as a coordinate-update 
extension of the batch primal-dual algorithm developed by 
Chambolle and Pock~\cite{ChambollePock11}.
We also discuss two very recent work \cite{SSZhang13acclSDCA,LinLuXiao14apcg}
which achieve the same batch complexity~\eqref{eqn:accl-complexity}.

In Section~\ref{sec:sparse-impl}, we discuss efficient implementation 
of the SPDC method when the feature vectors $a_i$ are sparse.
We focus on two popular cases:
when~$g$ is a squared $\ell_2$-norm penalty
and when~$g$ is an $\ell_1+\ell_2$ penalty.
We show that the computational cost per iteration of SPDC only depends
on the number of non-zero elements in the feature vectors.

In Section~\ref{sec:experiments}, we present experiment results comparing
SPDC with several state-of-the-art optimization methods,
including both batch algorithms and randomized incremental and coordinate
gradient methods.
On all scenarios we tested, SPDC has comparable or better performance.

\section{The SPDC method}
\label{sec:spdc-method}

\begin{algorithm}[t]
\DontPrintSemicolon
\KwIn{parameters $\tau,\sigma,\theta\in \R_+$, 
    number of iterations $T$, and initial points $x^{(0)}$ and $y^{(0)}$.}
\vspace{2pt}
\textbf{Initialize:} 
$\xbar^{(0)} = x^{(0)}$, $\uvec^{(0)} = (1/n)\sum_{i=1}^n y^{(0)}_i a_i$. \;
\vspace{2pt}
\For{$t=0,1,2,\dots,T-1$}
{
\vspace{2pt}
	Pick an index $k\in \{1,2,\dots,n\}$ uniformly at random, and execute the following updates:
	\begin{align}
		\yvec_i^\suptp &= \left\{ \begin{array}{ll}
            \arg\max_{\beta\in \R} \left\{ \beta \langle a_i, \xbar^\supt \rangle - \phistar_{i}(\beta) - \frac{1}{2\sigma}(\beta-\yvec_i^\supt)^2 \right\}
            & \mbox{if}~i = k, \\
            \yvec_i^\supt & \mbox{if}~i\neq k,
		\end{array} \right. \label{eqn:spdc-maximize}\\
        x^\suptp &= \arg\min_{x\in \R^d} \left\{ g(x) + \left\langle \uvec^\supt + (\yvec_k^\suptp - \yvec_k^\supt) a_k, ~x \right\rangle +  \frac{\ltwos{x-x^\supt}^2}{2\tau} \right\},
		\label{eqn:spdc-minimize}\\
        \uvec^\suptp &= \uvec^\supt + \frac{1}{n}(\yvec_k^\suptp - \yvec_k^\supt) a_k, \label{eqn:spdc-update-b}\\
		\xbar^\suptp &= x^\suptp + \theta (x^\suptp - x^\supt). \label{eqn:spdc-update-xbar}
	\end{align}
	\vspace{-20pt}
}
\KwOut{$x^{(T)}$ and $\yvec^{(T)}$}
\caption{The SPDC method}
\label{alg:fast-spdc}
\end{algorithm}

\begin{algorithm}[t]
\DontPrintSemicolon
\KwIn{mini-batch size~$m$, parameters $\tau,\sigma,\theta\in \R_+$,
    number of iterations $T$, and $x^{(0)}$ and $y^{(0)}$.}
\vspace{2pt}
\textbf{Initialize:} 
$\xbar^{(0)} = x^{(0)}$, $\uvec^{(0)} = (1/n)\sum_{i=1}^n y^{(0)}_i a_i$. \;
\vspace{2pt}
\For{$t=0,1,2,\dots,T-1$}
{
\vspace{2pt}
	Randomly pick a subset of indices $K\subset \{1,2,\dots,n\}$ of size $m$, 
    such that the probability of each index being picked is equal to $m/n$.
	Execute the following updates:
	\begin{align}
		\yvec_i^\suptp &= \left\{ \begin{array}{ll}
            \arg\max_{\beta\in \R} \left\{\beta \langle a_i, \xbar^\supt \rangle - \phistar_{i}(\beta) - \frac{1}{2\sigma}(\beta-\yvec_i^\supt)^2 \right\}
            & \mbox{if}~i \in K,\\
            \yvec_i^\supt & \mbox{if}~i\notin K,
		\end{array} \right. \label{eqn:spdc-minibatch-maximize}\\
		\uvec^\suptp &= \uvec^\supt + \frac{1}{n}\sum_{k\in K} (\yvec_k^\suptp - \yvec_k^\supt) a_k, \nonumber\\
        x^\suptp &= \arg\min_{x\in \R^d} \left\{ g(x) + \left\langle \uvec^\supt + \frac{n}{m}(u^\suptp-u^\supt), ~x \right\rangle +  \frac{\ltwos{x-x^\supt}^2}{2\tau} \right\},
		\label{eqn:spdc-minibatch-minimize}\\
		\xbar^\suptp &= x^\suptp + \theta (x^\suptp - x^\supt).\nonumber
	\end{align}
	\vspace{-20pt}
}
\KwOut{$x^{(T)}$ and $\yvec^{(T)}$}
\caption{The Mini-Batch SPDC method}
\label{alg:spdc-minibatch}
\end{algorithm}

In this section, we describe and analyze the 
Stochastic Primal-Dual Coordinate~(SPDC) method.
The basic idea of SPDC is quite simple:
to approach the saddle point of~$f(x,y)$ defined in~\eqref{eqn:min-max-saddle},
we alternatively maximize~$f$ with respect to~$\yvec$, 
and minimize~$f$ with respect to~$x$. 
Since the dual vector~$\yvec$ has~$n$ coordinates and each coordinate is 
associated with a feature vector $a_i\in\R^d$,
maximizing~$f$ with respect to~$\yvec$ takes $\order(n d)$ computation,
which can be very expensive if~$n$ is large. 
We reduce the computational cost by randomly picking a single 
coordinate of~$\yvec$ at a time, and maximizing~$f$ only with respect to 
this coordinate. 
Consequently, the computational cost of each iteration is $\order(d)$.

We give the details of the SPDC method in Algorithm~\ref{alg:fast-spdc}.
The dual coordinate update and primal vector update are given in
equations~\eqref{eqn:spdc-maximize} and~\eqref{eqn:spdc-minimize}
respectively.
Instead of maximizing~$f$ over~$y_k$ and minimizing~$f$ over~$x$ directly, 
we add two quadratic regularization terms to penalize 
$y^\suptp_k$ and $x^\suptp$ from deviating from $y^\supt_k$ and $x^\supt$.
The parameters~$\sigma$ and~$\tau$ control their regularization strength,
which we will specify in the convergence analysis 
(Theorem~\ref{thm:spdc-convergence}).
Moreover, we introduce two auxiliary variables~$\uvec^\supt$ and~$\xbar^\supt$. 
From the initialization $u^{(0)}=(1/n)\sum_{i=1}^n y^{(0)}_i a_i$
and the update rules~\eqref{eqn:spdc-maximize}
and~\eqref{eqn:spdc-update-b}, we have
\[
    u^{(t)} = \frac{1}{n}\sum_{i=1}^n y^{(t)}_i a_i, \qquad t=0,\ldots, T.
\]
Equation~\eqref{eqn:spdc-update-xbar} obtains
$\xbar^\suptp$ based on extrapolation from $x^\supt$ and $x^\suptp$. 
This step is similar to Nesterov's acceleration 
technique~\cite[Section~2.2]{Nesterov04book}, 
and yields faster convergence rate. 

The Mini-Batch SPDC method in Algorithm~\ref{alg:spdc-minibatch}
is a natural extension of SPDC in Algorithm~\ref{alg:fast-spdc}.
The difference between these two algorithms is that, the Mini-Batch SPDC method
may simultaneously select more than one dual coordinates to update.
Let~$m$ be the mini-batch size.
During each iteration, the Mini-Batch SPDC method randomly picks a subset of
indices $K\subset\{1,\ldots,n\}$ of size~$m$, 
such that the probability of each index being picked is equal to $m/n$.
The following is a simple procedure to achieve this.
First, partition the set of indices into~$m$ disjoint subsets, so that the 
cardinality of each subset is equal to $n/m$ (assuming~$m$ divides~$n$). 
Then, during each iteration, 
randomly select a single index from each subset and add it to~$K$. 
Other approaches for mini-batch selection are also possible;
see the discussions in \cite{RichtarikTakac12bigdata}.

In Algorithm~\ref{alg:spdc-minibatch}, 
we also switched the order of updating $x^\suptp$ and $u^\suptp$
(comparing with Algorithm~\ref{alg:fast-spdc}),
to better illustrate that $x^\suptp$ is obtained based on an extrapolation
from $u^\supt$ to $u^\suptp$.
However, this form is not recommended in implementation,
because $u^\supt$ is usually a dense vector even if
the feature vectors $a_k$ are sparse.
Details on efficient implementation of SPDC are given in 
Section~\ref{sec:sparse-impl}.
In the following discussion, we do not make sparseness assumptions.

With a single processor, each iteration of Algorithm~\ref{alg:spdc-minibatch} 
takes $\order(m d)$ time to accomplish. 
Since the updates of each coordinate $\yvec_k$ are independent of each other, 
we can use parallel computing to accelerate the Mini-Batch SPDC method. 
Concretely, we can use~$m$ processors to update the~$m$ coordinates
in the subset~$K$ in parallel, then aggregate them to update $x^\suptp$. 
In terms of wall-clock time, each iteration takes $\order(d)$ time, 
which is the same as running one iteration of the basic SPDC algorithm. 
Not surprisingly, we will show that the Mini-Batch SPDC algorithm converges 
faster than SPDC in terms of the iteration complexity,
because it processes multiple dual coordinates in a single iteration.

\subsection{Convergence analysis}
\label{sec:saddle-approx-analysis}

Since the basic SPDC algorithm is a special case of  Mini-Batch SPDC 
with $m=1$, 
we only present a convergence theorem for the mini-batch version.
The expectations in the following results are taken with respect to 
the random variables $\{ K^{(0)},\ldots,K^{(T-1)} \}$, where
$K^{(t)}$ denotes the random subset $K\subset\{1,\ldots,n\}$ 
picked at the $t$-th iteration of the SPDC method.

\begin{theorem}\label{thm:spdc-convergence}
Assume that each $\phi_i$ is $(1/\gamma)$-smooth and 
$g$ is $\lambda$-strongly convex (Assumption~\ref{asmp:smooth-convex}).
Let $(x^\star, y^\star)$ be the unique saddle point of~$f$ 
defined in~\eqref{eqn:min-max-saddle},
$R = \max \{\ltwo{a_1},\ldots,\ltwo{a_n}\}$, and define
\begin{align}
\Delta^{(t)} 
=& \left( \frac{1}{2\tau} + \frac{\lambda}{2}\right)\ltwos{x^{(t)} - \xhat}^2 + \left( \frac{1}{4\sigma} + \frac{\gamma}{2}\right)\frac{\ltwos{\yvec^{(t)} - \yhat}^2}{m} \nonumber\\
& + \spf(x^{(t)}, y^\star)-\spf(x^\star,y^\star) + \frac{n}{m}\left( \spf(x^\star,y^\star) - \spf(x^\star, y^{(t)})\right) .
\label{eqn:Delta-def}
\end{align}
If the parameters $\tau,\sigma$ and~$\theta$ in 
Algorithm~\ref{alg:spdc-minibatch} are chosen such that
\begin{equation}\label{eqn:tau-sigma-theta}
	\tau = \frac{1}{R} \sqrt{\frac{m\gamma}{n\lambda}},\qquad 
    \sigma = \frac{1}{R} \sqrt{\frac{n\lambda}{m\gamma}}, \qquad
    \theta = 1 - \left(\frac{n}{m} + R\sqrt{\frac{n}{m\lambda\gamma}}\right)^{-1},
\end{equation}
then for each $t\geq 1$, the Mini-Batch SPDC algorithm achieves
\begin{align*}
  \E[ \Delta^{(t)}] \,\leq\, \theta^{t} \left(\Delta^{(0)} +\frac{\ltwos{y^{(0)}-y^\star}^2}{4m\sigma}\right) .
\end{align*}
\end{theorem}
\vspace{5pt}

The proof of Theorem~\ref{thm:spdc-convergence} is given in 
Appendix~\ref{sec:thm1-proof}.
The following corollary establishes the expected iteration complexity 
of Mini-Batch SPDC for obtaining an $\epsilon$-accurate solution.

\begin{corollary}\label{coro:spdc-complexity}
Suppose Assumption~\ref{asmp:smooth-convex} holds
and the parameters $\tau$, $\sigma$ and $\theta$ are set as 
in~\eqref{eqn:tau-sigma-theta}. 
In order for Algorithm~\ref{alg:spdc-minibatch} to obtain
\begin{equation}\label{eqn:saddle-epsilon}
    \E[\|x^{(T)}-\xhat\|_2^2] \leq\epsilon, \qquad
    \E[\|y^{(T)}-\yhat\|_2^2] \leq\epsilon, 
\end{equation}
it suffices to have the number of iterations~$T$ satisfy
    \[
        T \geq \bigg(\frac{n}{m}+R\sqrt{\frac{n}{m\lambda\gamma}}\bigg)\log\left(\frac{C}{\epsilon}\right),
    \]
    where
    \[
      C = \frac{\Delta^{(0)}+\ltwo{y^\supt-y^\star}^2/(4m\sigma)}{\min\bigl\{ 1/(2\tau)+\lambda/2,~ (1/(4\sigma) + \gamma/2)/m \bigr\}}.
    \]
\end{corollary}
\begin{proof}
  By Theorem~\ref{thm:spdc-convergence}, for each $t>0$, we have
    $\E[\|x^{(t)}-\xhat\|_2^2] \leq \theta^t C$ and
    $\E[\|y^{(t)}-\yhat\|_2^2] \leq \theta^t C$.
    To obtain~\eqref{eqn:saddle-epsilon},
    it suffices to ensure that $\theta^T C\leq \epsilon$, which is equivalent to
    \[
        T ~\geq~ \frac{\log(C/\epsilon)}{-\log(\theta)} 
        ~=~ \frac{\log(C/\epsilon)}{-\log\Bigl(1-\Bigl((n/m)+R\sqrt{(n/m)/(\lambda\gamma)}\Bigr)^{-1}\Bigr)} .
    \]
    Applying the inequality $-\log(1-x)\geq x$ to the denominator above 
    completes the proof.
\end{proof}

Recall the definition of the condition number $\kappa=R^2/(\lambda\gamma)$
in~\eqref{eqn:condition-number}.
Corollary~\ref{coro:spdc-complexity} establishes that the iteration complexity
of the Mini-Batch SPDC method for achieving~\eqref{eqn:saddle-epsilon} is
\[
\order\left(\bigl( (n/m) +\sqrt{\kappa(n/m)}\bigr)\log(1/\epsilon) \right).
\]
So a larger batch size~$m$ leads to less number of iterations.
In the extreme case of $n=m$, we obtain a full batch algorithm, which has
iteration or batch complexity $\order((1+\sqrt{\kappa})\log(1/\epsilon))$.
This complexity is also shared by the AFG methods
\cite{Nesterov04book,Nesterov13composite} (see Section~\ref{sec:cond-number}),
as well as the batch primal-dual algorithm of Chambolle and Pock
\cite{ChambollePock11} 
(see discussions on related work in Section~\ref{sec:related-work}).

Since an equivalent pass over the dataset corresponds to $n/m$ iterations,
the batch complexity (the number of equivalent passes over the data) 
of Mini-Batch SPDC is
\[
    \order\left(\big(1+\sqrt{\kappa(m/n)}\big)\log(1/\epsilon) \right).
\]
The above expression implies that a smaller batch size~$m$ leads to less 
number of passes through the data.
In this sense, the basic SPDC method with $m=1$ is the most efficient one.
However, if we prefer the least amount of wall-clock time, then the best choice
is to choose a mini-batch size~$m$ that matches 
the number of parallel processors available.

\subsection{Convergence rate of primal-dual gap}
In the previous subsection, we established iteration complexity of the 
Mini-Batch SPDC method in terms of approximating the saddle point
of the minimax problem~\eqref{eqn:min-max-saddle}, 
more specifically, to meet the requirement in~\eqref{eqn:saddle-epsilon}.
Next we show that it has the same order of complexity in 
reducing the primal-dual objective gap $P(x^\supt) - D(y^\supt)$, where
$P(x)$ is defined in~\eqref{eqn:min-primal} and
\begin{equation}\label{eqn:dual-function}
  D(y) \eqdef \min_{x\in\R^d} f(x, y)
    = \frac{1}{n}\sum_{i=1}^n -\phi_i^*(y_i)
    - g^*\biggl(-\frac{1}{n} \sum_{i=1}^n y_i a_i\biggr) .
\end{equation}
where $g^*(u)=\sup_{x\in\R^d}\{x^T u-g(x)\}$ is the conjugate function of~$g$.

Under Assumption~\ref{asmp:smooth-convex}, the function $\spf(x,y)$
defined in~\eqref{eqn:min-max-saddle}
has a unique saddle point $(x^\star, y^\star)$, and 
\[
  P(x^\star) = \spf(x^\star,y^\star)=D(y^\star).
\]
However, in general, for any point $(x,y)\in\dom(g)\times\dom(\phi^*)$, we have
\[
  P(x) = \max_y \spf(x, y) \geq \spf(x,y^\star), \qquad
  D(y) = \min_x \spf(x, y) \leq \spf(x^\star,y) .
\]
Thus the result in Theorem~\ref{thm:spdc-convergence} does not translate
directly into a convergence bound on the primal-dual gap.
We need to bound $P(x)$ and $D(y)$ by $f(x,y^\star)$ and $f(x^\star,y)$,
respectively, in the opposite directions.
For this purpose, we need the following lemma, 
which we extracted from \cite{YuLinYang15}.
We provide the proof in Appendix~\ref{sec:proof-gap-by-saddle} for completeness.

\begin{lemma}[\cite{YuLinYang15}]
  \label{lem:gap-by-saddle}
  Suppose Assumption~\ref{asmp:smooth-convex} holds.
  Let $(x^\star, y^\star)$ is the unique saddle-point of $\spf(x,y)$,
  and $R=\max_{1\leq i\leq n} \ltwos{a_i}$.
  Then for any point $(x,y)\in\dom(g)\times\dom(\phi^*)$, we have
\begin{align*}
  P(x) \leq \spf(x, y^\star) + \frac{R^2}{2\gamma}\ltwos{x-x^\star}^2, \qquad
  D(y) \geq \spf(x^\star, y) - \frac{R^2}{2\lambda n}\ltwos{y^t-y^\star}^2 .
\end{align*}
\end{lemma}

\begin{corollary}\label{coro:gap-complexity}
Suppose Assumption~\ref{asmp:smooth-convex} holds
and the parameters $\tau$, $\sigma$ and $\theta$ are set as 
in~\eqref{eqn:tau-sigma-theta}. Let $\widetilde \Delta^{(0)}
:= \Delta^{(0)} +\frac{\ltwos{y^{(0)}-y^\star}^2}{4m\sigma}$.
Then for any $\epsilon\geq 0$, the iterates of
Algorithm~\ref{alg:spdc-minibatch} satisfy
\[
  \E [ P(x^{(T)}) - D(y^{(T)}) ] \leq \epsilon
\]
whenever
\[
  T \geq \bigg(\frac{n}{m}+R\sqrt{\frac{n}{m\lambda\gamma}}\bigg)
  \log\left(\left(1+\frac{R^2}{\lambda\gamma}\right)\frac{\widetilde \Delta^{(0)}}{\epsilon}\right) .
\]
\end{corollary}
\begin{proof}
The function $\spf(x,y^\star)$ is strongly convex in~$x$
with parameter $\lambda$, and~$x^\star$ is the minimizer.
Similarly, $-\spf(x^\star, y)$ is strongly convex in~$y$
with parameter $\gamma/n$, and is minimized by~$y^\star$.
Therefore,
\begin{equation}\label{eqn:distance-by-saddle}
  \frac{\lambda}{2}\ltwos{x^\supt-x^\star}^2 
  \leq \spf(x^\supt,y^\star)-\spf(x^\star,y^\star) ,  \qquad
  \frac{\gamma}{2n}\ltwos{y^\supt-y^\star}^2
  \leq \spf(x^\star,y^\star)-\spf(x^\star, y^\supt). 
\end{equation}
We bound the following weighted primal-dual gap
\begin{align*}
  & P(x^\supt)-P(x^\star) + \frac{n}{m}\left(D(y^\star)-D(y^\supt)\right) \\
  \leq ~& \spf(x^\supt,y^\star) -\spf(x^\star, y^\star)  + \frac{n}{m}\left( \spf(x^\star,y^\star) - \spf(x^\star, y^\supt) \right) + \frac{R^2}{2\gamma}\ltwos{x^\supt-x^\star}^2 + \frac{n}{m} \frac{R^2}{2n\lambda}\ltwos{y^\supt-y^\star}^2 \\
\leq ~& \Delta^\supt + \frac{R^2}{\lambda\gamma}\left(\frac{\lambda}{2}\ltwos{x^\supt-x^\star}^2 + \frac{n}{m} \frac{\gamma}{2n}\ltwos{y^\supt-y^\star}^2 \right) \\
\leq ~& \Delta^\supt + \frac{R^2}{\lambda\gamma} \left( \spf(x^\supt,y^\star) -\spf(x^\star, y^\star)  + \frac{n}{m}\left( \spf(x^\star,y^\star) - \spf(x^\star, y^\supt) \right) \right) \\
\leq ~& \left(1+\frac{R^2}{\lambda\gamma}\right) \Delta^\supt .
\end{align*}
The first inequality above is due to Lemma~\ref{lem:gap-by-saddle},
the second and fourth inequalities are due to the definition of~$\Delta^\supt$,
and the third inequality is due to~\eqref{eqn:distance-by-saddle}.
Taking expectations on both sides of the above inequality, then applying
Theorem~\ref{thm:spdc-convergence}, we obtain
\[
\E\left[P(x^\supt)-P(x^\star)+\frac{n}{m}\left(D(y^\star)-D(y^\supt)\right)\right]
~\leq ~ \theta^t \left(1+\frac{R^2}{\lambda\gamma}\right)
 \widetilde\Delta^{(0)} = (1+\kappa) \widetilde\Delta^\supt .
\]
Since $n\geq m$ and $D(y^\star)-D(y^\supt))\geq 0$, 
this implies the desired result. 
\end{proof}


\section{Extensions to non-smooth or non-strongly convex functions}
\label{sec:non-smooth}

The complexity bounds established in Section~\ref{sec:spdc-method}
require each $\phi_i$ be $(1/\gamma)$-smooth,
and the function~$g$ be $\lambda$-strongly convex. 
For general loss functions where either or both of these conditions fail
(e.g., the hinge loss and $\ell_1$-regularization),
we can slightly perturb the saddle-point function $f(x,y)$ so that 
the SPDC method can still be applied.

To be concise, we only consider the case where 
neither $\phi_i$ is smooth nor~$g$ is strongly convex.
Formally, we assume that each $\phi_i$ and~$g$ are convex and Lipschitz continuous,
and~$f(x,y)$ has a saddle point $(\xhat,\yhat)$.
We choose a scalar $\delta > 0$ and consider the modified saddle-point function:
\begin{align}
	f_\delta (x,y) \eqdef \frac{1}{n} \sum_{i=1}^n \left( \yvec_i \langle a_i, x\rangle - \Big( \phistar_i(\yvec_i) + \frac{\delta \yvec_i^2}{2} \Big) \right)
	+  g(x) + \frac{\delta}{2} \ltwos{x}^2 .
\end{align}
Denote by $(\xhat_\delta,\yhat_\delta)$ the saddle-point of $f_\delta$.
We employ the Mini-Batch SPDC method (Algorithm~\ref{alg:spdc-minibatch})
to approximate $(\xhat_\delta,\yhat_\delta)$,
treating $\phistar_i + \frac{\delta}{2}(\cdot)^2$ as $\phistar_i$ and
$g+ \frac{\delta}{2}\ltwos{\cdot}^2$ as $g$,
which are all $\delta$-strongly convex.
We note that adding strongly convex perturbation on $\phi_i^*$ is equivalent
to smoothing~$\phi_i$, which becomes $(1/\delta)$-smooth
(see, e.g., \cite{Nesterov05smooth}).
Letting $\gamma=\lambda=\delta$, the parameters $\tau$, $\sigma$ and $\theta$ 
in~\eqref{eqn:tau-sigma-theta} become 
\[
\tau = \frac{1}{R} \sqrt{\frac{m}{n}},\qquad
\sigma = \frac{1}{R} \sqrt{\frac{n}{m}},\qquad
~\mbox{and}\quad
\theta = 1 - \bigg( \frac{n}{m} + \frac{R}{\delta}\sqrt{\frac{n}{m}}\bigg)^{-1}.
\]
Although $(\xhat_\delta,\yhat_\delta)$ is not exactly the saddle point of $f$, 
the following corollary shows that applying the SPDC method to the perturbed
function $f_\delta$ effectively minimizes the original loss function $P$.
Similar results for the convergence of the primal-dual gap can also be
established.

\begin{corollary}\label{coro:general-function-iteration-complexity}
Assume that each $\phi_i$ is convex and $G_\phi$-Lipschitz continuous, 
and $g$ is convex and $G_g$-Lipschitz continuous. Define two constants:
\begin{align*}
	C_1 = (\ltwos{\xhat}^2 + G_\phi^2), \qquad 
    C_2 = (G_\phi R+G_g)^2\bigg(\frac{\Delta_{\delta}^{(0)}+\ltwo{y^{(0)}-y_{\delta}^\star}^2 R/(4\sqrt{mn})}{1/(2\tau)+\lambda/2}\bigg),
\end{align*}
where $\Delta_\delta^{(0)}$ is evaluated in terms of the perturbed function
$f_\delta$.
If we choose $\delta \leq \epsilon/C_1$, 
then we have $\E[P(x^{(T)}) - P(\xhat)] \leq \epsilon$ whenever
\[
    T \geq \bigg( \frac{n}{m} + \frac{R}{\delta} \sqrt{\frac{n}{m}} \bigg) 
    \log \bigg(\frac{4 C_2}{\epsilon^2}\bigg).
\]
\end{corollary}

\begin{proof}
Let $\ytilde =\arg\max_y f(\xhat_\delta, y)$ be a shorthand notation. 
We have
\begin{align*}
	P(\xhat_\delta) & ~\stackrel{(i)} =~ f(\xhat_\delta, \ytilde)  ~\stackrel{(ii)}\leq~ f_\delta(\xhat_\delta, \ytilde) + \frac{\delta\ltwos{\ytilde}^2}{2 n}
	~\stackrel{(iii)}\leq~ f_\delta(\xhat_\delta, \yhat_\delta) + \frac{\delta\ltwos{\ytilde}^2}{2 n}
	~\stackrel{(iv)} \leq~ f_\delta(\xhat, \yhat_\delta) + \frac{\delta\ltwos{\ytilde}^2}{2 n}\\
	& ~\stackrel{(v)}\leq~ f(\xhat, \yhat_\delta) + \frac{\delta\ltwos{\xhat}^2}{2} + \frac{\delta\ltwos{\ytilde}^2}{2 n}
	~\stackrel{(vi)}\leq~ f(\xhat, \yhat) + \frac{\delta\ltwos{\xhat}^2}{2} + \frac{\delta\ltwos{\ytilde}^2}{2 n}\\
	& ~\stackrel{(vii)} = P(\xhat) + \frac{\delta\ltwos{\xhat}^2}{2} + \frac{\delta\ltwos{\ytilde}^2}{2 n} .
\end{align*}
Here, equations~(i) and~(vii) use the definition of the function $f$,
inequalities~(ii) and (v) use the definition of the function $f_\delta$,
inequalities~(iii) and~(iv) use the fact that $(\xhat_\delta,\yhat_\delta)$
is the saddle point of $f_\delta$,
and inequality~(vi) is due to the fact that $(\xhat,\yhat)$ 
is the saddle point of~$f$.

Since $\phi_i$ is $G_\phi$-Lipschitz continuous, the domain of $\phistar_i$ 
is in the interval $[-G_\phi,G_\phi]$, which implies
$\ltwos{\ytilde}^2 \leq n G_\phi^2$
(see, e.g., \cite[Lemma 1]{SSZhang13acclSDCA}). 
Thus, we have
\begin{align}
	P(\xhat_\delta) - P(\xhat) \leq  \frac{\delta}{2}(\ltwos{\xhat}^2 + G_\phi^2) = \frac{\delta}{2}C_1 .
	\label{eqn:general-function-bound-1}
\end{align}
On the other hand, since $P$ is $(G_\phi R+G_g)$-Lipschitz continuous, 
Theorem~\ref{thm:spdc-convergence} implies
\begin{align}
	\E[ P(x^{(T)}) -P(\xhat_\delta)]  &~\leq~ (G_\phi R+G_g) \E[\ltwos{x^{(T)} - \xhat_\delta}] 
    ~\leq~  \sqrt{C_2}\left( 1 - \bigg( \frac{n}{m} + \frac{R}{\delta} \sqrt{\frac{n}{m}}\bigg)^{-1}\right)^{T/2}. 
    \label{eqn:general-function-bound-2}
\end{align}	
Combining~\eqref{eqn:general-function-bound-1} 
and~\eqref{eqn:general-function-bound-2},
in order to obtain
$\E[ P(x^{(T)}) -P(\xhat)]\leq \epsilon$, 
it suffices to have $C_1\delta \leq \epsilon$ and
\begin{align}\label{eqn:general-function-iteration-condition}
    \sqrt{C_2}\bigg( 1 - \bigg( \frac{n}{m} + \frac{R}{\delta} \sqrt{\frac{n}{m}}\bigg)^{-1}\bigg)^{T/2}  \leq \frac{\epsilon}{2}.
\end{align}
The corollary is established by finding the smallest $T$ that satisfies inequality~\eqref{eqn:general-function-iteration-condition}.
\end{proof}

\begin{table}[t]
\renewcommand{\arraystretch}{1.3}
    \centering
    \begin{tabular}{|c|c|c|}
        \hline
        $\phi_i$ & $g$ & iteration complexity $\widetilde\order(\cdot)$ \\
        \hline
        $(1/\gamma)$-smooth & $\lambda$-strongly convex & $n/m+\sqrt{(n/m)/(\lambda\gamma)}$ \\
        $(1/\gamma)$-smooth & non-strongly convex & $n/m+\sqrt{(n/m)/(\epsilon\gamma)}$ \\
        non-smooth & $\lambda$-strongly convex & $n/m+\sqrt{(n/m)/(\epsilon\lambda)}$ \\
        non-smooth & non-strongly convex & $n/m+\sqrt{n/m}/\epsilon$ \\
        \hline
    \end{tabular}
    \caption{Iteration complexities of the SPDC method under different assumptions on the functions~$\phi_i$ and $g$. For the last three cases, we solve the perturbed 
    saddle-point problem with $\delta=\epsilon/C_1$.}
    \label{tab:spdc-complexities}
\end{table}

There are two other cases that can be considered:
when $\phi_i$ is not smooth but~$g$ is strongly convex,
and when $\phi_i$ is smooth but~$g$ is not strongly convex.
They can be handled with the same technique described above, 
and we omit the details here.
In Table~\ref{tab:spdc-complexities},
we list the complexities of the Mini-Batch SPDC method for finding an
$\epsilon$-optimal solution of problem~\eqref{eqn:min-primal}
under various assumptions.
Similar results are also obtained in \cite{SSZhang13acclSDCA}.


\section{SPDC with non-uniform sampling}
\label{sec:non-uniform}

One potential drawback of the SPDC algorithm is that, 
its convergence rate depends on a problem-specific constant~$R$, 
which is the largest $\ell_2$-norm of the feature vectors~$a_i$. 
As a consequence, the algorithm may perform badly on unnormalized data, 
especially if the $\ell_2$-norms of some feature vectors are substantially 
larger than others. 
In this section, we propose an extension of the SPDC method to mitigate 
this problem, which is given in Algorithm~\ref{alg:spdc-nonuniform}.

\begin{algorithm}[t]
\DontPrintSemicolon
\KwIn{parameters $\tau,\sigma,\theta\in \R_+$,
    number of iterations $T$, and initial points $x^{(0)}$ and $y^{(0)}$.}
\vspace{2pt}
\textbf{Initialize:} 
$\xbar^{(0)}=x^{(0)}$, $\uvec^{(0)}=(1/n)\sum_{i=1}^n y^{(0)}_i a_i$.\; 
\vspace{2pt}
\For{$t=0,1,2,\dots,T-1$}
{
Randomly pick $k\in \{1,2,\dots,n\}$, with probability $p_k$ given 
in~\eqref{eqn:nonuniform-prob}.\;
Execute the following updates:
	\begin{align}
		\yvec_i^\suptp &= \left\{ \begin{array}{ll}
        \arg\max_{\beta\in \R} \left\{ \beta \langle a_i, \xbar^\supt \rangle - \phistar_{i}(\beta) -  \frac{p_i n }{2\sigma}(\beta-\yvec_i^\supt)^2 \right\}
		& i = k,\\
		\yvec_i^\supt & i\neq k,
		\end{array} \right. \label{eqn:spdc-robust-maximize} \\
        \uvec^\suptp &= \uvec^\supt + \frac{1}{n}(\yvec_k^\suptp - \yvec_k^\supt) a_k, \nonumber\\
        x^\suptp &= \arg\min_{x\in \R^d} \left\{ g(x) + \Bigl\langle \uvec^\supt + \frac{1}{p_k} (u^\suptp-u^\supt), ~x \Bigr\rangle +  \frac{\ltwos{x-x^\supt}^2}{2\tau} \right\},
		\label{eqn:spdc-robust-minimize}\\
		\xbar ^\suptp &= x^\suptp + \theta (x^\suptp - x^\supt). \nonumber
	\end{align}
\vspace{-20pt}
}
\KwOut{$x^{(T)}$ and $\yvec^{(T)}$}
\caption{SPDC method with weighted sampling}
\label{alg:spdc-nonuniform}
\end{algorithm}

The basic idea is to use non-uniform 
sampling in picking the dual coordinate to update at each iteration.
In Algorithm~\ref{alg:spdc-nonuniform}, we pick coordinate~$k$ with the probability 
\begin{equation}\label{eqn:nonuniform-prob}
    p_k = (1-\alpha) \frac{1}{n} + \alpha \,\frac{\ltwos{a_k}}{\sum_{i=1}^n \ltwos{a_i}}, 
    \qquad k=1,\ldots, n,
\end{equation}
where $\alpha\in(0,1)$ is a parameter. 
In other words, this distribution is a (strict) convex combination of the
uniform distribution and the distribution that is proportional to the 
feature norms.
Therefore, instances with large feature norms are sampled more frequently,
controlled by~$\alpha$.
Simultaneously, we adopt an adaptive regularization in 
step~\eqref{eqn:spdc-robust-maximize},
imposing stronger regularization on such instances. 
In addition, we adjust the weight of~$a_k$ in~\eqref{eqn:spdc-robust-minimize}
for updating the primal variable.
As a consequence, the convergence rate of Algorithm~\ref{alg:spdc-nonuniform}
depends on the average norm of feature vectors, 
as well as the parameter~$\alpha$.
This is summarized in the following theorem.

\begin{theorem}\label{thm:nonuniform-spdc-convergence}
Suppose Assumption~\ref{asmp:smooth-convex} holds. 
Let $\barR =\frac{1}{n} \sum_{i=1}^n \ltwos{a_i}$.
If the parameters $\tau,\sigma,\theta$ in Algorithm~\ref{alg:spdc-nonuniform}
are chosen such that
\begin{align}
    \tau = \frac{\alpha}{2\barR} \sqrt{\frac{\gamma}{n\lambda}}, \qquad 
    \sigma = \frac{\alpha}{2\barR} \sqrt{\frac{n\lambda}{\gamma}}, \qquad 
    \theta = 1 - \left(\frac{n}{1-\alpha} + \frac{\barR}{\alpha}\sqrt{\frac{n}{\lambda\gamma}}\right)^{-1}, 
    \label{eqn:tau-sigma-theta-alpha}
\end{align}
then for each $t\geq 1$, we have
\begin{align*}
& \Big( \frac{1}{2\tau}+\lambda\Big) {\E\bigl[\ltwos{x^{(t)} - \xhat}^2\bigr]} 
+ \Big(\frac{1}{4\sigma}+\frac{\gamma}{n}\Big) {\E\bigl[\ltwos{\yvec^{(t)} - \yhat}^2\bigr]} \\
\leq~ & \theta^{\,t} \, \bigg( \Big( \frac{1}{2\tau} + \lambda\Big)\ltwos{x^{(0)}-\xhat}^2 + \Big( \frac{1}{2\sigma} + \frac{\gamma}{1-\alpha}\Big)\ltwos{y^{(0)}-\yhat}^2 \bigg).
\end{align*}
\end{theorem}

Choosing $\alpha = 1/2$ and comparing with Theorem~\ref{thm:spdc-convergence}, 
the parameters $\tau$, $\sigma$, and~$\theta$ in 
Theorem~\ref{thm:nonuniform-spdc-convergence}
are determined by the average norm of the features, 
$\barR=\frac{1}{n} \sum_{i=1}^n \ltwos{a_i}$, 
instead of the largest one $R=\max\{\|a_1\|_2,\ldots,\|a_n\|_2\}$.
This difference makes Algorithm~\ref{alg:spdc-nonuniform} 
more robust to unnormalized feature vectors. 
For example, if the $a_i$'s are sampled i.i.d.\ from a multivariate normal 
distribution, then $\max_{i}\{\ltwos{a_i}\}$ almost surely goes to infinity 
as $n\to \infty$, but the average norm $\frac{1}{n} \sum_{i=1}^n \ltwos{a_i}$
converges to $\E[\ltwos{a_i}]$. 

Since~$\theta$ is a bound on the convergence factor, we would like to make 
it as small as possible. 
For its expression in~\eqref{eqn:tau-sigma-theta-alpha}, it can be minimized
by choosing
\[
  \alpha^\star = \frac{1}{1+(n/\bar\kappa)^{1/4}} ,
\]
where $\bar\kappa = \barR^2/(\lambda\gamma)$ is an average condition number.
We have $\alpha^\star=1/2$ if $\bar\kappa=n$.
The value of $\alpha^\star$ 
decreases slowly to zero as the ratio $n/\bar\kappa$ grows, 
and increases to one as the ratio $n/\bar\kappa$ drops.
Thus, we may choose a relatively uniform distribution for well conditioned
problems, 
but a more aggressively weighted distribution for ill-conditioned problems.

For simplicity of presentation, we described 
in Algorithm~\ref{alg:spdc-nonuniform} a weighted sampling SPDC 
method with single dual coordinate update, i.e., the case of $m=1$.
It is not hard to see that the non-uniform sampling scheme 
can also be extended to Mini-Batch SPDC with $m>1$.
Here, we omit the technical details.

\section{Related Work}
\label{sec:related-work}

Chambolle and Pock \cite{ChambollePock11} considered a class of convex
optimization problems with the following saddle-point structure:
\begin{equation}\label{eqn:CP11-saddle-point}
    \min_{x\in\R^d} \; \max_{y\in\R^n} ~\bigl\{ \langle Kx, y\rangle + G(x) - F^*(y) \bigr\} ,
\end{equation}
where $K\in\R^{m\times d}$, $G$ and $F^*$ are proper closed convex functions,
with $F^*$ itself being the conjugate of a convex function~$F$.
They developed the following first-order primal-dual algorithm:
\begin{align}
    y^\suptp &= \arg\max_{y\in\R^n} \left\{ \langle K\xbar^\supt, y\rangle - F^*(y)-\frac{1}{2\sigma}\|y-y^\supt\|_2^2 \right\}, \label{eqn:CP11-dual-update}\\
    x^\suptp &= \arg\min_{x\in\R^d} \left\{ \langle K^T y^\suptp, x\rangle + G(x) + \frac{1}{2\tau}\|x-x^\supt\|_2^2 \right\}, \label{eqn:CP11-primal-update}\\
    \xbar^\suptp &= x^\suptp + \theta (x^\suptp-x^\supt) \label{eqn:CP11-extrapolation}.
\end{align}
When both $F^*$ and~$G$ are strongly convex and the parameters~$\tau$, 
$\sigma$ and~$\theta$ are chosen appropriately, this algorithm obtains
accelerated linear convergence rate \cite[Theorem~3]{ChambollePock11}.

We can map the saddle-point problem~\eqref{eqn:min-max-saddle} 
into the form of~\eqref{eqn:CP11-saddle-point} by letting
$A=[a_1,\ldots,a_n]^T$ and 
\begin{equation}\label{eqn:CP11-correspondence}
    K = \frac{1}{n} A, \qquad G(x)=g(x), \qquad 
    F^*(y) =\frac{1}{n}\sum_{i=1}^n\phi_i^*(y_i) .
\end{equation}
The SPDC method developed in this paper can be viewed as an extension 
of the batch method \eqref{eqn:CP11-dual-update}-\eqref{eqn:CP11-extrapolation},
where the dual update step~\eqref{eqn:CP11-dual-update}
is replaced by a single coordinate update~\eqref{eqn:spdc-maximize}
or a mini-batch update~\eqref{eqn:spdc-minibatch-maximize}.
However, in order to obtain accelerated convergence rate, 
more subtle changes are necessary in the primal update step. 
More specifically, we introduced the auxiliary variable
$u^\supt = \frac{1}{n} \sum_{i=1}^n y_i^\supt a_i = K^T y^\supt$,
and replaced the primal update step~\eqref{eqn:CP11-primal-update}  
by~\eqref{eqn:spdc-minimize} and~\eqref{eqn:spdc-minibatch-minimize}.
The primal extrapolation step~\eqref{eqn:CP11-extrapolation} stays the same.

To compare the batch complexity of SPDC with that of 
\eqref{eqn:CP11-dual-update}-\eqref{eqn:CP11-extrapolation},
we use the following facts implied by Assumption~\ref{asmp:smooth-convex}
and the relations in~\eqref{eqn:CP11-correspondence}:
\[
    \|K\|_2 = \frac{1}{n} \|A\|_2, \quad 
    G(x) ~\mbox{is}~ \lambda\mbox{-strongly convex}, \quad
    \mbox{and}~F^*(y) ~\mbox{is}~ (\gamma/n)\mbox{-strongly convex}.
\]
Based on these conditions, we list in Table~\ref{tab:CP11-SPDC}
the equivalent parameters used in \cite[Algorithm~3]{ChambollePock11}
and the batch complexity obtained in \cite[Theorem~3]{ChambollePock11},
and compare them with SPDC.

The batch complexity of the Chambolle-Pock algorithm is
$\widetilde\order(1+\|A\|_2/(2\sqrt{n\lambda\gamma}))$,
where the $\widetilde\order(\cdot)$ notation hides 
the $\log(1/\epsilon)$ factor.
We can bound the spectral norm $\|A\|_2$ by the Frobenius norm $\|A\|_F$
and obtain
\[
    \|A\|_2 \leq \|A\|_F \leq \sqrt{n}\max_i \{ \|a_i\|_2\} = \sqrt{n} R.
\]
(Note that the second inequality above would be an equality 
if the columns of~$A$ are normalized.)
So in the worst case, the batch complexity of the Chambolle-Pock algorithm 
becomes
\[
    \widetilde\order\left( 1+R/\sqrt{\lambda\gamma} \right)
    =\widetilde\order\left( 1+\sqrt{\kappa} \right), 
    \qquad \mbox{where}~\kappa=R^2/(\lambda\gamma),
\]
which matches the worst-case complexity of the AFG methods
\cite{Nesterov04book,Nesterov13composite} 
(see Section~\ref{sec:cond-number} and also the discussions 
in \cite[Section~5]{LinLuXiao14apcg}).
This is also of the same order as the complexity of SPDC with $m=n$
(see Section~\ref{sec:saddle-approx-analysis}).
When the condition number $\kappa\gg 1$,
they can be $\sqrt{n}$ worse than the batch complexity of SPDC with $m=1$,
which is $\widetilde\order(1+\sqrt{\kappa/n})$.

\begin{table}
\renewcommand{\arraystretch}{1.8}
    \centering
    \begin{tabular}{|c|c|c|c|c|}
        \hline
        algorithm & $\tau$ & $\sigma$ & $\theta$ & batch complexity \\ 
        \hline
        Chambolle-Pock \cite{ChambollePock11} & $\frac{\sqrt{n}}{\|A\|_2}\sqrt{\frac{\gamma}{\lambda}}$ & $\frac{\sqrt{n}}{\|A\|_2}\sqrt{\frac{\lambda}{\gamma}}$ & $1 - \frac{1}{1+\|A\|_2/(2\sqrt{n\lambda\gamma})}$ & $\left(1+\frac{\|A\|_2}{2\sqrt{n\lambda\gamma}}\right) \log(1/\epsilon)$ \\ 
        SPDC with $m=n$ & $\frac{1}{R}\sqrt{\frac{\gamma}{\lambda}}$ & $\frac{1}{R}\sqrt{\frac{\lambda}{\gamma}}$ & $1-\frac{1}{1+R/\sqrt{\lambda\gamma}}$ & $\left(1+\frac{R}{\sqrt{\lambda\gamma}}\right)\log(1/\epsilon)$ \\
        SPDC with $m=1$ & $\frac{1}{R}\sqrt{\frac{\gamma}{n\lambda}}$ & $\frac{1}{R}\sqrt{\frac{n\lambda}{\gamma}}$ & $1-\frac{1}{n+R\sqrt{n/\lambda\gamma}}$ & $\left(1+\frac{R}{\sqrt{n\lambda\gamma}}\right) \log(1/\epsilon)$ \\
        \hline
    \end{tabular}
    \caption{Comparing SPDC with Chambolle and Pock \cite[Algorithm~3, Theorem~3]{ChambollePock11}.}
    \label{tab:CP11-SPDC}
\end{table}

If either $G(x)$ or $F^*(y)$ in~\eqref{eqn:CP11-saddle-point} is not strongly 
convex, Chambolle and Pock proposed variants of the primal-dual batch
algorithm to achieve accelerated sublinear convergence rates
\cite[Section~5.1]{ChambollePock11}.
It is also possible to extend them to coordinate update methods
for solving problem~\eqref{eqn:min-primal}
when either $\phi_i^*$ or~$g$ is not strongly convex. 
Their complexities would be similar to those 
in Table~\ref{tab:spdc-complexities}.

Our algorithms and theory can be readily generalized to solve the problem of 
\[
  \minimize_{x\in\R^d} \quad \frac{1}{n}\sum_{i=1}^n \phi_i(A_i^T x) + g(x),
\]
where each $A_i$ is an $d_i\times d$ matrix, and $\phi_i:\R^{d_i}\to\R$ is
a smooth convex function.
This more general formulation is used, e.g., in \cite{SSZhang13acclSDCA}.
Most recently, Lan \cite{Lan2015RPDG} considered a special case with 
$d_i=d$ and $A_i=I_d$, and recognized that the dual coordinate proximal mapping
used in~\eqref{eqn:spdc-maximize} and~\eqref{eqn:spdc-robust-maximize} 
is equivalent to computing the primal gradients $\nabla\phi_i$ at a 
particular sequence of points $\underline{x}^\supt$.
Based on this observation, he derived a similar randomized incremental 
gradient algorithm which  
share the same order of iteration complexity as we presented in this paper.


\subsection{Dual coordinate ascent methods}
We can also solve the primal problem~\eqref{eqn:min-primal} via its dual:
\begin{equation}\label{eqn:max-dual}
    \maximize_{y\in\R^n} 
    ~\biggl\{ D(y) \eqdef \frac{1}{n}\sum_{i=1}^n -\phi_i^*(y_i)
    - g^*\Bigl(-\frac{1}{n} \sum_{i=1}^n y_i a_i\Bigr) \biggr\},
\end{equation}
Because of the problem structure, coordinate ascent methods
(e.g., \cite{Platt99,ChangHsiehLin08,HsChLiKeSu08,SSZhang13SDCA})
can be more efficient than full gradient methods. 
In the stochastic dual coordinate ascent (SDCA) method \cite{SSZhang13SDCA},
a dual coordinate $y_i$ is picked at random during each iteration 
and updated to increase the dual objective value.
Shalev-Shwartz and Zhang~\cite{SSZhang13SDCA} showed that 
the iteration complexity of SDCA is $O\left((n+\kappa)\log(1/\epsilon)\right)$,
which corresponds to the batch complexity
$\widetilde\order(1+\kappa/n)$.

For more general convex optimization problems, 
there is a vast literature on coordinate descent methods;
see, e.g., the recent overview by Wright \cite{Wright15CD}.
In particular, 
Nesterov's work on randomized coordinate descent \cite{Nesterov12rcdm}
sparked a lot of recent activities on this topic. 
Richt\'{a}rik and Tak\'{a}\v{c} \cite{RichtarikTakac12} extended the algorithm
and analysis to composite convex optimization.
When applied to the dual problem~\eqref{eqn:max-dual}, it becomes one variant
of SDCA studied in \cite{SSZhang13SDCA}.
Mini-batch and distributed versions of SDCA have been proposed and analyzed in
\cite{TakacBijralRichtarikSrebro13} and \cite{Yang13DistrSDCA} respectively.
Non-uniform sampling schemes 
have been studied for both stochastic gradient and SDCA methods
(e.g., \cite{Needell14stochastic,XiaoZhang14ProxSVRG,ZhaoZhang14Sampling,QuRichtarikZhang14}).

Shalev-Shwartz and Zhang \cite{SSZhang13minibatchSDCA} proposed an 
accelerated mini-batch SDCA method which incorporates additional primal updates
than SDCA, and bears some similarity to our Mini-Batch SPDC method.
They showed that its complexity interpolates between that of SDCA and AFG
by varying the mini-batch size~$m$.
In particular, for $m=n$, it matches that of the AFG methods (as SPDC does).
But for $m=1$, the complexity of their method is the same as SDCA,
which is worse than SPDC for ill-conditioned problems.

In addition, Shalev-Shwartz and Zhang \cite{SSZhang13acclSDCA}
developed an accelerated proximal SDCA method which achieves the same 
batch complexity $\widetilde\order\bigl(1+\sqrt{\kappa/n}\bigr)$ as SPDC.
Their method is an inner-outer iteration procedure, where the outer loop
is a full-dimensional accelerated gradient method 
in the primal space $x\in\R^d$.
At each iteration of the outer loop, the SDCA method \cite{SSZhang13SDCA}
is called to solve the dual problem~\eqref{eqn:max-dual} with 
customized regularization parameter and precision.
In contrast, SPDC is a straightforward single-loop 
coordinate optimization methods.

More recently, Lin et al.\ \cite{LinLuXiao14apcg} developed an
accelerated proximal coordinate gradient (APCG) method for solving 
a more general class of composite convex optimization problems.
When applied to the dual problem~\eqref{eqn:max-dual}, 
APCG enjoys the same batch complexity 
$\widetilde\order\bigl(1+\sqrt{\kappa/n}\bigr)$ as of SPDC.
However, it needs an extra primal proximal-gradient 
step to have theoretical guarantees on the convergence of primal-dual gap 
\cite[Section~5.1]{LinLuXiao14apcg}. 
The computational cost of this additional step is equivalent to
one pass of the dataset, thus it does not affect the overall complexity. 

\subsection{Other related work}

Another way to approach problem~\eqref{eqn:min-primal} is to reformulate it
as a constrained optimization problem
\begin{align}
    & \minimize \quad  \frac{1}{n}\sum_{i=1}^n \phi_i(z_i) + g(x) 
    \label{eqn:constrained-obj} \\
    & \mbox{subject to} \quad a_i^T x = z_i, \quad i=1,\ldots,n, \nonumber
\end{align}
and solve it by ADMM type of operator-splitting methods
(e.g., \cite{LionsMercier79}). 
In fact, as shown in \cite{ChambollePock11}, the batch primal-dual algorithm
\eqref{eqn:CP11-dual-update}-\eqref{eqn:CP11-extrapolation}
is equivalent to a pre-conditioned ADMM
(or inexact Uzawa method; see, e.g., \cite{ZhangBurgerOsher11}).
Several authors \cite{WangBanerjee12,Ouyang13,Suzuki13,ZhongKwok14}
have considered a more general formulation 
than~\eqref{eqn:constrained-obj},
where each $\phi_i$ is a function of the whole vector $z\in\R^n$.
They proposed online or stochastic versions of ADMM which operate on
only one $\phi_i$ in each iteration, and obtained sublinear convergence rates.
However, their cost per iteration is $\order(nd)$ instead of $\order(d)$.

Suzuki \cite{Suzuki14sdca_admm} considered a problem similar 
to~\eqref{eqn:min-primal}, but with more complex regularization function~$g$,
meaning that~$g$ does not have a simple proximal mapping.
Thus primal updates such as step~\eqref{eqn:spdc-minimize} 
or~\eqref{eqn:spdc-minibatch-minimize} in SPDC 
and similar steps in SDCA cannot be computed efficiently.
He proposed an algorithm that combines SDCA \cite{SSZhang13SDCA} and 
ADMM (e.g., \cite{Boyd10ADMM}), and showed that it has linear rate of 
convergence under similar conditions as Assumption~\ref{asmp:smooth-convex}. 
It would be interesting to see if the SPDC method can be extended to 
their setting to obtain accelerated linear convergence rate.

\section{Efficient Implementation with Sparse Data}
\label{sec:sparse-impl}

During each iteration of the SPDC methods, the updates of primal variables
(i.e., computing $x^\suptp$) require full $d$-dimensional vector operations;
see the step~\eqref{eqn:spdc-minimize} of Algorithm~\ref{alg:fast-spdc},
the step~\eqref{eqn:spdc-minibatch-minimize} of Algorithm~\ref{alg:spdc-minibatch}
and the step~\eqref{eqn:spdc-robust-minimize} of Algorithm~\ref{alg:spdc-nonuniform}.
So the computational cost per iteration is $\order(d)$, and this can be
too expensive if the dimension~$d$ is very high. 
In this section, we show how to exploit problem structure to avoid high-dimensional vector operations when the feature vectors~$a_i$ are sparse.
We illustrate the efficient implementation for two popular cases: 
when $g$ is an squared-$\ell_2$ penalty and 
when $g$ is an $\ell_1+\ell_2$ penalty.
For both cases, we show that the computation cost per iteration only depends on
the number of non-zero components of the feature vector.

\subsection{Squared $\ell_2$-norm penalty}\label{sec:l2-norm-penalty}

Suppose that $g(x) =\frac{\lambda}{2}\ltwos{x}^2$. 
For this case, the updates for each coordinate of $x$ are independent
of each other.
More specifically, $x^\suptp$ can be computed coordinate-wise in closed form:
\begin{align}
	x^\suptp_j = \frac{1}{1 + \lambda \tau} (x_j^\supt - \tau \uvec_j^\supt - \tau \Delta \uvec_j), \quad j=1,\ldots,n,
\label{eqn:l2-norm-x-update}
\end{align}
where $\Delta \uvec$ denotes $(\yvec_k^\suptp - \yvec_k^\supt)a_k$ in 
Algorithm~\ref{alg:fast-spdc},
or $\frac{1}{m} \sum_{k\in K}(\yvec_k^\suptp - \yvec_k^\supt)a_k$ in
Algorithm~\ref{alg:spdc-minibatch},
or $(\yvec_k^\suptp - \yvec_k^\supt)a_k /(p_k n)$ in 
Algorithm~\ref{alg:spdc-nonuniform}, and $\Delta \uvec_{j}$ represents
the $j$-th coordinate of $\Delta \uvec$.

Although the dimension~$d$ can be very large, we assume that each feature 
vector~$a_k$ is sparse. 
We denote by $J^\supt$ the set of non-zero coordinates at iteration $t$, 
that is, 
if for some index $k\in K$ picked at iteration $t$ we have $a_{kj}\neq 0$, 
then $j\in J^\supt$. 
If $j\notin J^\supt$, then the SPDC algorithm (and its variants) updates 
$\yvec^\suptp$ without using the value of $x_j^\supt$ or $\xbar_j^\supt$. 
This can be seen from the updates in~\eqref{eqn:spdc-maximize},
\eqref{eqn:spdc-minibatch-maximize} and~\eqref{eqn:spdc-robust-maximize},
where the value of the inner product $\langle a_k, \xbar^\supt\rangle$
does not depend on the value of $\xbar^\supt_j$.
As a consequence, we can delay the updates on $x_j$ and $\xbar_j$ 
whenever $j\notin J^\supt$ without affecting the updates on $y^\supt$,
and process all the missing updates at the next time when $j\in J^\supt$.

Such a delayed update can be carried out very efficiently. 
We assume that $t_0$ is the last time when $j\in J^\supt$, and
$t_1$ is the current iteration where we want to update $x_j$ and $\xbar_j$. 
Since $j\notin J^\supt$ implies $\Delta \uvec_j = 0$, we have
\begin{align}\label{eqn:l2-recursive-formula}
	x^{t+1}_j = \frac{1}{1 + \lambda \tau} (x_j^{(t)} - \tau \uvec_j^{(t)}),
    \qquad t=t_0+1, t_0+2, \dots, t_1-1.
\end{align}
Notice that $\uvec_j^\supt$ is updated only at iterations where $j\in J^\supt$.
The value of $\uvec_j^\supt$ doesn't change during iterations $[t_0+1,t_1]$, so we have 
$\uvec_j^{(t)} \equiv \uvec_j^{(t_0+1)}$ for $t \in [t_0+1, t_1]$. Substituting this equation
into the recursive formula~\eqref{eqn:l2-recursive-formula}, we obtain
\begin{align}
	x^{(t_1)}_j = \frac{1}{(1+\lambda\tau)^{t_1-t_0-1}} \left( x^{(t_0+1)}_j + \frac{\uvec_j^{(t_0+1)}}{\lambda}\right) - \frac{\uvec_j^{(t_0+1)}}{\lambda}.
	\label{eqn:l2-norm-x-efficient-update}
\end{align}
The update~\eqref{eqn:l2-norm-x-efficient-update} takes $\order(1)$ time to compute. Using the same formula,
we can compute $x^{(t_1-1)}_j$ and subsequently compute $\xbar^{(t_1)}_j = x^{(t_1)}_j + \theta(x^{(t_1)}_j - x^{(t_1-1)}_j)$.
Thus, the computational complexity of a single iteration in SPDC 
is proportional to $| J^\supt |$, independent of the dimension~$d$.

\subsection{$(\ell_1 + \ell_2)$-norm penalty}
\label{sec:l1+l2-penalty}

Suppose that $g(x) = \lambda_1 \lones{x} + \frac{\lambda_2}{2}\ltwos{x}^2$. 
Since both the $\ell_1$-norm and the squared $\ell_2$-norm are decomposable, 
the updates for each coordinate of $x^\suptp$ are independent.
More specifically, 
\begin{align}\label{eqn:l1l2-norm-x-update-original}
x_j^\suptp &= \arg\min_{\alpha\in\R} \left\{\lambda_1 |\alpha| + \frac{\lambda_2 \alpha^2}{2} 
+ (\uvec_j^\supt + \Delta \uvec_j)\alpha +  \frac{(\alpha-x_j^\supt)^2}{2\tau} \right\},
\end{align}
where $\Delta \uvec_j$ follows the definition in
Section~\ref{sec:l2-norm-penalty}.
If $j\notin J^\supt$, then $\Delta \uvec_j = 0$ and equation~\eqref{eqn:l1l2-norm-x-update-original} can be simplified as
\begin{align}\label{eqn:l1l2-norm-x-update-zero-original}
x_j^\suptp &= \left\{ 
	\begin{array}{ll}
		\frac{1}{1 + \lambda_2 \tau} (x_j^\supt - \tau \uvec^\supt_j - \tau\lambda_1) & \mbox{if } x_j^\supt - \tau \uvec^\supt_j > \tau\lambda_1,\\
		\frac{1}{1 + \lambda_2 \tau} (x_j^\supt - \tau \uvec^\supt_j + \tau\lambda_1) & \mbox{if } x_j^\supt - \tau \uvec^\supt_j < - \tau\lambda_1,\\
		0 & \mbox{otherwise.}
	\end{array}
\right.
\end{align}

Similar to the approach of Section~\ref{sec:l2-norm-penalty}, we delay the update of $x_j$ until
 $j\in J^\supt$. We assume $t_0$ to be the last iteration when $j\in J^\supt$, and let $t_1$ be the current
iteration when we want to update $x_j$. During iterations $[t_0+1,t_1]$, the value of $\uvec^\supt_j$ doesn't change, so we have
$\uvec_j^{(t)} \equiv \uvec_j^{(t_0+1)}$ for $t \in [t_0+1, t_1]$. 
Using equation~\eqref{eqn:l1l2-norm-x-update-zero-original} and the invariance of $\uvec_j^{(t)}$ for $t\in[t_0+1, t_1]$, we have an $\order(1)$
time algorithm to calculate $x_j^{(t_1)}$, which we detail in Appendix~\ref{sec:l1l2-update-detail}.
The vector $\xbar^{(t_1)}_j$ can be updated by the same algorithm since it is a linear combination of $x_j^{(t_1)}$ and $x_j^{(t_1-1)}$. 
As a consequence, the computational complexity of each iteration in SPDC 
is proportional to $|J^\supt|$, independent of the dimension~$d$.


\section{Experiments}
\label{sec:experiments}

In this section, we compare the basic SPDC method 
(Algorithm~\ref{alg:fast-spdc}) with several state-of-the-art optimization
algorithms for solving problem~\eqref{eqn:min-primal}.
They include two batch-update algorithms: 
the accelerated full gradient (FAG) method~\cite[Section~2.2]{Nesterov04book},
and the limited-memory quasi-Newton method L-BFGS 
\cite[Section~7.2]{NocedalWrightbook}).
For the AFG method, we adopt an adaptive line search scheme
(e.g., \cite{Nesterov13composite}) to improve its efficiency.
For the L-BFGS method, we use the memory size~30 as suggested by
\cite{NocedalWrightbook}.
We also compare SPDC with three stochastic algorithms:
the stochastic average gradient (SAG) method
\cite{LeRouxSchmidtBach12,SchmidtLeRouxBach13},
the stochastic dual coordinate descent (SDCA) method~\cite{SSZhang13SDCA} and the accelerated
stochastic dual coordinate descent (ASDCA) method~\cite{SSZhang13acclSDCA}. 
We conduct experiments on a synthetic dataset and three real datasets.

\subsection{Ridge regression with synthetic data}

We first compare SPDC with other algorithms on a simple quadratic problem 
using synthetic data. 
We generate $n=500$ i.i.d.~training examples $\{a_i,b_i\}_{i=1}^n$ 
according to the model
\[
    b = \langle a, x^*\rangle + \varepsilon, \quad 
    a \sim \mathcal{N}(0,\Sigma), \quad \varepsilon \sim \mathcal{N}(0,1),
\]
where $a\in \R^{d}$ and $d=500$, and $x^*$ is the all-ones vector. 
To make the problem ill-conditioned, the covariance matrix $\Sigma$ 
is set to be diagonal with $\Sigma_{jj} = j^{-2}$, for $j=1,\ldots,d$. 
Given the set of examples $\{a_i,b_i\}_{i=1}^n$, 
we then solved a standard ridge regression problem 
\begin{align*}
    \minimize_{x\in \R^{d}} ~\left\{ P(x) \eqdef \frac{1}{n}\sum_{i=1}^n \frac{1}{2}(a_i^T x - b_i)^2 + \frac{\lambda}{2}\ltwos{x}^2 \right\}.
\end{align*}
In the form of problem~\eqref{eqn:min-primal}, 
we have $\phi_i(z) = z^2/2$ and $g(x) = (1/2)\ltwos{x}^2$. 
As a consequence, the derivative of $\phi_i$ is
$1$-Lipschitz continuous and $g$ is $\lambda$-strongly convex.

\begin{figure}[t]
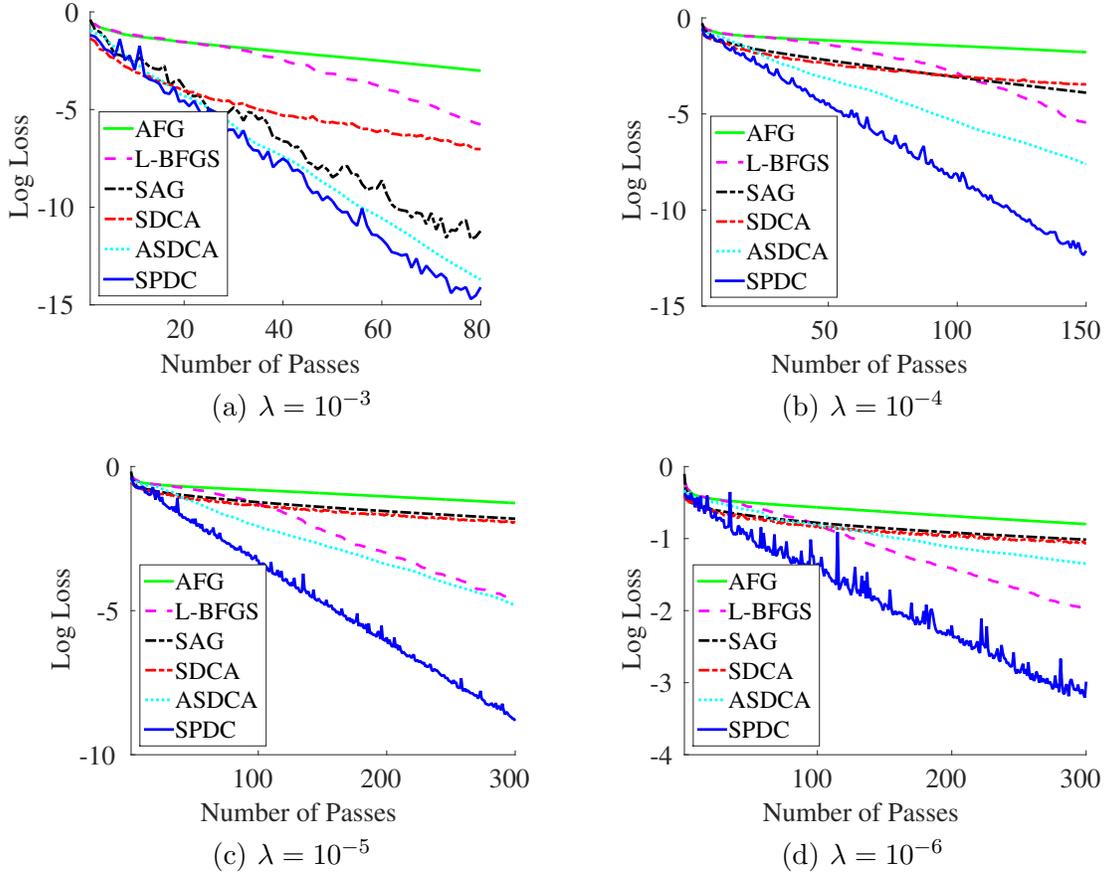

\centering
\begin{tabular}{cc}
\includegraphics[width = 0.4\textwidth]{syn-1e-3} \hspace{1cm} &
\includegraphics[width = 0.4\textwidth]{syn-1e-4} \\
(a) $\lambda = 10^{-3}$ & (b) $\lambda = 10^{-4}$\\[2ex]
\includegraphics[width = 0.4\textwidth]{syn-1e-5} &
\includegraphics[width = 0.4\textwidth]{syn-1e-6} \\
(c) $\lambda = 10^{-5}$ & (d) $\lambda = 10^{-6}$
\end{tabular}
\vspace{2ex}
\caption{Comparing SPDC with other methods on synthetic data, with
the regularization coefficient
$\lambda\in\{10^{-3}, 10^{-4}, 10^{-5}, 10^{-6}\}$. 
The horizontal axis is the number of passes through the entire dataset,
and the vertical axis is the logarithmic gap $\log(P(x^{(T)}) - P(\xhat))$.}
\label{fig:synthetic-compare}
\end{figure}

We evaluate the algorithms by the logarithmic optimality gap 
$\log(P(x^{(t)}) - P(\xhat))$, where $x^{(t)}$ is the output of
the algorithms after $t$ passes over the entire dataset,
and $\xhat$ is the global minimum.
When the regularization coefficient is relatively large, e.g.,
$\lambda=10^{-1}$ or $10^{-2}$,
the problem is well-conditioned and we observe fast convergence of 
the stochastic algorithms SAG, SDCA, ASDCA and SPDC, which are substantially  faster than the two batch methods AFG and L-BFGS. 

Figure~\ref{fig:synthetic-compare} shows the convergence 
of the five different algorithms 
when we varied~$\lambda$ from $10^{-3}$ to $10^{-6}$.
As the plot shows, when the condition number is greater than $n$,
the SPDC algorithm also converges substantially faster than the other 
two stochastic methods SAG and SDCA.
It is also notably faster than L-BFGS.
These results support our theory that SPDC enjoys a faster convergence rate on ill-conditioned problems.
In terms of their batch complexities, SPDC is up to $\sqrt{n}$ times faster than AFG, and $(\lambda n)^{-1/2}$ times faster than SAG and SDCA. 

Theoretically, ASDCA enjoys the same batch complexity as SPDC up to a 
multiplicative constant factor. 
Figure~\ref{fig:synthetic-compare} shows that the empirical performance of SPDC
is substantially faster that of ASDCA for small~$\lambda$. 
This may due to the fact that ASDCA follows an inner-outer iteration procedure,
while SPDC is a single-loop algorithm, 
explaining why it is empirically more efficient.

\subsection{Binary classification with real data}

\begin{table}[t]
\centering
\begin{tabular}{|c|r|r|r|}
  \hline
  Dataset name  & number of samples~$n$ & number of features~$d$ & sparsity \\\hline
  Covtype & 581,012 & 54  & 22\% \\
  RCV1 & 20,242 & 47,236 & 0.16\% \\
  News20 & 19,996 & 1,355,191  & 0.04\% \\
  \hline
\end{tabular}
\caption{Characteristics of three real datasets obtained from LIBSVM data~\cite{LIBSVMdata}.}
\label{tab:data-summary}
\end{table}

Finally we show the results of solving the binary classification problem on 
three real datasets.
The datasets are obtained from LIBSVM data~\cite{LIBSVMdata} and summarized
in Table~\ref{tab:data-summary}. 
The three datasets are selected to reflect different relations between 
the sample size $n$ and the feature dimensionality $d$, which cover
$n \gg d$ (Covtype), $n \approx d$ (RCV1) and $n \ll d$ (News20).
For all tasks, the data points take the form of $(a_i,b_i)$,
where $a_i\in \R^d$ is the feature vector, and $b_i\in\{-1,1\}$ is 
the binary class label.
Our goal is to minimize the regularized empirical risk:
\begin{align*}
	P(x) = \frac{1}{n} \sum_{i=1}^n \phi_i ( a_i^{T} x) + \frac{\lambda}{2}\ltwos{x}^2 \quad \mbox{where}\quad 
		\phi_i(z) = \left\{\begin{array}{ll}
            0 & \mbox{if $b_i z \geq 1$}\\[0.5ex]
        \frac{1}{2} - b_i z & \mbox{if $b_i z \leq 0$}\\[0.5ex]
		\frac{1}{2}(1-b_i z)^2 & \mbox{otherwise}.
	\end{array}
	\right.
\end{align*}
Here, $\phi_i$ is the smoothed hinge loss (see, e.g., \cite{SSZhang13SDCA}).
It is easy to verify that  the conjugate function of $\phi_i$ is 
$\phistar_i(\beta) = b_i\beta + \frac{1}{2}\beta^2$
for $b_i\beta \in [-1,0]$ and $\infty$ otherwise.

\begin{figure}[p]
  \psfrag{Log Loss}[bc]{}
  \psfrag{Number of Passes}[tc]{}
\begin{tabular}{c|ccc}
$\lambda$ & RCV1 & Covtype & News20 \\\hline
&&&\\
$10^{-4}$&
\raisebox{-.5\height}{\includegraphics[width = 0.28\textwidth]{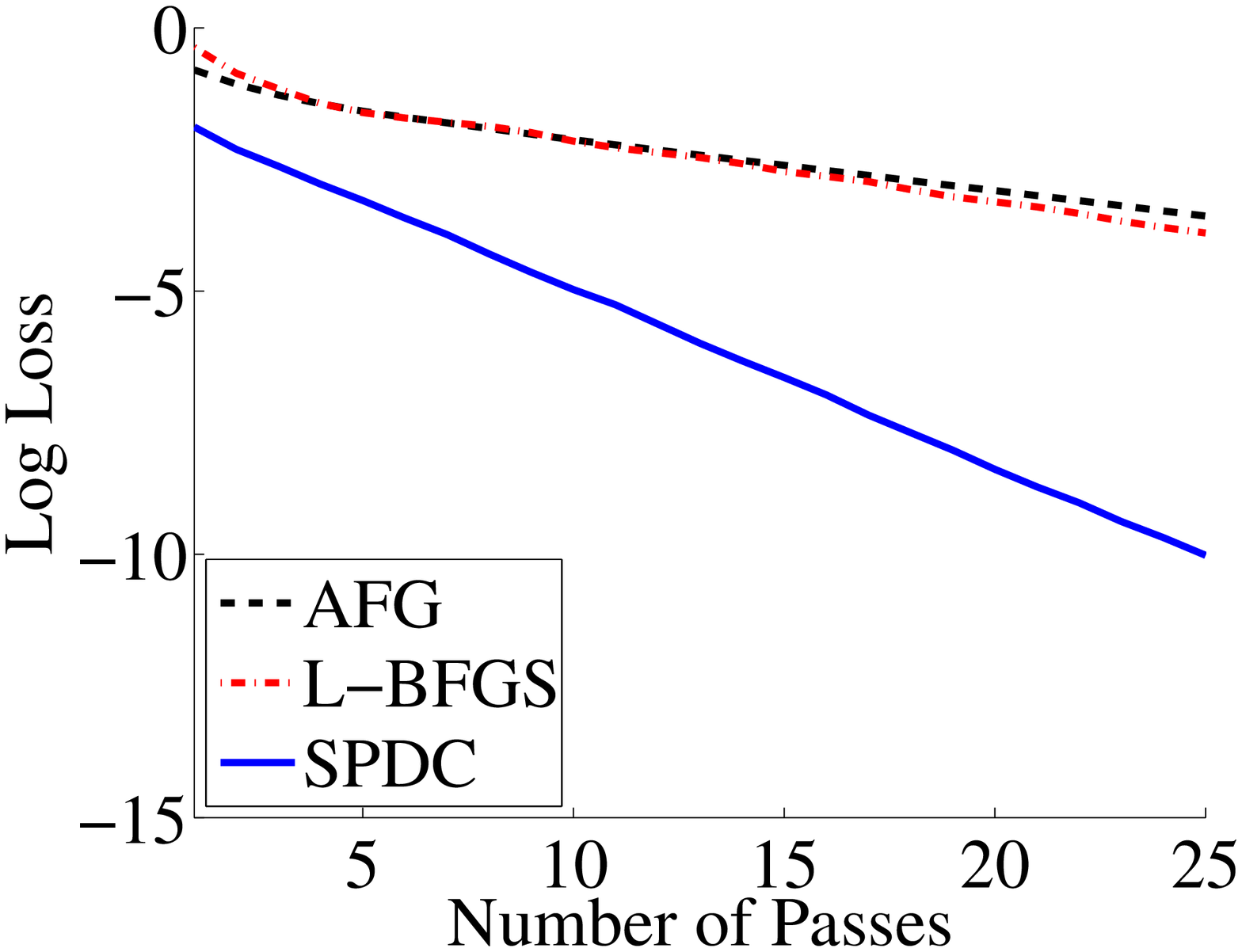}} &
\raisebox{-.5\height}{\includegraphics[width = 0.28\textwidth]{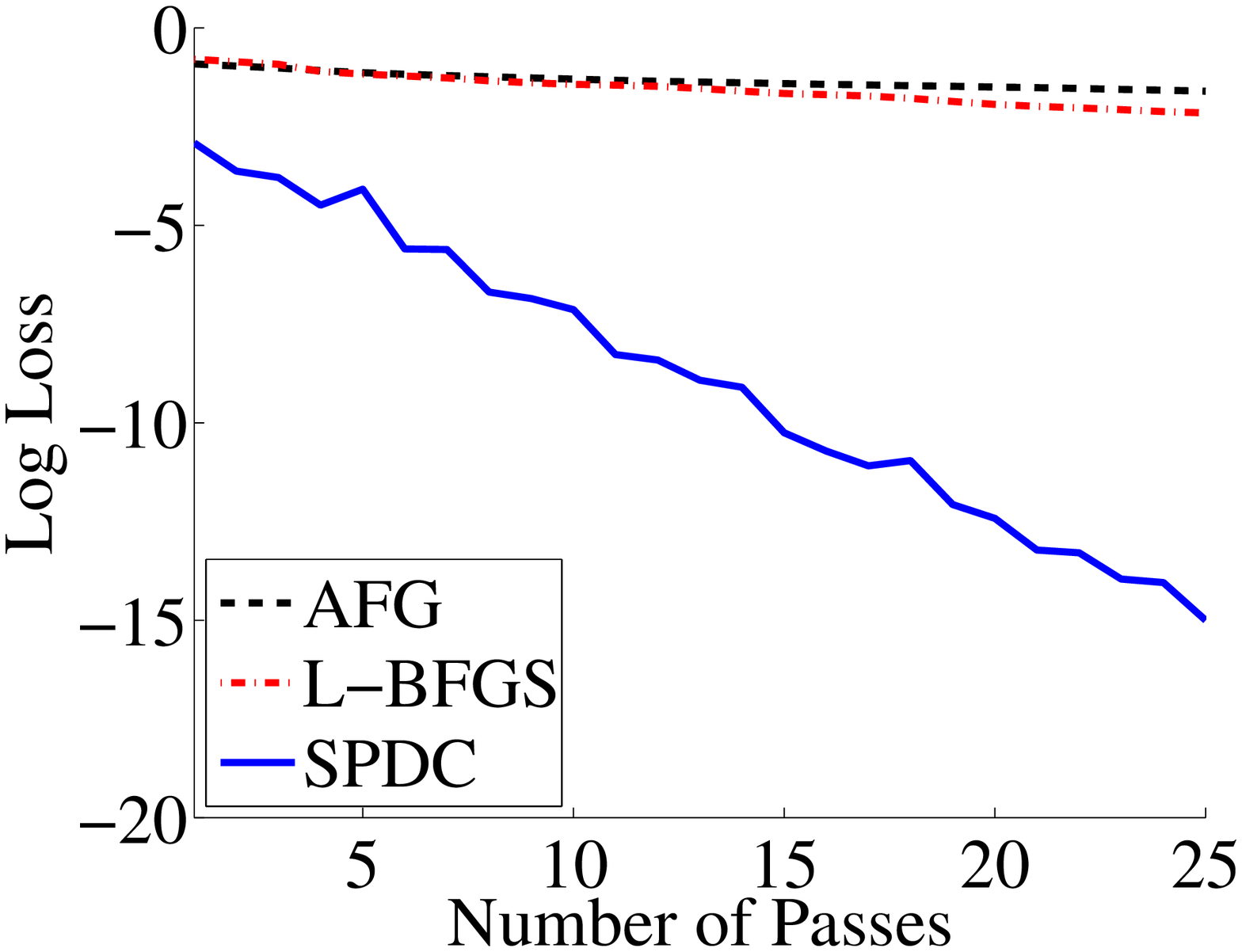}} &
\raisebox{-.5\height}{\includegraphics[width = 0.28\textwidth]{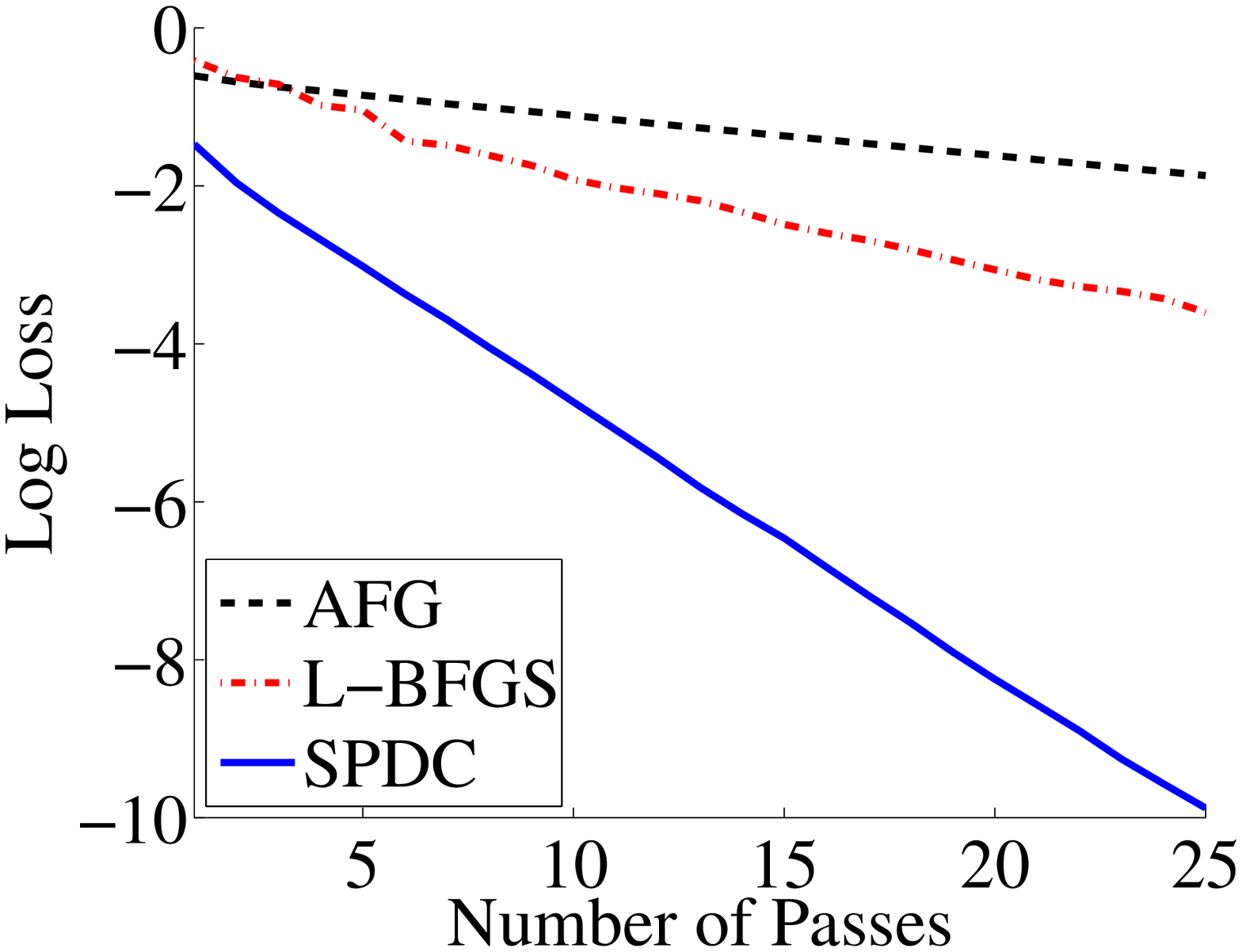}} \\
$10^{-5}$&
\raisebox{-.5\height}{\includegraphics[width = 0.28\textwidth]{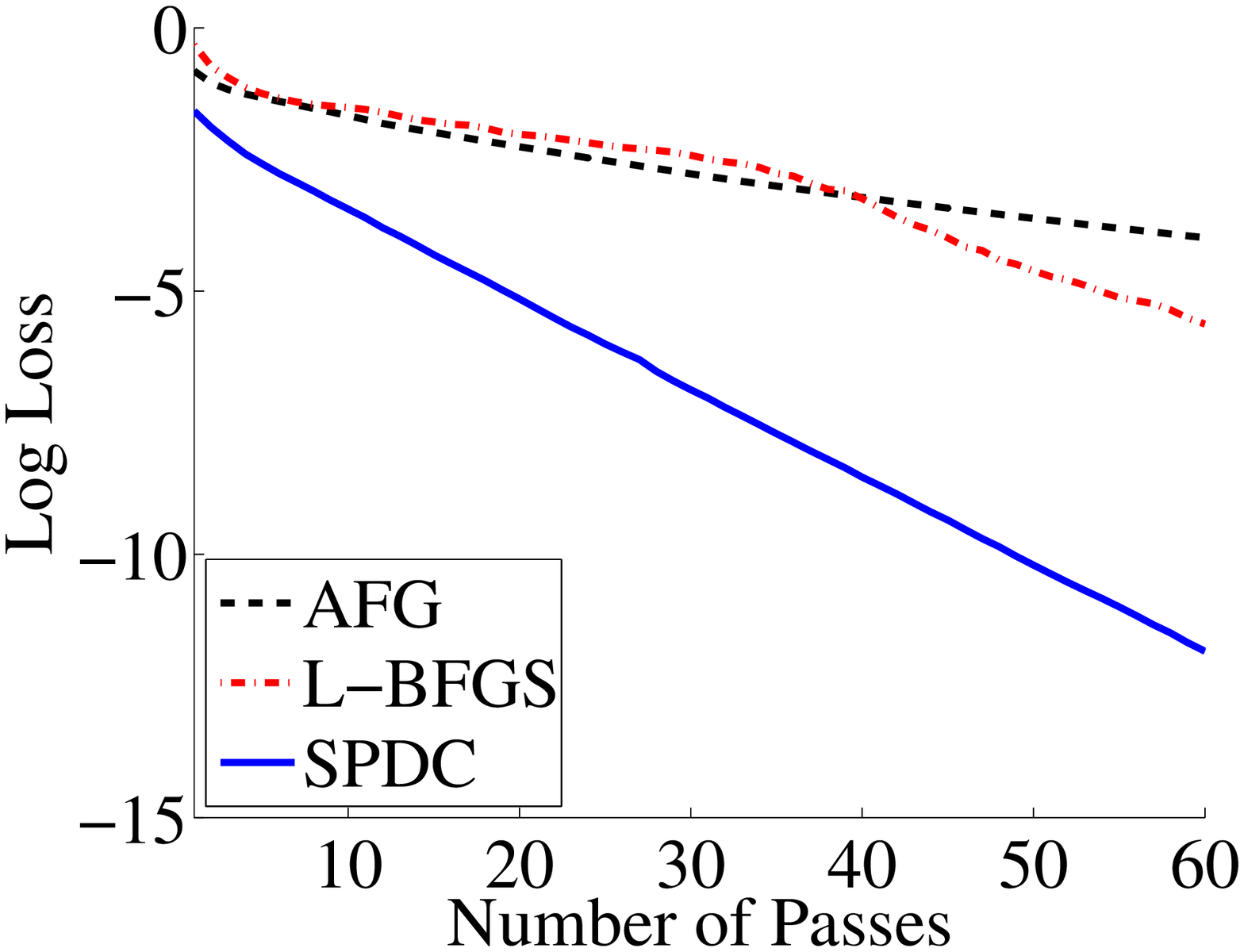}} &
\raisebox{-.5\height}{\includegraphics[width = 0.28\textwidth]{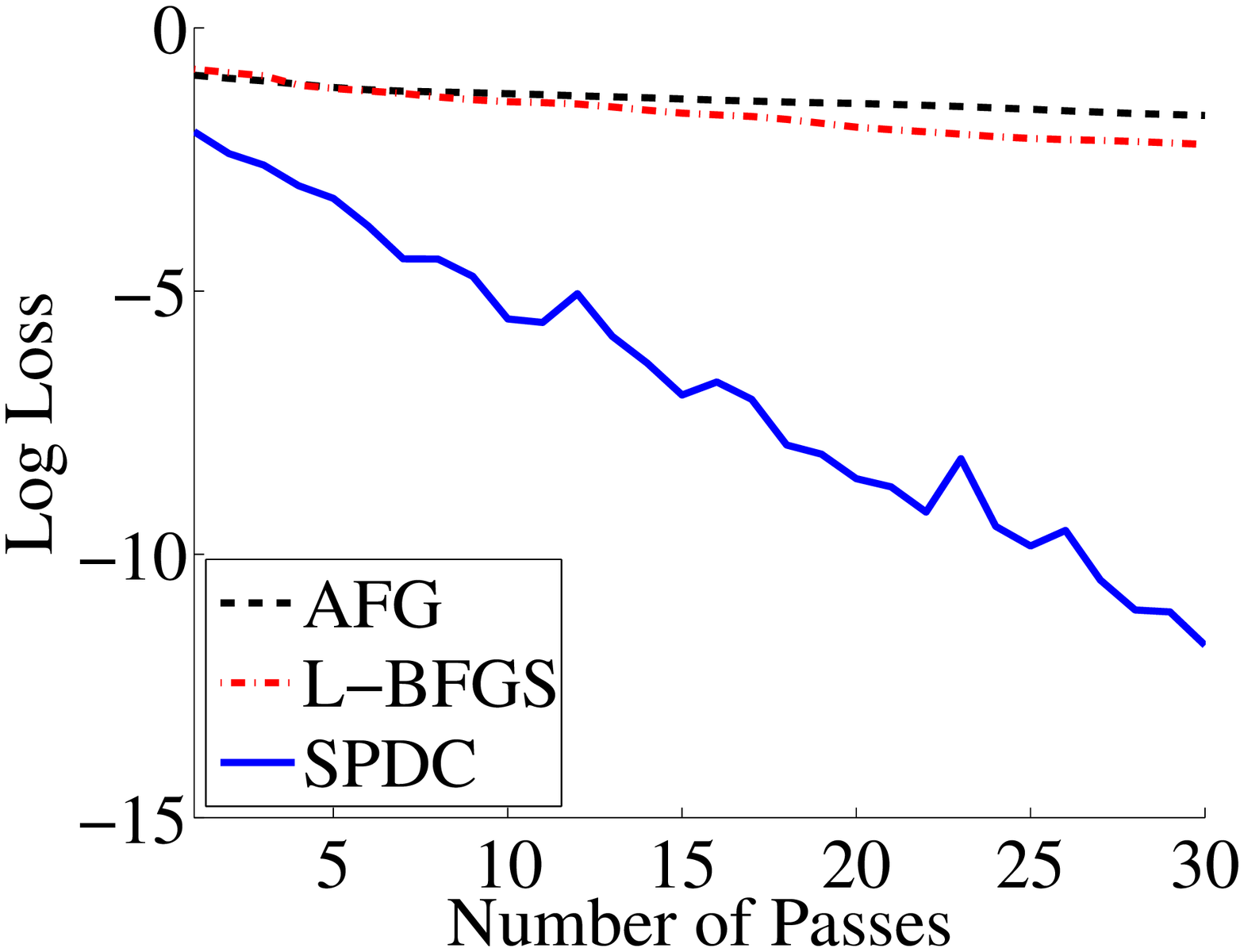}} &
\raisebox{-.5\height}{\includegraphics[width = 0.28\textwidth]{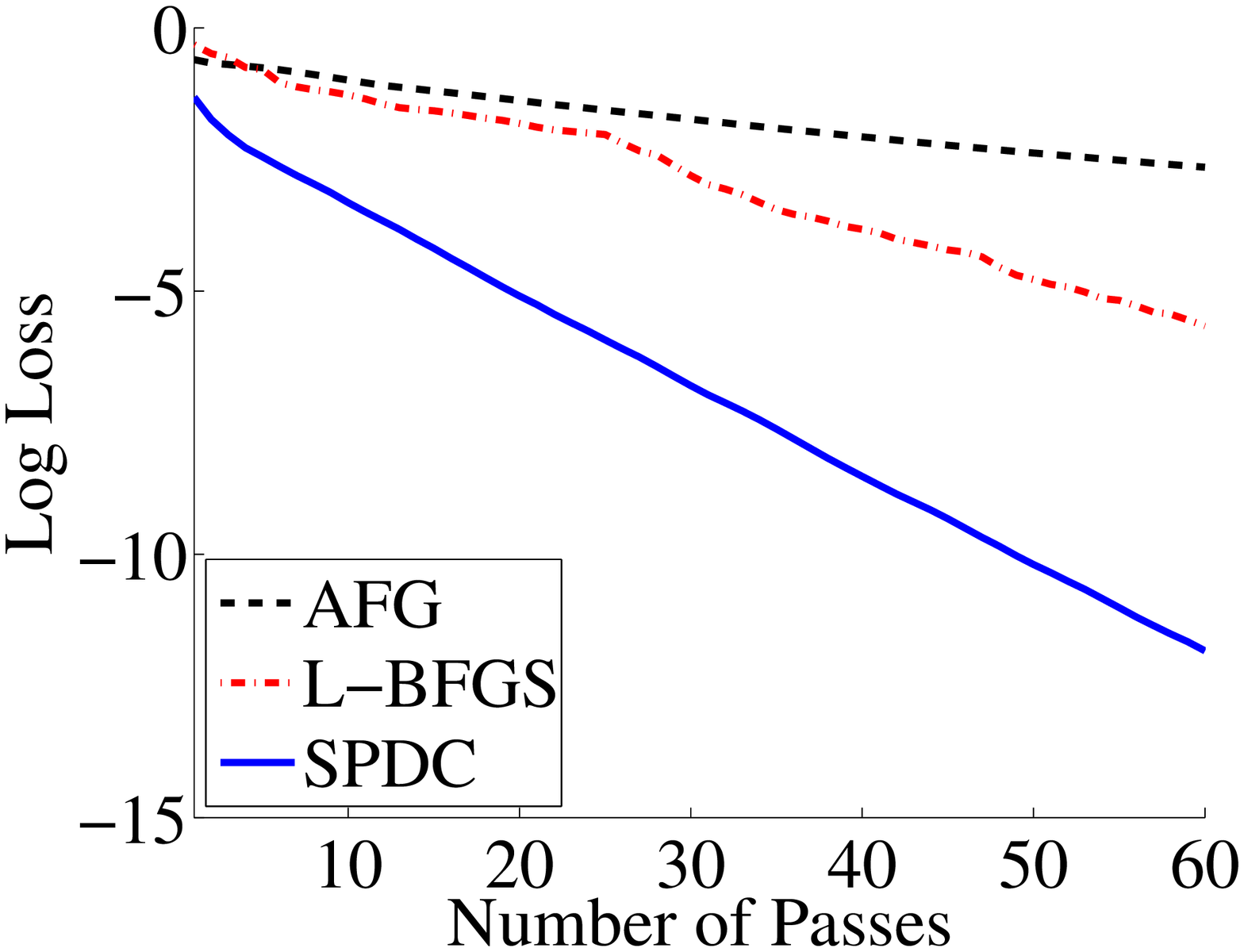}} \\
$10^{-6}$&
\raisebox{-.5\height}{\includegraphics[width = 0.28\textwidth]{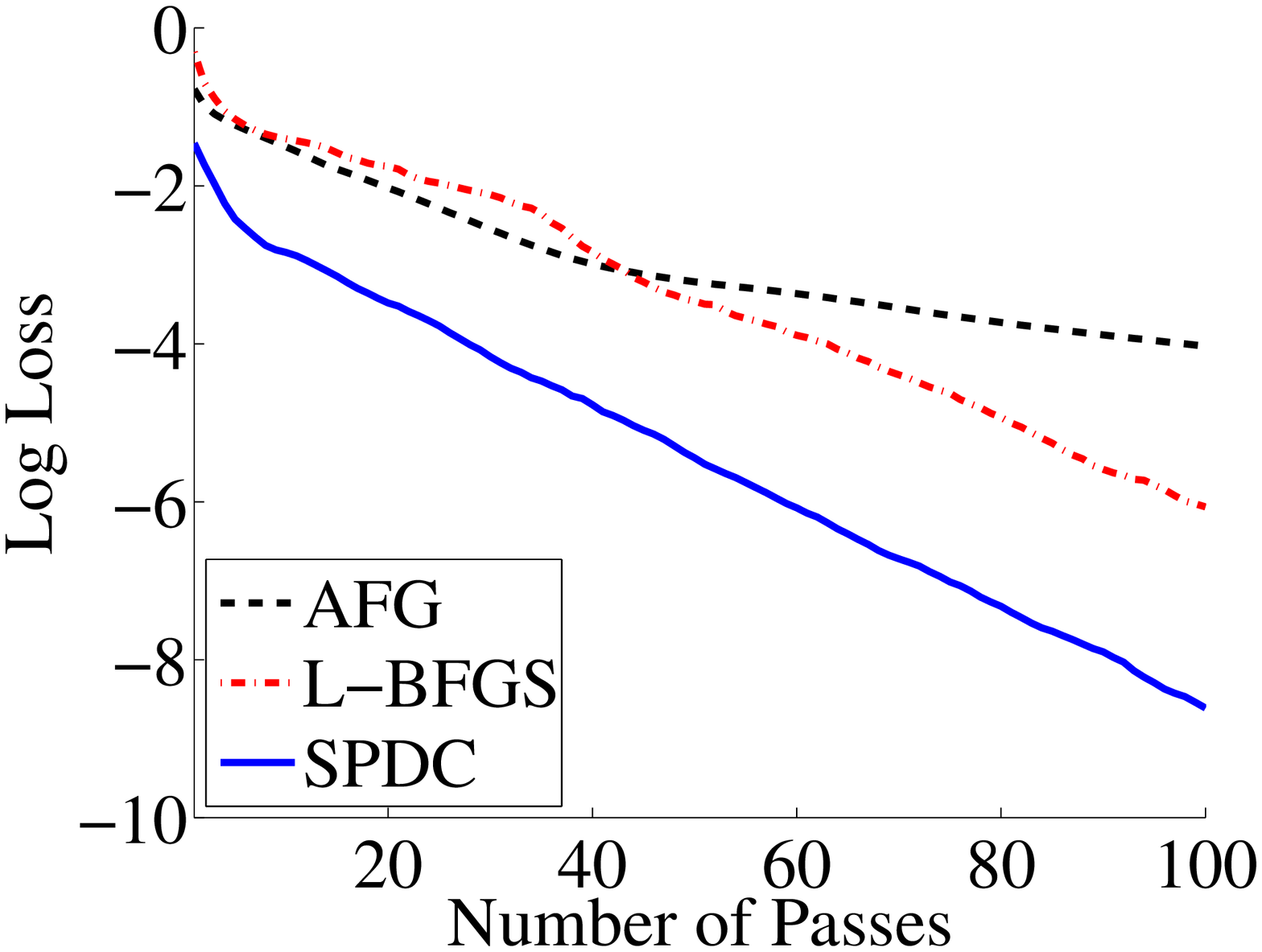}} &
\raisebox{-.5\height}{\includegraphics[width = 0.28\textwidth]{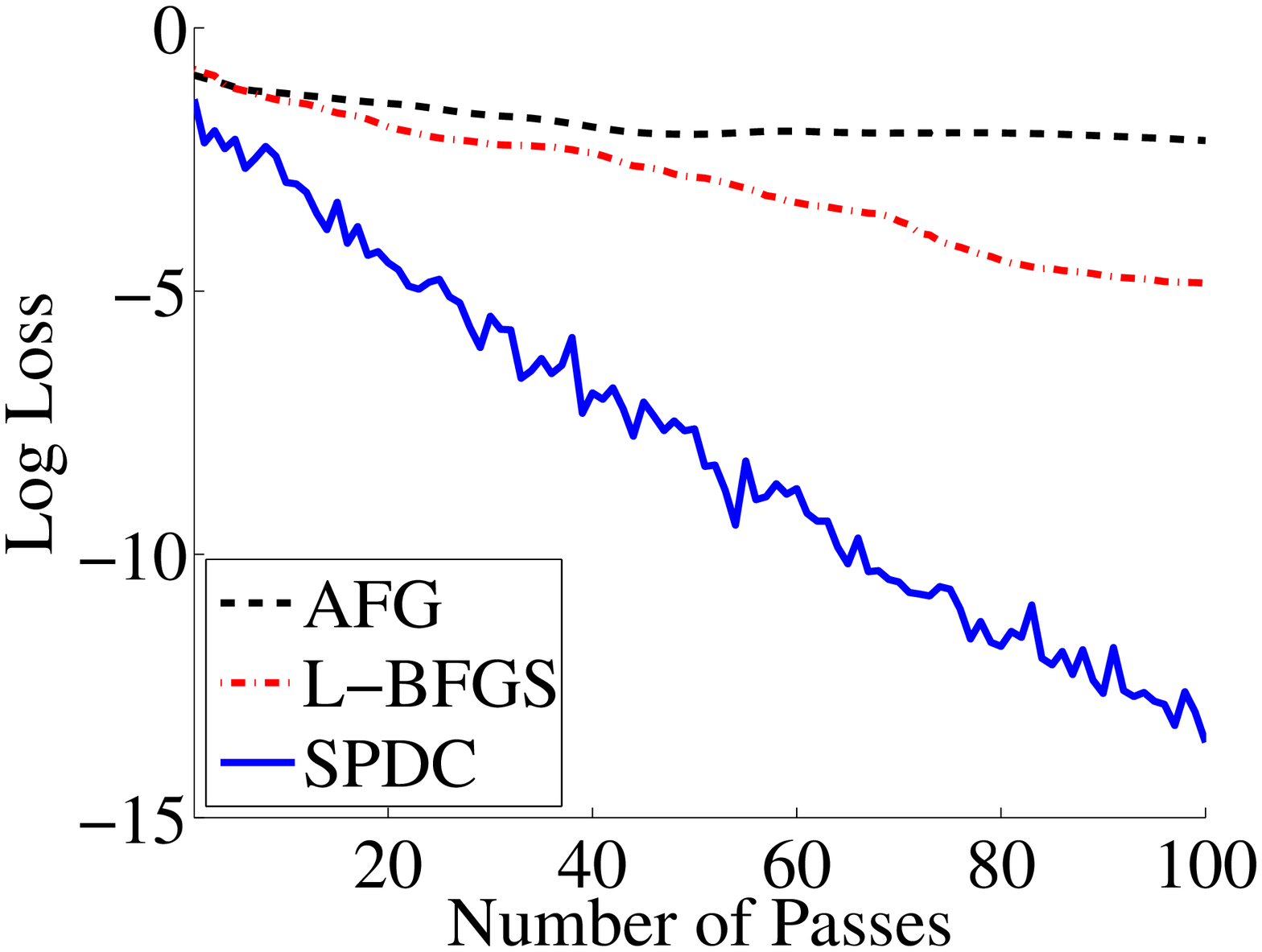}} &
\raisebox{-.5\height}{\includegraphics[width = 0.28\textwidth]{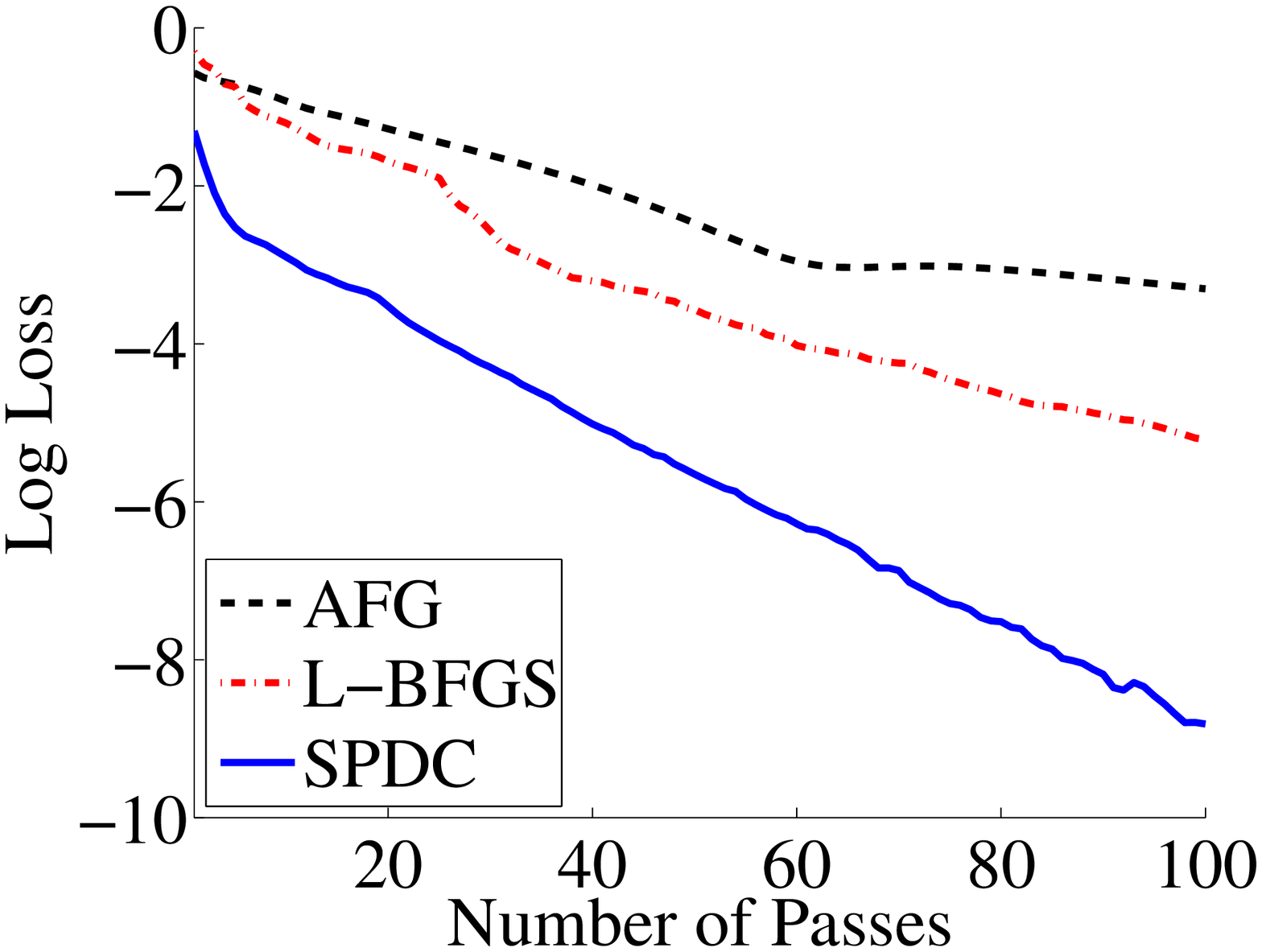}} \\
$10^{-7}$&
\raisebox{-.5\height}{\includegraphics[width = 0.28\textwidth]{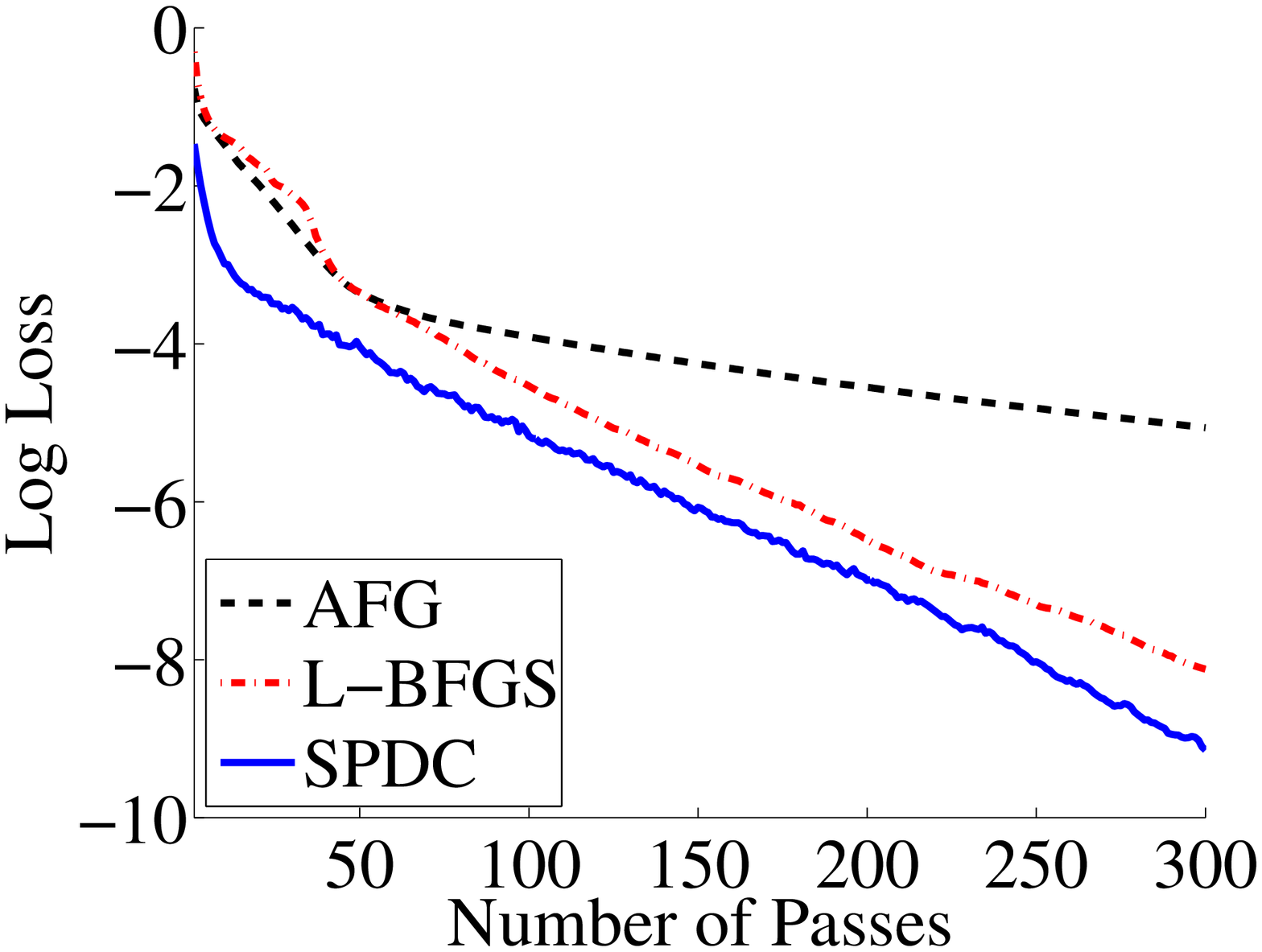}} &
\raisebox{-.5\height}{\includegraphics[width = 0.28\textwidth]{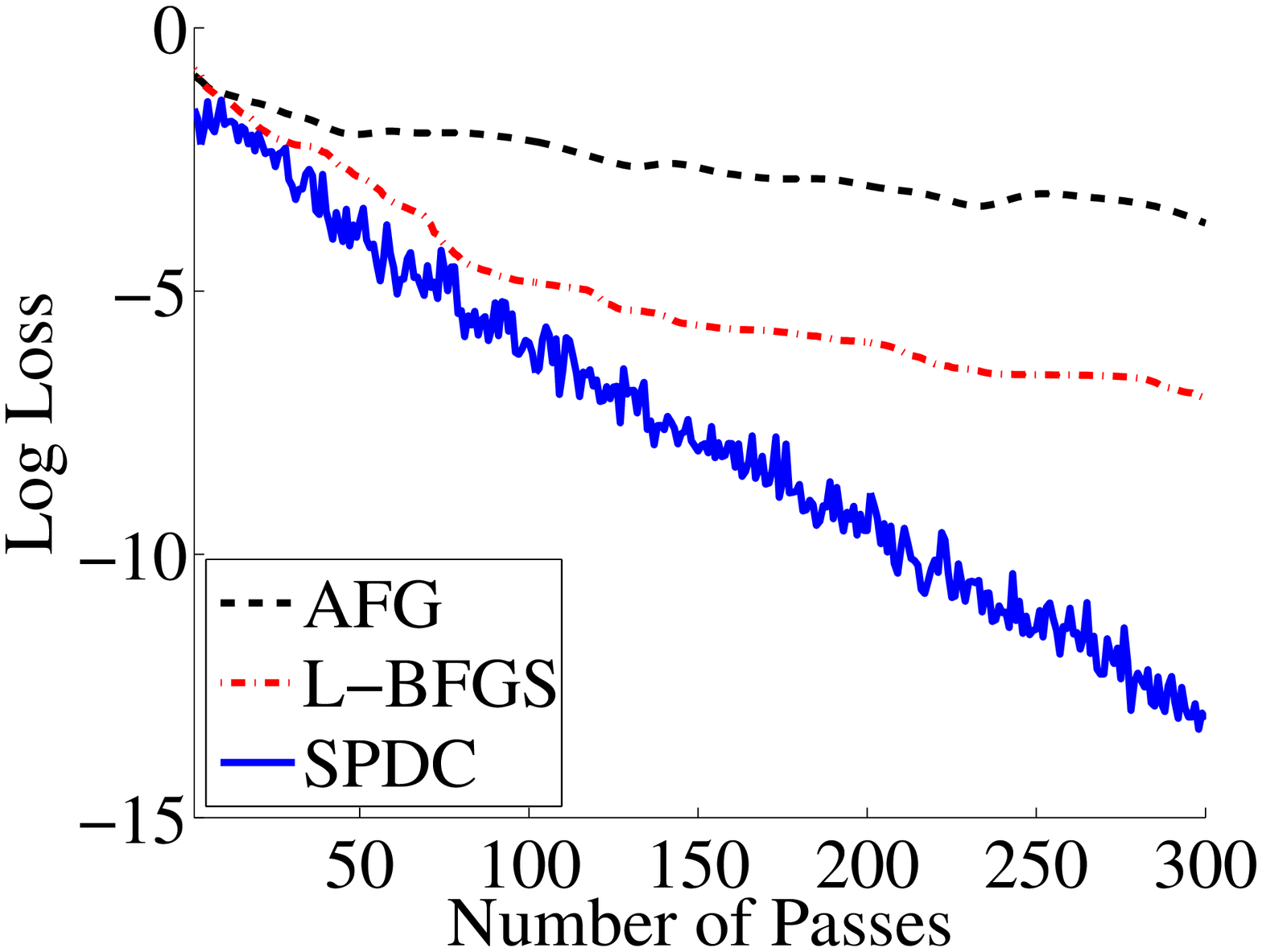}} &
\raisebox{-.5\height}{\includegraphics[width = 0.28\textwidth]{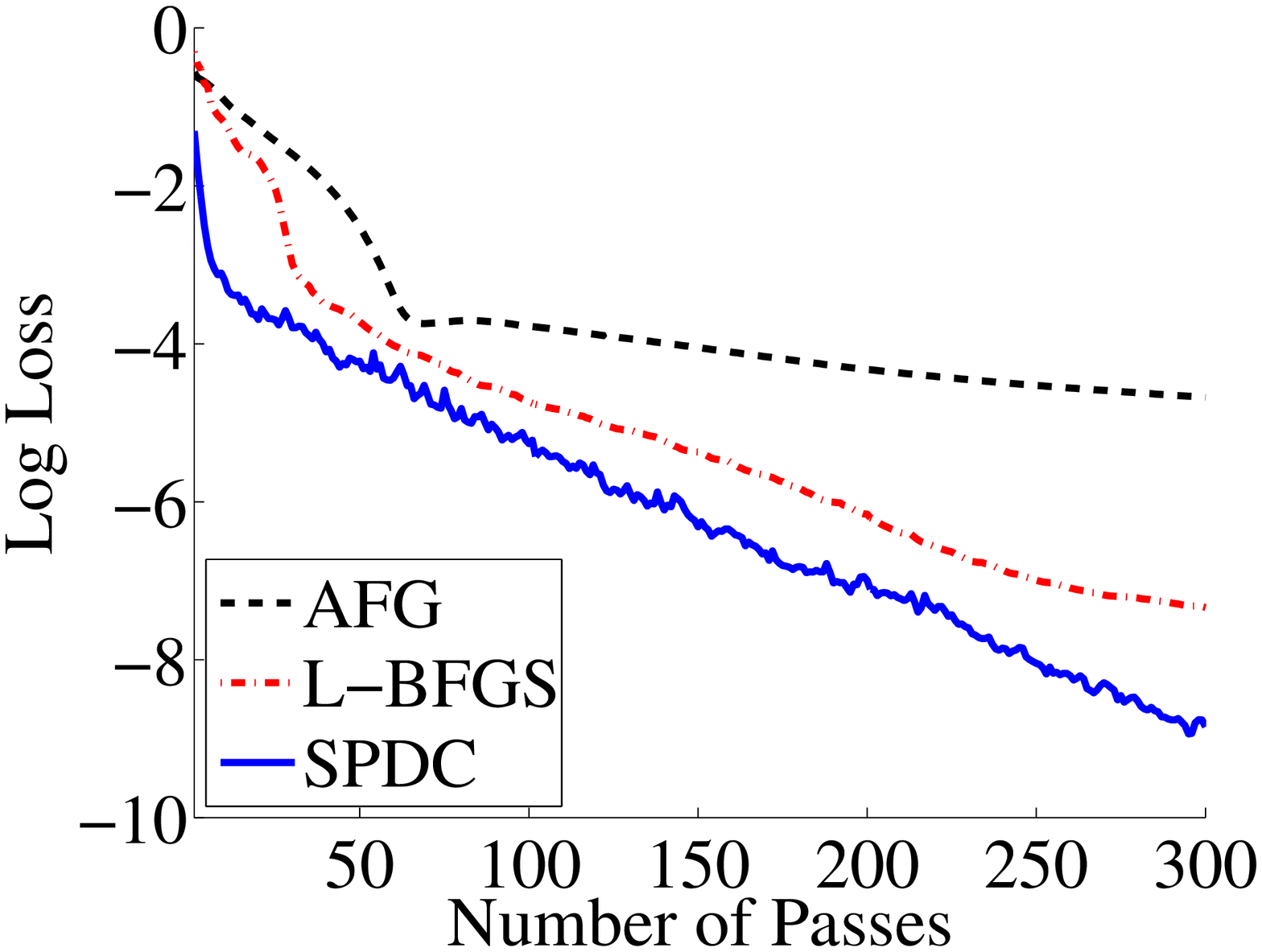}} \\
$10^{-8}$&
\raisebox{-.5\height}{\includegraphics[width = 0.28\textwidth]{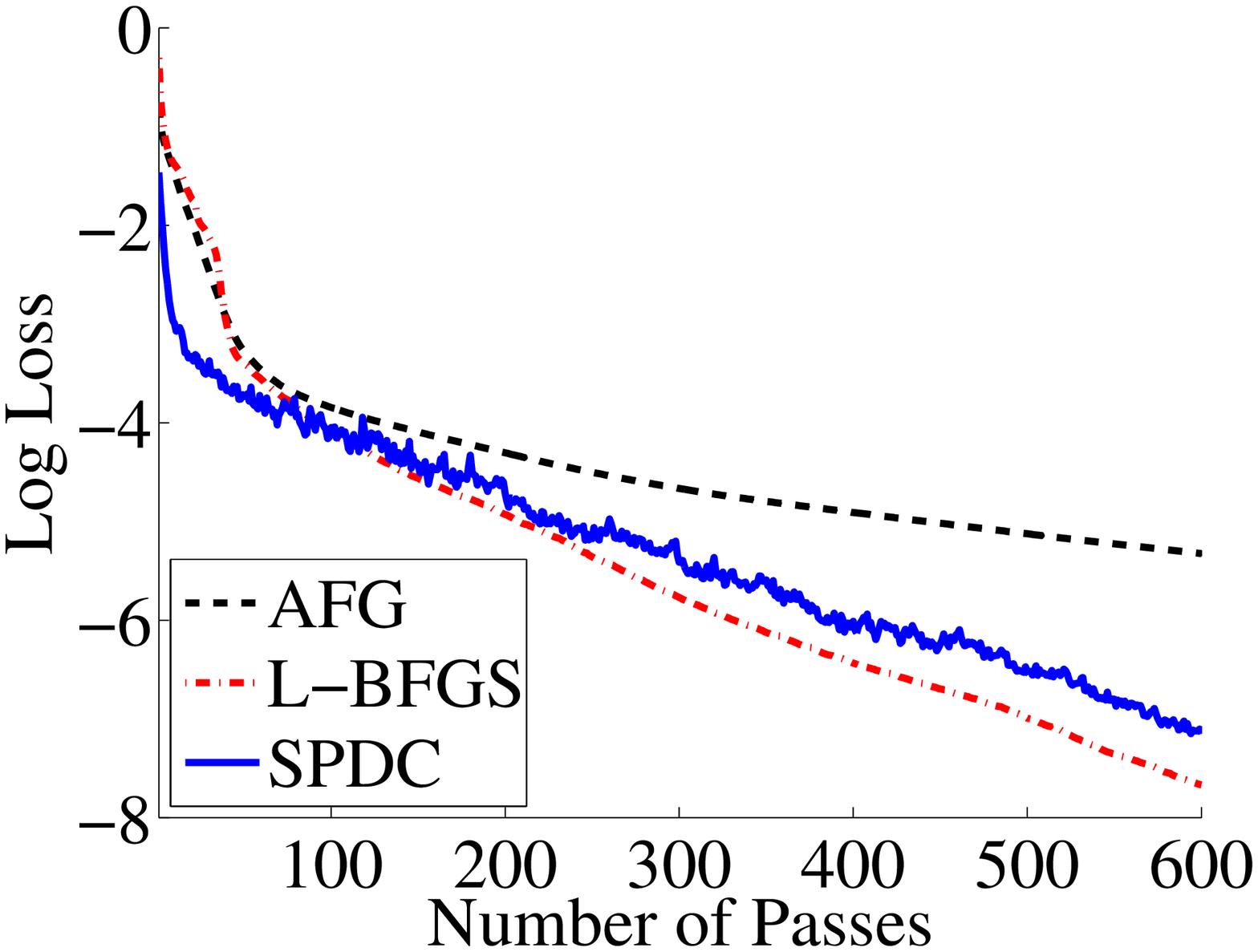}} &
\raisebox{-.5\height}{\includegraphics[width = 0.28\textwidth]{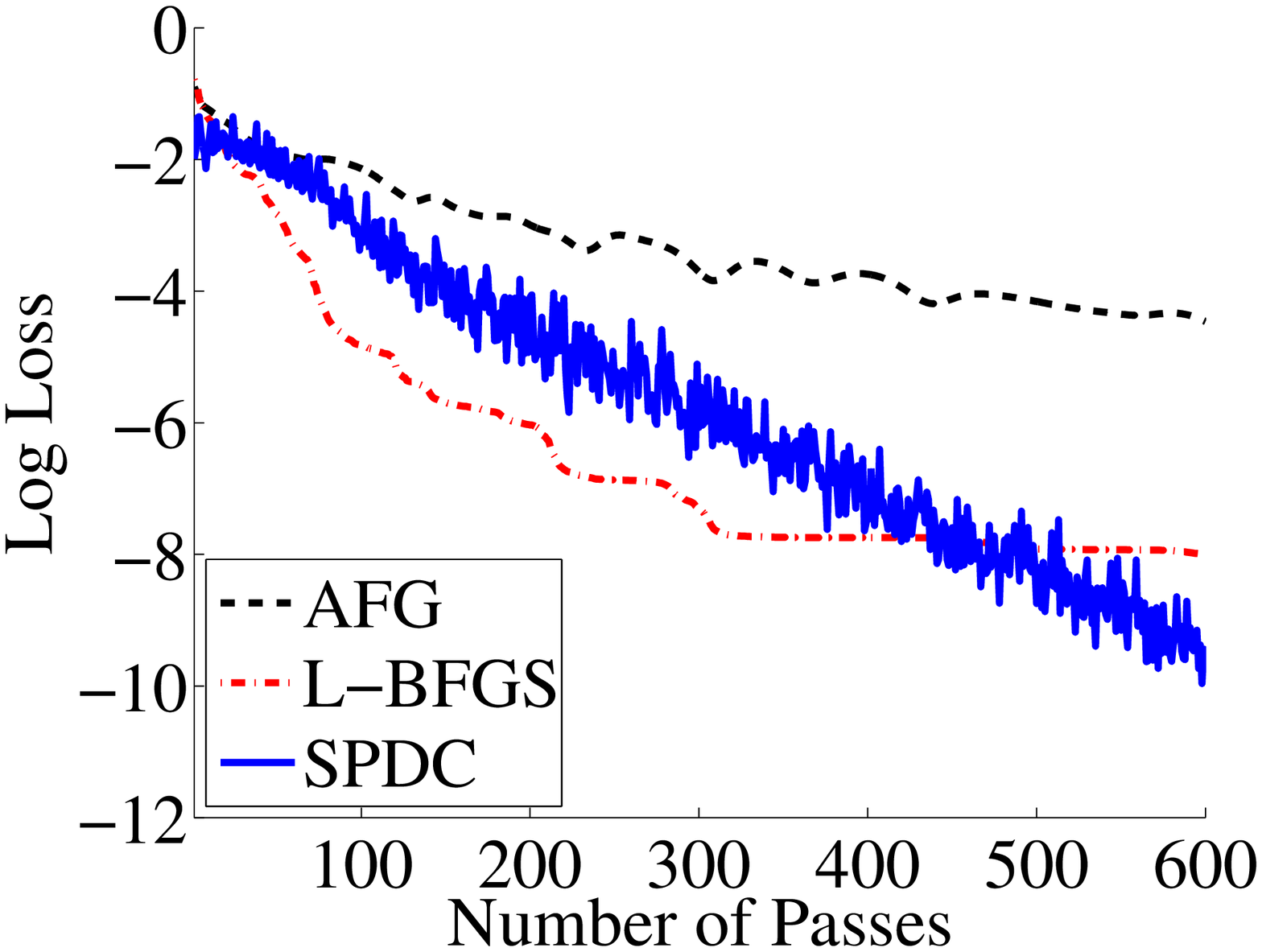}} &
\raisebox{-.5\height}{\includegraphics[width = 0.28\textwidth]{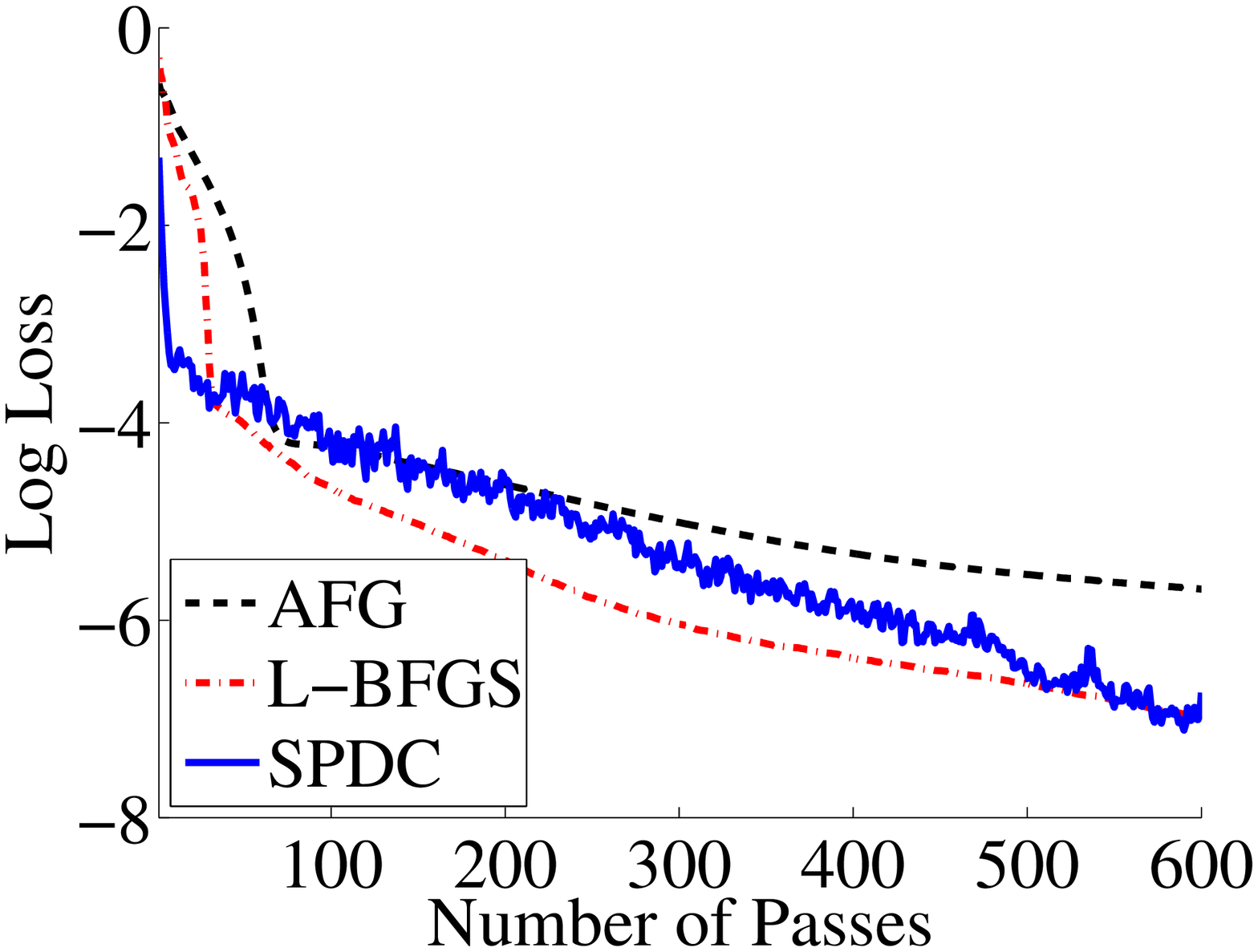}} \\
\end{tabular}
\vspace{2ex}
\caption{Comparing SPDC with AFG and L-BFGS on three real datasets with smoothed
    hinge loss. 
The horizontal axis is the number of passes through the entire dataset, and
the vertical axis is the logarithmic optimality gap $\log(P(x^{(t)}) - P(\xhat))$. 
The SPDC algorithm is faster than the two batch methods when~$\lambda$ 
is relatively large.}
\label{fig:real-compare-batch}
\end{figure}

\begin{figure}[p]
  \psfrag{Log Loss}[bc]{}
  \psfrag{Number of Passes}[tc]{}
\begin{tabular}{c|ccc}
$\lambda$ & RCV1 & Covtype & News20 \\\hline
&&&\\
$10^{-4}$&
\raisebox{-.5\height}{\includegraphics[width = 0.28\textwidth]{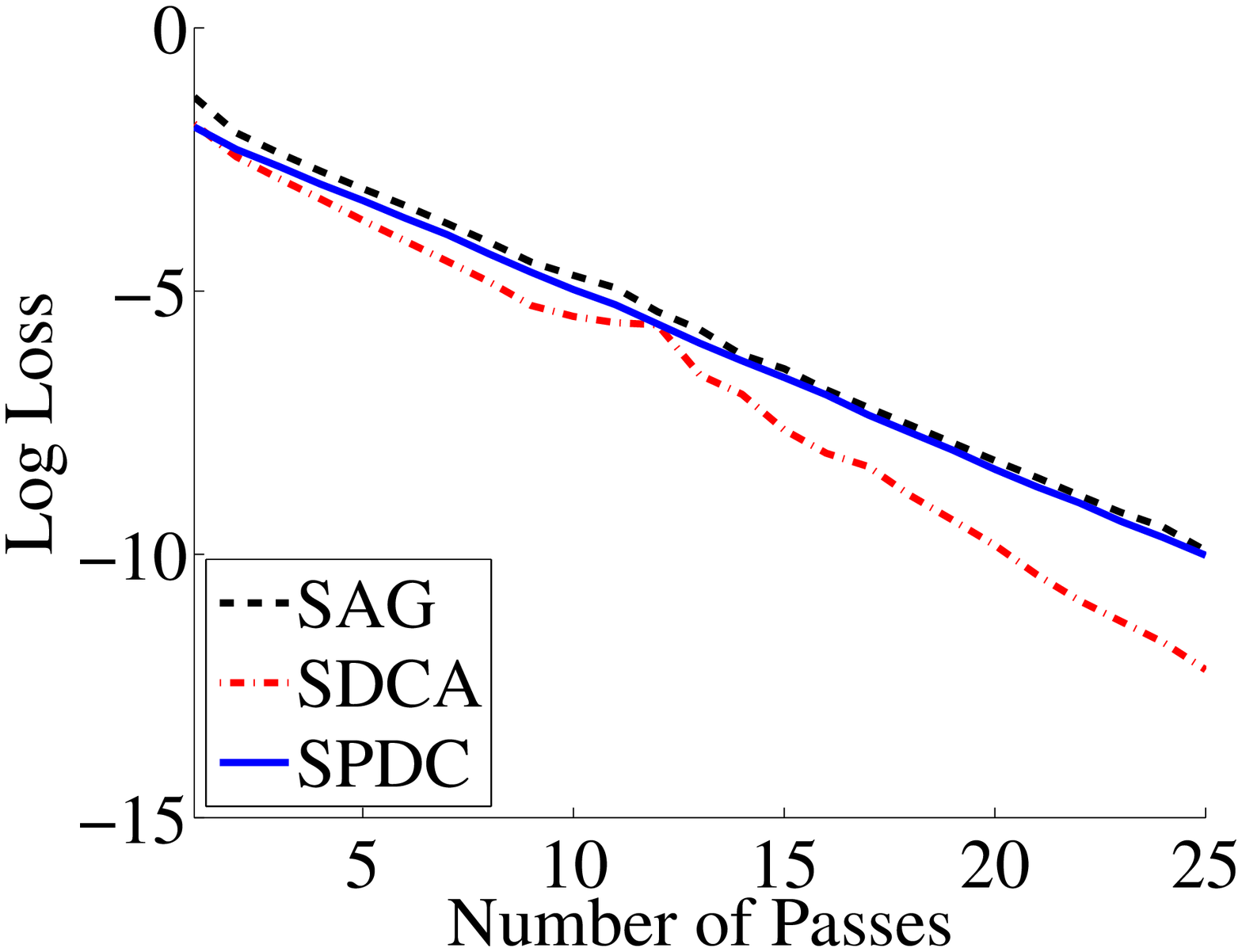}} &
\raisebox{-.5\height}{\includegraphics[width = 0.28\textwidth]{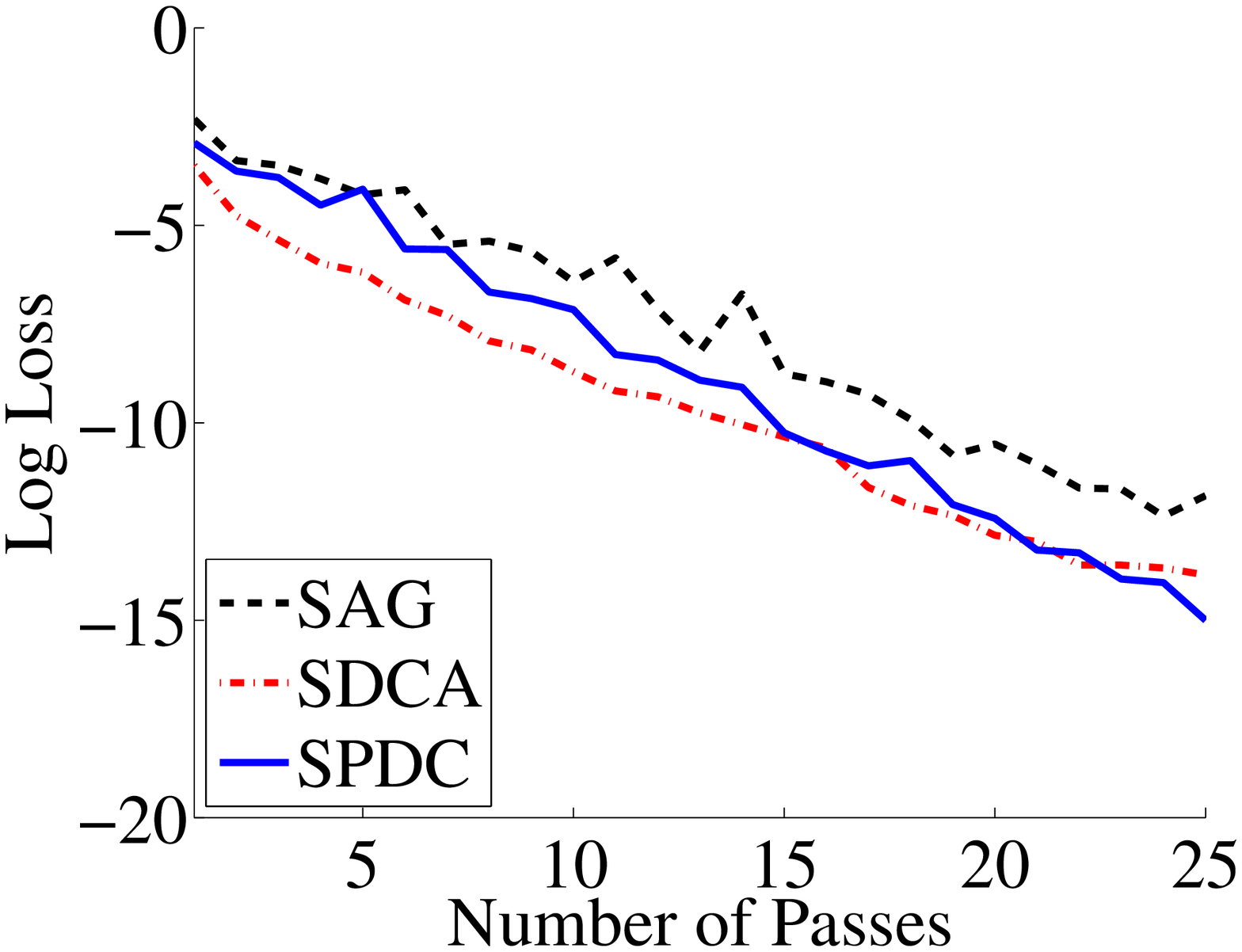}} &
\raisebox{-.5\height}{\includegraphics[width = 0.28\textwidth]{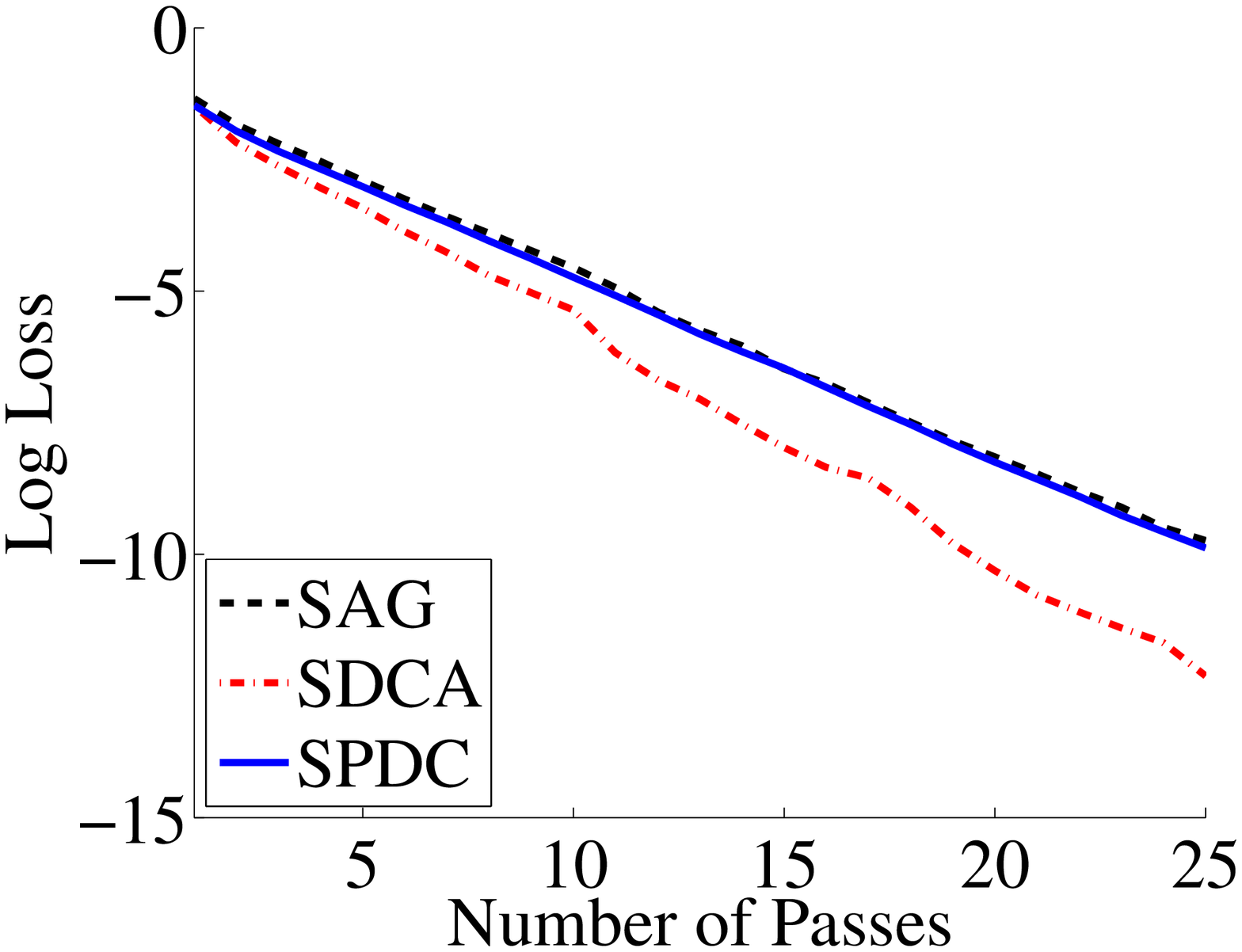}} \\
$10^{-5}$&
\raisebox{-.5\height}{\includegraphics[width = 0.28\textwidth]{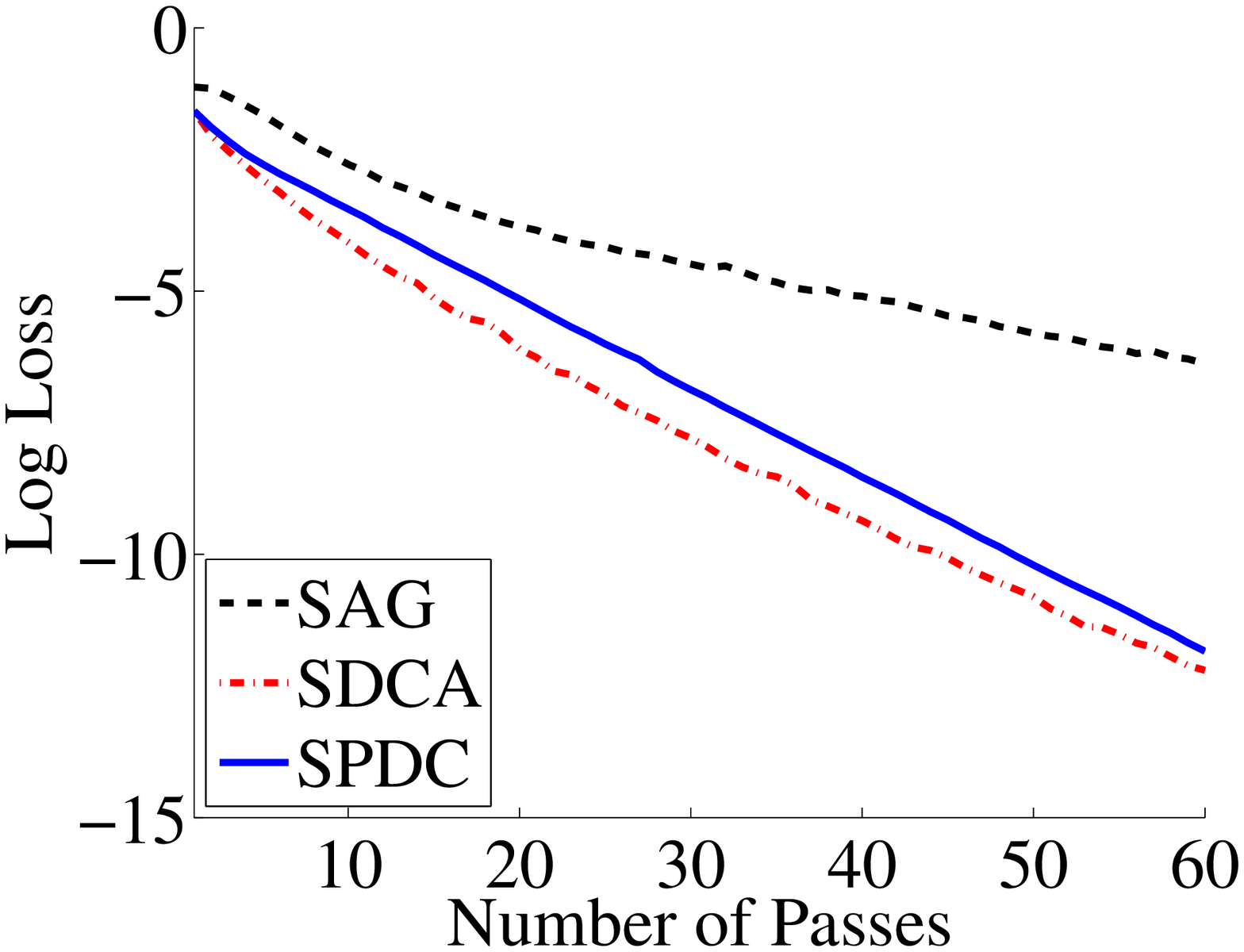}} &
\raisebox{-.5\height}{\includegraphics[width = 0.28\textwidth]{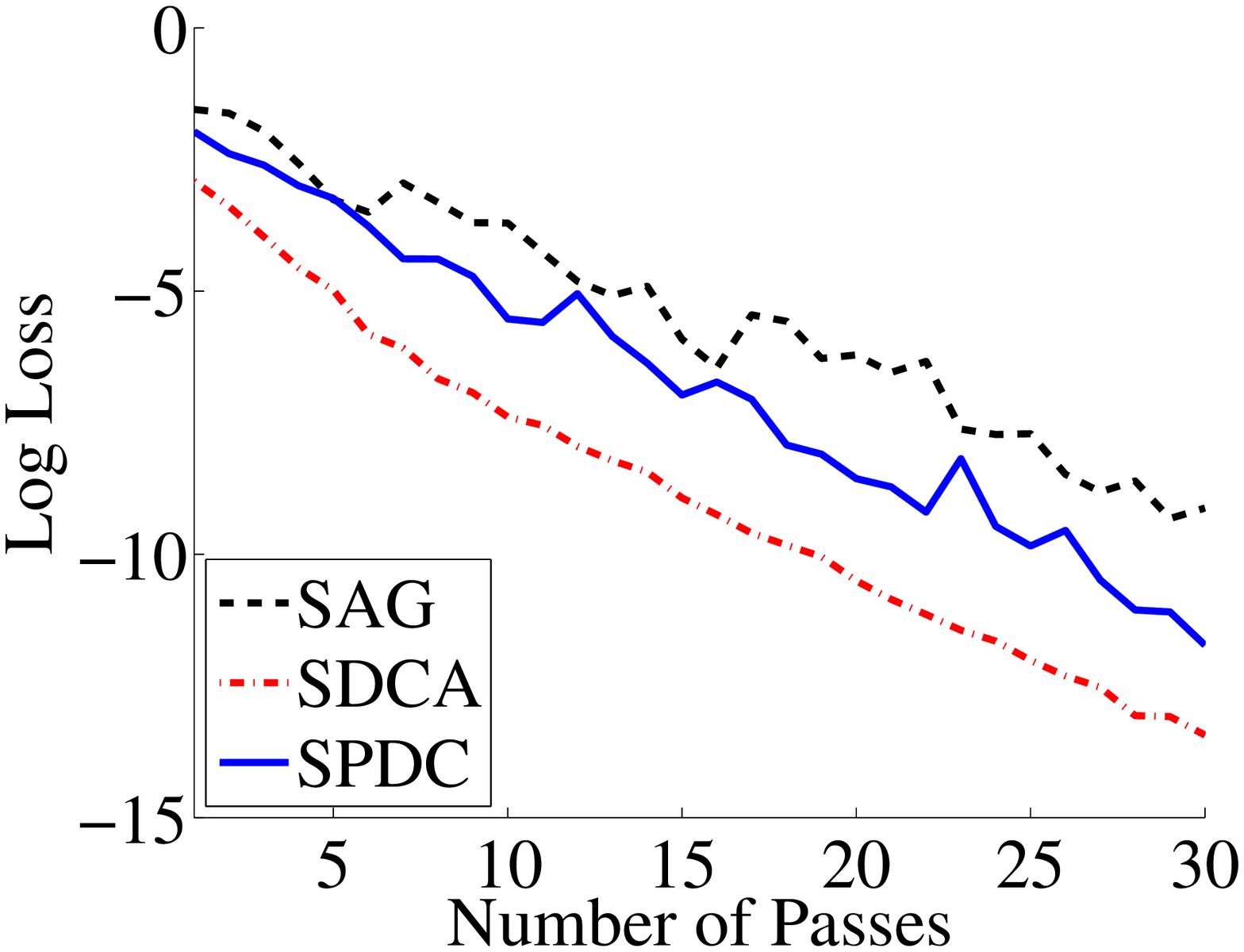}} &
\raisebox{-.5\height}{\includegraphics[width = 0.28\textwidth]{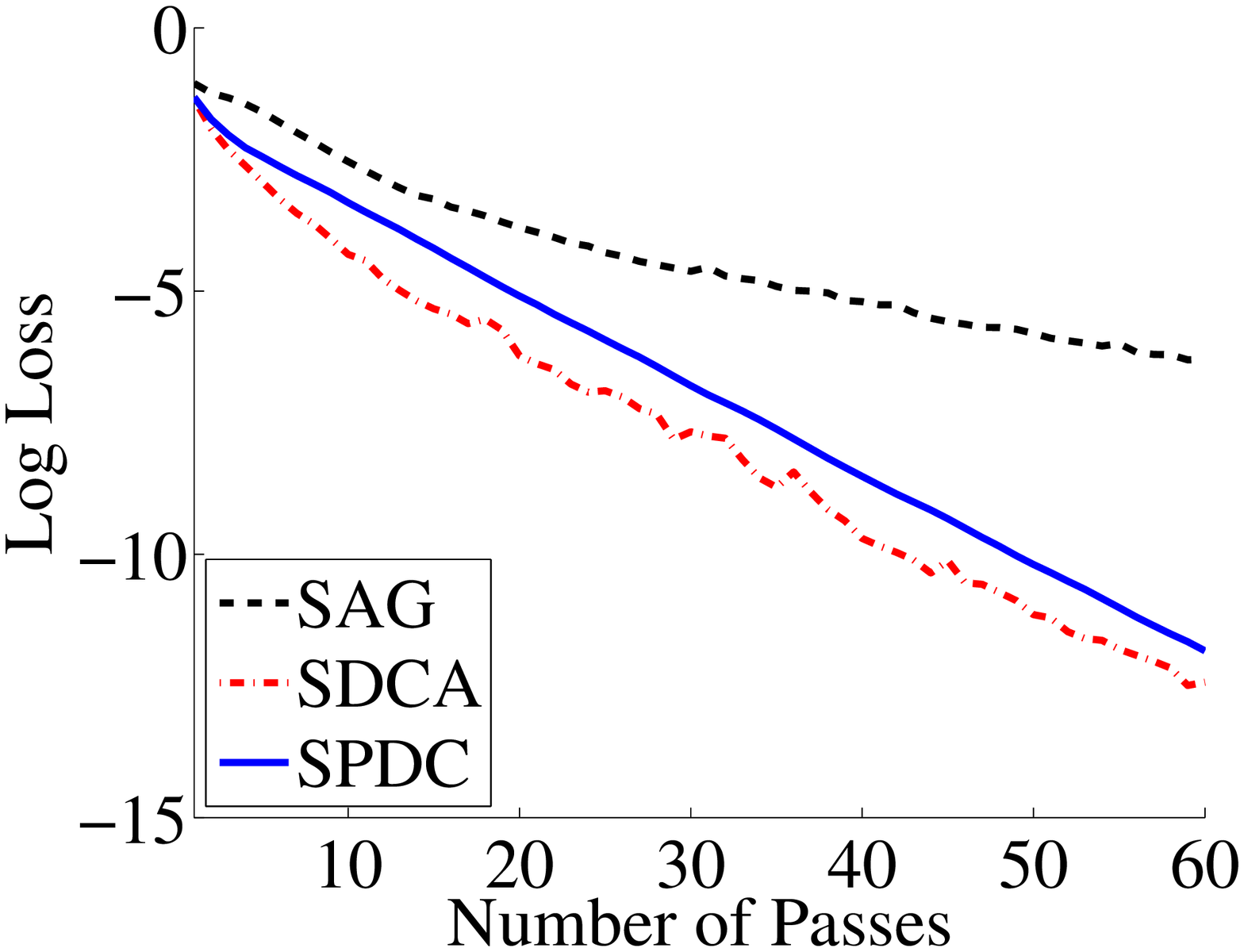}} \\
$10^{-6}$&
\raisebox{-.5\height}{\includegraphics[width = 0.28\textwidth]{rcv1-stochastic-1e-6}} &
\raisebox{-.5\height}{\includegraphics[width = 0.28\textwidth]{covtype-stochastic-1e-6}} &
\raisebox{-.5\height}{\includegraphics[width = 0.28\textwidth]{news20-stochastic-1e-6}} \\
$10^{-7}$&
\raisebox{-.5\height}{\includegraphics[width = 0.28\textwidth]{rcv1-stochastic-1e-7}} &
\raisebox{-.5\height}{\includegraphics[width = 0.28\textwidth]{covtype-stochastic-1e-7}} &
\raisebox{-.5\height}{\includegraphics[width = 0.28\textwidth]{news20-stochastic-1e-7}} \\
$10^{-8}$&
\raisebox{-.5\height}{\includegraphics[width = 0.28\textwidth]{rcv1-stochastic-1e-8}} &
\raisebox{-.5\height}{\includegraphics[width = 0.28\textwidth]{covtype-stochastic-1e-8}} &
\raisebox{-.5\height}{\includegraphics[width = 0.28\textwidth]{news20-stochastic-1e-8}} \\
\end{tabular}
\vspace{2ex}
\caption{Comparing SPDC with SAG, SDCA and ASDCA on three real datasets with smoothed
    hinge loss.
The horizontal axis is the number of passes through the entire dataset, and
the vertical axis is the logarithmic optimality gap $\log(P(x^{(T)}) - P(\xhat))$.
The SPDC algorithm is faster than SAG and SDCA
when~$\lambda$ is small. It is faster than ASDCA on datasets RCV1 and News20.}
\label{fig:real-compare-stochastic}
\end{figure}

The performance of the five algorithms are plotted in Figure~\ref{fig:real-compare-batch}
and Figure~\ref{fig:real-compare-stochastic}. 
In Figure~\ref{fig:real-compare-batch}, we compare SPDC with the two batch 
methods: AFG and L-BFGS. 
The results show that SPDC is substantially faster than AFG and L-BFGS 
for relatively large $\lambda$, illustrating the advantage of stochastic 
methods over batch methods on well-conditioned problems. 
As~$\lambda$ decreases to $10^{-8}$, the batch methods (especially L-BFGS) 
become comparable to SPDC. 

In Figure~\ref{fig:real-compare-stochastic},
we compare SPDC with the three stochastic methods: SAG, SDCA and ASDCA.
Note that the specification of ASDCA~\cite{SSZhang13acclSDCA} requires the regularization coefficient $\lambda$ satisfies $\lambda \leq \frac{R^2}{10n}$ where $R$ is the maximum $\ell_2$-norm of feature vectors. To satisfy this constraint, we run ASDCA with $\lambda\in\{10^{-6},10^{-7},10^{-8}\}$. In Figure~\ref{fig:real-compare-stochastic}, the observations are just the opposite to that of Figure~\ref{fig:real-compare-batch}.
All stochastic algorithms have comparable performances on relatively large $\lambda$, but SPDC and ASDCA becomes substantially faster when~$\lambda$ gets closer to zero. In particular, ASDCA converges faster than SPDC on the Covtype dataset, but SPDC is faster on the remaining two datasets. In addition, due to the outer-inner loop structure of the ASDCA algorithm, its error rate oscillates and might be bad at early iterations. In contrast, the curve of SPDC is almost linear and it is more stable than ASDCA.

Summarizing Figure~\ref{fig:real-compare-batch} and Figure~\ref{fig:real-compare-stochastic},
the performance of SPDC are always comparable or better
than the other methods in comparison.



\appendix

\section{Proof of Theorem~\ref{thm:spdc-convergence}}
\label{sec:thm1-proof}

We focus on characterizing the values of $x$ and $\yvec$ after the 
$t$-th update in Algorithm~\ref{alg:spdc-minibatch}.
For any $i\in\{1,\ldots,n\}$, let $\ytilde_i$ be the value of $\yvec_i^\suptp$
if $i\in K$, i.e., 
\[
    \ytilde_i = \arg\max_{\beta\in \R} \bigg\{  \beta \langle a_i, \xbar^\supt \rangle - \phistar_{i}(\beta) - \frac{ (\beta -\yvec_i^\supt)^2}{2\sigma} \bigg\}.
\]
Since $\phi_i$ is $(1/\gamma)$-smooth by assumption, 
its conjugate $\phistar_i$ is $\gamma$-strongly convex
(e.g., \cite[Theorem~4.2.2]{HUL01book}).
Thus the function being maximized
above is $(1/\sigma+\gamma)$-strongly concave. Therefore,
\begin{align*}
    -\yhat_i\langle a_i, \xbar^\supt\rangle +\phi_i^*(\yhat_i)+\frac{(\yhat_i-\yvec_i^\supt)^2}{2\sigma} 
    \geq &
    -\ytilde_i\langle a_i, \xbar^\supt\rangle +\phi_i^*(\ytilde_i)+\frac{(\ytilde_i-\yvec_i^\supt)^2}{2\sigma} \\
    & +\Big(\frac{1}{\sigma}+\gamma\Big)\frac{(\ytilde_i-\yhat_i)^2}{2} .
\end{align*}
Multiplying both sides of the above inequality by $m/n$ and re-arrange terms, 
we have
\begin{align}
  \frac{m}{2\sigma n}(\yvec_i^\supt-\yhat_i)^2  \,\geq\,
& \Big(\frac{1}{\sigma}+\gamma\Big)\frac{m}{2n}(\ytilde_i-\yhat_i)^2 
  + \frac{m}{2\sigma n}(\ytilde_i-\yvec_i^\supt)^2 \nonumber \\
&  - \frac{m}{n}(\ytilde_i-\yhat_i)\langle a_i, \xbar^\supt \rangle 
  + \frac{m}{n}\bigl(\phi_i^*(\ytilde_i)-\phi_i^*(\yhat_i)\bigr) .
\label{eqn:y-update-sc}
\end{align}
According to Algorithm~\ref{alg:spdc-minibatch}, the set~$K$ of indices 
to be updated are chosen randomly. 
For every specific index $i$, the event $i\in K$ happens with probability $m/n$.
If $i\in K$, then $\yvec_i^\suptp$ is updated to the value $\ytilde_i$, which
satisfies inequality~\eqref{eqn:y-update-sc}. 
Otherwise, $\yvec_i^\suptp$ is assigned by its old value $\yvec_i^\supt$. 
Let $\field_t$ be the sigma field generated by all random variables 
defined before round $t$, and taking expectation conditioned on $\field_t$, 
we have
\begin{align*}
	\E[(\yvec_i^\suptp - \yhat_i)^2 | \field_t ] &= \frac{m (\ytilde_i - \yhat_i)^2}{n} + \frac{(n-m)(\yvec_i^\supt - \yhat_i)^2}{n},\\
	\E[(\yvec_i^\suptp - \yvec_i^\supt)^2 | \field_t ] &= \frac{m (\ytilde_i - \yvec_i^\supt)^2}{n},\\
	\E[\yvec_i^\suptp | \field_t ] &= \frac{m \ytilde_i}{n} + \frac{(n-m)\yvec_i^\supt}{n} \\
  \E[\phi_i^*(\yvec_i^\suptp)|\field_t] &= \frac{m}{n} \phi_i^*(\ytilde_i) + \frac{n-m}{n}\phi_i^*(\yvec_i^\supt) .
\end{align*}
As a result, we can represent 
$(\ytilde_i - \yhat_i)^2$, $(\ytilde_i - \yvec_i^\supt)^2$, $\ytilde_i$
and $\phi_i^*(\ytilde_i)$
in terms of the conditional expectations on $(\yvec_i^\suptp - \yhat_i)^2$, 
$(\yvec_i^\suptp - \yvec_i^\supt)^2$, $\yvec_i^\suptp$ and 
$\phi_i^*(\yvec_i^\suptp)$, respectively.
Plugging these representations into inequality~\eqref{eqn:y-update-sc}
and re-arranging terms, we obtain
\begin{align}
 \left( \frac{1}{2\sigma} + \frac{(n-m)\gamma}{2n} \right)(\yvec_i^\supt - \yhat_i)^2 
\geq & \left( \frac{1}{2\sigma} + \frac{\gamma}{2}\right)
 \E[(\yvec_i^\suptp - \yhat_i)^2 | \field_t] + \frac{1}{2\sigma} \E[(\yvec_i^\suptp - \yvec_i^\supt)^2| \field_t] \nonumber\\
 & - \left( \frac{m}{n}(\yvec_i^\supt-\yhat_i)+\E[\yvec_i^\suptp-\yvec_i^\supt|\field_t] \right) \langle a_i, \xbar^\supt \rangle \nonumber \\
 & + \E[\phi_i^*(\yvec_i^\suptp)|\field_t]-\phi_i^*(\yvec_i^\supt) 
  + \frac{m}{n}\bigl(\phi_i^*(\yvec_i^\supt)-\phi_i^*(\yhat_i)\bigr) .
\end{align}
Then summing over all indices $i=1,2,\dots,n$ and dividing both sides of the resulting inequality by~$m$, we have
\begin{align}
  \left( \frac{1}{2\sigma} + \frac{(n-m )\gamma}{2n} \right) \frac{\ltwos{\yvec^\supt - \yhat}^2 }{m}
& \geq \left( \frac{1}{2\sigma} + \frac{\gamma}{2}\right)
     \frac{\E[\ltwos{\yvec^\suptp - \yhat}^2 | \field_t]}{m} 
     +\frac{1}{2\sigma} \frac{\E[\ltwos{\yvec^\suptp - \yvec^\supt}^2| \field_t]}{m}\nonumber\\
&\quad +\E\Big[\frac{1}{m}\sum_{k\in K}\bigl(\phi_k^*(\yvec_k^\suptp)-\phi_k^*(\yvec_k^\supt) \bigl)\Big| \field_t \Big] 
+\frac{1}{n}\sum_{i=1}^n \bigl(\phi_i^*(\yvec_i^\supt)-\phi_i^*(\yhat_i)\bigr) \nonumber \\
&\quad - \E\Big[ \Big\langle \uvec^\supt - \uhat + \frac{n}{m}(u^\suptp-u^\supt) ,\, \xbar^\supt \Big\rangle \Big| \field_t \Big] ,
	\label{eqn:alpha-expected-inequality-final}
\end{align}
where we used the shorthand notations
(appeared in Algorithm~\ref{alg:spdc-minibatch})
\begin{equation}\label{eqn:u-notation}
  \uvec^\supt = \frac{1}{n}\sum_{i=1}^n \yvec_i^\supt a_i, \qquad
  \uvec^\suptp = \frac{1}{n}\sum_{i=1}^n \yvec_i^\suptp a_i , \qquad
  \textrm{and} \qquad
  \uhat = \frac{1}{n}\sum_{i=1}^n \yhat_i a_i . 
\end{equation}
Since only the dual coordinates with indices in~$K$ are updated, we have
\[
  \frac{n}{m} (u^\suptp-u^\supt) = \frac{1}{m} \sum_{i=1}^n (\yvec_i^\suptp - \yvec_i^\supt)a_i = \frac{1}{m}\sum_{k\in K}(\yvec_k^\suptp - \yvec_k^\supt)a_k .
\]

\bigskip

We also derive an inequality characterizing the relation between $x^\suptp$ and $x^\supt$. 
Since the function being minimized on the right-hand side 
of~\eqref{eqn:spdc-minibatch-minimize} has strong convexity parameter
$1/\tau+\lambda$ and $x^\suptp$ is the minimizer, we have
\begin{align}
& g(\xhat) + \Big\langle \uvec^\supt + \frac{n}{m}(u^\suptp-u^\supt), ~\xhat \Big\rangle +  \frac{\ltwos{x^\supt-\xhat}^2}{2\tau} \label{eqn:thm-1-aux-1}\\
\geq ~&
g(x^\suptp) + \Big\langle \uvec^\supt + \frac{n}{m}(u^\suptp-u^\supt), ~x^\suptp \Big\rangle 
+\left(\frac{1}{2\tau}+\frac{\lambda}{2}\right) \ltwos{x^\suptp-\xhat}^2 \nonumber \nonumber \\
& +  \frac{\ltwos{x^\suptp-x^\supt}^2}{2\tau} . \nonumber
\end{align}
Rearranging terms and taking expectation conditioned on $\field_t$, we have 
\begin{align}
\frac{\ltwos{x^\supt-\xhat}^2}{2\tau} 
& \geq \left(\frac{1}{2\tau}+\frac{\lambda}{2}\right) \E[\ltwos{x^\suptp-\xhat}^2|\field_t] + \frac{\E[\ltwos{x^\suptp-x^\supt}^2|\field_t]}{2\tau} \nonumber\\
&\quad + \E\left[g(x^\suptp) - g(x^\star)|\field_t\right] \nonumber\\
&\quad + \E\Big[\Big\langle \uvec^\supt + \frac{n}{m}(u^\suptp-u^\supt), ~x^\suptp - \xhat \Big\rangle \Big| \field_t \Big] .
\label{eqn:x-side-ineq}
\end{align}

\bigskip

In addition, we consider a particular combination of the saddle-point function
values at different points.
By the definition of $f(x,y)$ in~\eqref{eqn:min-max-saddle} and the notations
in~\eqref{eqn:u-notation}, we have
\begin{align}
&  f(x^\suptp,y^\star) - f(x^\star, y^\star) + \frac{n}{m}\left(f(x^\star, y^\star) - f(x^\star, y^\suptp)\right) - \frac{n-m}{m}\left( f(x^\star, y^\star) - f(x^\star, y^\supt) \right) \nonumber \\
= ~&  f(x^\suptp,y^\star) - f(x^\star, y^\supt) + \frac{n}{m}\left(f(x^\star, y^\supt) - f(x^\star, y^\suptp)\right) \nonumber \\
= ~& \langle u^\star, x^\suptp \rangle - \frac{1}{n}\sum_{i=1}^n \phi_i^*(y^\star_i) + g(x^\suptp)
- \langle u^\supt, x^\star \rangle + \frac{1}{n}\sum_{i=1}^n \phi_i^*(y^\supt_i) - g(x^\star) \nonumber \\
& +\frac{n}{m} \biggl( \langle u^\supt, x^\star \rangle - \frac{1}{n}\sum_{i=1}^n \phi_i^*(y^\supt_i) + g(x^\star) - \langle u^\suptp, x^\star \rangle + \frac{1}{n}\sum_{i=1}^n \phi_i^*(y^\suptp_i) - g(x^\star) \biggr) \nonumber \\
= ~& \frac{1}{n}\sum_{i=1}^n \left(\phi_i^*(y^\supt_i)-\phi_i^*(y^\star_i)\right) + \frac{1}{m}\sum_{k\in K}\left(\phi_i^*(y^\suptp_k)-\phi_i^*(y^\supt_k)\right) + g(x^\suptp) - g(x^\star) \nonumber \\
& + \langle u^\star, x^\suptp \rangle - \langle u^\supt, x^\star \rangle
+ \frac{n}{m} \langle u^\supt-u^\suptp, x^\star \rangle .
\label{eqn:weighted-saddle-func}
\end{align}

Next we add both sides of the
inequalities~\eqref{eqn:alpha-expected-inequality-final}
and~\eqref{eqn:x-side-ineq} together, and then subtract
equality~\eqref{eqn:weighted-saddle-func} after taking expectation
with respect to~$\field_t$.
This leads to the following inequality:
\begin{align}
& \frac{\ltwos{x^\supt-\xhat}^2}{2\tau} 
+\left( \frac{1}{2\sigma} + \frac{(n-m )\gamma}{2n} \right) \frac{\ltwos{\yvec^\supt - \yhat}^2 }{m}
+ \frac{n-m}{m}\left( f(x^\star, y^\star) - f(x^\star, y^\supt) \right) \nonumber\\
\geq ~&  \left(\frac{1}{2\tau}+\frac{\lambda}{2}\right) \E[\ltwos{x^\suptp-\xhat}^2|\field_t] 
+ \left( \frac{1}{2\sigma} + \frac{\gamma}{2}\right)
     \frac{\E[\ltwos{\yvec^\suptp - \yhat}^2 | \field_t]}{m} 
+ \frac{\E[\ltwos{x^\suptp-x^\supt}^2|\field_t]}{2\tau} \nonumber\\
& + \frac{\E[\ltwos{\yvec^\suptp - \yvec^\supt}^2| \field_t]}{2\sigma m}
+ \E \left[ f(x^\suptp,y^\star) - f(x^\star, y^\star) + \frac{n}{m}\left(f(x^\star, y^\star) - f(x^\star, y^\suptp)\right) \bigg| \field_t \right] \nonumber \\
& + \E\left[ \left\langle u^\supt-u^\star+\frac{n}{m} (u^\suptp-u^\supt), 
~ x^\suptp-\xbar^\supt \right\rangle \Big| \field_t \right] .
\label{eqn:pd-inequality-basic}
\end{align}
We need to lower bound the last term on the right-hand-side of the above inequality.
To this end, we have
\begin{align}
& \left\langle u^\supt-u^\star+\frac{n}{m} (u^\suptp-u^\supt), ~ x^\suptp-\xbar^\supt \right\rangle \nonumber \\ 
= ~& \biggl(\frac{\yvec^\supt - \yhat}{n} + \frac{\yvec^\suptp - \yvec^\supt}{m}\biggr)^{T} A (x^\suptp - x^\supt - \theta(x^\supt - x^\suptm)) \nonumber\\
= ~& \frac{(\yvec^\suptp - \yhat)^T A (x^\suptp - x^\supt)}{n} 
 - \frac{\theta(\yvec^\supt - \yhat)^T A (x^\supt - x^\suptm)}{n} \nonumber\\
& + \frac{n-m}{m n}(\yvec^\suptp - \yvec^\supt)^T A (x^\suptp - x^\supt) 
 - \frac{\theta}{m}(\yvec^\suptp - \yvec^\supt)^T A (x^\supt - x^\suptm).
\label{eqn:last-term-expansion}
\end{align}
Recall that $\ltwos{a_k}\leq R$ and,
according to~\eqref{eqn:tau-sigma-theta},
$1 / \tau  = 4 \sigma R^2$. 
Therefore,
\begin{align*}
	|(\yvec^\suptp - \yvec^\supt)^T A (x^\suptp - x^\supt)| &\leq \frac{\ltwos{x^\suptp - x^\supt}^2}{4\tau / m} 
	+ \frac{\ltwos{(\yvec^\suptp - \yvec^\supt)^T A}^2}{m/\tau} \\
	&= \frac{\ltwos{x^\suptp - x^\supt}^2}{4\tau / m} 
	+ \frac{(\sum_{k\in K} |\yvec_k^\suptp - \yvec_k^\supt| \cdot \ltwos{a_k})^2}{4 m \sigma R^2} \\
	&\leq \frac{m \ltwos{x^\suptp - x^\supt}^2}{4\tau} + \frac{\ltwos{\yvec^\suptp - \yvec^\supt}^2}{4\sigma},
\end{align*}
Similarly, we have
\begin{align*}
	|(\yvec^\suptp - \yvec^\supt)^T A (x^\supt - x^\suptm)| 
    &\leq \frac{m \ltwos{x^\supt - x^\suptm}^2}{4\tau} 
    + \frac{\ltwos{\yvec^\suptp - \yvec^\supt}^2}{4\sigma}.
\end{align*}
The above upper bounds on the absolute values imply
\begin{align*}
	(\yvec^\suptp - \yvec^\supt)^T A (x^\suptp - x^\supt)
	&\geq -\frac{m \ltwos{x^\suptp - x^\supt}^2}{4\tau} 
    - \frac{\ltwos{\yvec^\suptp - \yvec^\supt}^2}{4\sigma}, \\
	(\yvec^\suptp - \yvec^\supt)^T A (x^\supt - x^\suptm)
    &\geq - \frac{m \ltwos{x^\supt - x^\suptm}^2}{4\tau} 
    - \frac{\ltwos{\yvec^\suptp - \yvec^\supt}^2}{4\sigma}.
\end{align*}
Combining the above two inequalities with~\eqref{eqn:pd-inequality-basic} 
and~\eqref{eqn:last-term-expansion}, we obtain
\begin{align}
& \frac{\ltwos{x^\supt - \xhat}^2}{2\tau} + \left( \frac{1}{2\sigma} + \frac{(n-m)\gamma}{2n} \right)\frac{\ltwos{\yvec^\supt - \yhat}^2}{m} \nonumber\\
& +\theta \bigl(f(x^\supt,y^\star) - f(x^\star, y^\star)\bigr)
 + \frac{n-m}{m}\left( f(x^\star, y^\star) - f(x^\star, y^\supt) \right) \nonumber\\
& + \theta \frac{\ltwos{x^\supt - x^\suptm}^2}{4\tau}
+ \theta \frac{(\yvec^\supt - \yhat)^T A (x^\supt - x^\suptm)}{n} \nonumber\\
\geq ~& \left( \frac{1}{2\tau} + \frac{\lambda}{2}\right)\E[\ltwos{x^\suptp - \xhat}^2 | \field_t] 
    + \left( \frac{1}{2\sigma} + \frac{\gamma}{2}\right) \frac{\E[\ltwos{\yvec^\suptp - \yhat}^2 | \field_t]}{m}   \nonumber \\
& + \E \left[ f(x^\suptp,y^\star) - f(x^\star, y^\star) + \frac{n}{m}\left(f(x^\star, y^\star) - f(x^\star, y^\suptp)\right) \bigg| \field_t \right] \nonumber \\
& + \frac{\E[\ltwos{x^\suptp-x^\supt}^2|\field_t]}{4\tau} 
 + \frac{\E[ (\yvec^\suptp - \yhat)^T A (x^\suptp - x^\supt) | \field_t]}{n}.
\label{eqn:pd-inequality-final}
\end{align}
Note that we have added the nonnegative term 
$\theta \bigl( f(x^\supt,y^\star) - f(x^\star, y^\star)\bigr)$ 
to the left-hand side in~\eqref{eqn:pd-inequality-final} to ensure that each
term on one side of the inequality has a corresponding term on the other side.

\bigskip

If the parameters $\tau$, $\sigma$, and $\theta$ are chosen as 
in~\eqref{eqn:tau-sigma-theta}, that is, 
\begin{align*}
	\tau = \frac{1}{R} \sqrt{\frac{m \gamma}{n\lambda}},\quad \sigma = \frac{1}{R} \sqrt{\frac{n\lambda}{m \gamma}}, \quad \mbox{and} \quad
	\theta = 1 - \frac{1}{(n/m) + R\sqrt{(n/m)/(\lambda\gamma)}},
\end{align*}
Then the ratios between the coefficients of the corresponding terms on both 
sides of the inequality~\eqref{eqn:pd-inequality-final} 
are either equal to~$\theta$ or bounded by~$\theta$.
More specifically,
\begin{align*}
  \frac{n-m}{m} \bigg/ \frac{n}{m} = 1 - \frac{m}{n} \leq \theta , \\
  \frac{1}{2\tau} \bigg/ \left(\frac{1}{2\tau} + \frac{\lambda}{2} \right) = 1 - \frac{1}{1+R\sqrt{(n/m)/(\lambda\gamma)}} \leq \theta, \\
  \left( \frac{1}{2\sigma} + \frac{(n-m)\gamma}{2n} \right) \bigg/ \left( \frac{1}{2\sigma} + \frac{\gamma}{2}\right) = 1 - \frac{1}{n/m+R\sqrt{(n/m)/(\lambda\gamma)}} = \theta .
\end{align*}
Therefore, if we define the following sequence,
\begin{align*}
\widetilde{\Delta}^\supt 
= ~& \left( \frac{1}{2\tau} + \frac{\lambda}{2}\right)\ltwos{x^\supt - \xhat}^2
+ \left( \frac{1}{2\sigma} + \frac{\gamma}{2}\right) \frac{\ltwos{\yvec^\supt - \yhat}^2}{m} \nonumber \\
& + f(x^\supt,y^\star) - f(x^\star, y^\star) + \frac{n}{m}\left(f(x^\star, y^\star) - f(x^\star, y^\supt)\right) \nonumber \\
& + \frac{\ltwos{x^\supt - x^\suptm}^2}{4\tau} + \frac{ (\yvec^\supt - \yhat)^T A (x^\supt - x^\suptm)}{n} ,
\end{align*}
then inequality~\eqref{eqn:pd-inequality-final} implies 
$\E\bigl[\widetilde{\Delta}^\suptp |\field_t\bigr]\leq\theta\, \widetilde{\Delta}^\supt$. 
Apply this relation recursively and taking expectation with respect to
all random variables up to time~$t$, we have 
\begin{equation} \label{eqn:Delta-all-converge}
\E\bigl[\widetilde{\Delta}^\supt\bigr]~\leq~\theta^t\,\widetilde{\Delta}^{(0)}.
\end{equation}
Comparing the definition of $\Delta^\supt$ in~\eqref{eqn:Delta-def}, we have
\begin{align}
\widetilde{\Delta}^\supt 
= ~& \Delta^\supt + \frac{\ltwos{\yvec^\supt - \yhat}^2}{4\sigma m} + \frac{\ltwos{x^\supt - x^\suptm}^2}{4\tau} + \frac{ (\yvec^\supt - \yhat)^T A (x^\supt - x^\suptm)}{n} .
\label{eqn:Delta-all-def}
\end{align}
For $t=0$, by letting $x^{(-1)}=x^{(0)}$, 
the last two terms in~\eqref{eqn:Delta-all-def} 
for $\widetilde{\Delta}^{(0)}$ disappears. 
Moreover, we can show that the sum of the last three terms 
in~\eqref{eqn:Delta-all-def} are nonnegative, 
and therefore we can replace $\widetilde{\Delta}^\supt$ 
with $\Delta^\supt$ on the left-hand side of~\eqref{eqn:Delta-all-converge}.
To see this, we bound the absolute value of the last term:
\begin{align*}
	\frac{\bigl|(\yvec^\supt - \yhat)^T A (x^\supt - x^\suptm)\bigr|}{n} 
	&\leq \frac{\ltwos{x^\supt - x^\suptm}^2}{4\tau} +  \frac{\|A\|_2^2\;\ltwos{\yvec^\supt - \yhat}^2}{n^2/\tau} \\
	&\leq \frac{\ltwos{x^\supt - x^\suptm}^2}{4\tau} +  \frac{n R^2 \ltwos{\yvec^\supt - \yhat}^2}{n^2/\tau} \\
	&= \frac{\ltwos{x^\supt - x^\suptm}^2}{4\tau} +  \frac{\ltwos{\yvec^\supt - \yhat}^2}{4 n \sigma} \\
	&\leq \frac{\ltwos{x^\supt - x^\suptm}^2}{4\tau} +  \frac{\ltwos{\yvec^\supt - \yhat}^2}{4 m \sigma} ,
\end{align*}
where in the second inequality we used $\|A\|_2^2\leq\|A\|_F^2\leq n R^2$, 
in the equality we used $\tau\sigma=1/(4R^2)$,
and in the last inequality we used $m\leq n$.
The above upper bound on absolute value implies
\[
	\frac{(\yvec^\supt - \yhat)^T A (x^\supt - x^\suptm)}{n} 
    ~\geq~
	-\frac{\ltwos{x^\supt - x^\suptm}^2}{4\tau} - \frac{\ltwos{\yvec^\supt - \yhat}^2}{4 m \sigma} .
\]
To summarize, we have proved
\[
  \E \left[ \Delta^\supt \right] \leq \theta^t\,
  \left(\Delta^{(0)} + \frac{\ltwos{y^{(0)}-y^\star}^2}{4m\sigma} \right),
\]
which is the desired result.

\section{Proof of Lemma~\ref{lem:gap-by-saddle}}
\label{sec:proof-gap-by-saddle}

  We can write $P(x)= F(x) + g(x)$ where
  \begin{align*}
    F(x)&= \frac{1}{n}\sum_{i=1}^n \phi_i(a_i^T x) 
    = \max_{y\in\R^n} \biggl\{ \frac{1}{n}y^T A x - \frac{1}{n}\sum_{i=1}^n \phi_i^*(y_i)\biggr\}. 
  \end{align*}
Assumption~\ref{asmp:smooth-convex} implies that $F(x)$ is smooth and 
$\nabla F(x)$ is Lipschitz continuous with constant 
$\|A\|_2^2/(n\gamma)$.
We can bound the spectral norm with the Frobenius norm, i.e.,
$\|A\|_2^2 \leq \|A\|_F^2 \leq n R^2$, which results in
$\|A\|_2^2/(n\gamma)\leq n R^2 /(n\gamma)=R^2/\gamma$.
By definition of the saddle point, 
the gradient of~$F$ at $x^\star$ is $\nabla F(x^\star)=(1/n)A^T y^\star$.
Therefore, we have
\begin{align*}
 F(x) &\leq F(x^\star) + \langle \nabla F(x^\star), x-x^\star\rangle
 + \frac{R^2}{2\gamma}\ltwos{x-x^\star}^2 \\
 &= \max_{y\in\R^n} \biggl\{ \frac{1}{n}y^T A x^\star - \frac{1}{n}\sum_{i=1}^n \phi_i^*(y_i)\biggr\} + \frac{1}{n}(y^\star)^T A (x-x^\star) + \frac{R^2}{2\gamma}\ltwos{x-x^\star}^2 \\
 &= \biggl\{\frac{1}{n}(y^\star)^T A x^\star - \frac{1}{n}\sum_{i=1}^n \phi_i^*(y^\star_i)\biggr\} + \frac{1}{n}(y^\star)^T A (x-x^\star) + \frac{R^2}{2\gamma}\ltwos{x-x^\star}^2 \\
 &= \frac{1}{n}(y^\star)^T A x - \frac{1}{n}\sum_{i=1}^n \phi_i^*(y^\star_i) + \frac{R^2}{2\gamma}\ltwos{x-x^\star}^2 .
\end{align*}
Combining the above inequality with $P(x)=F(x)+g(x)$, we have
\begin{align*}
 P(x) &\leq \frac{1}{n}(y^\star)^T A x - \frac{1}{n}\sum_{i=1}^n \phi_i^*(y^\star_i) + \frac{R^2}{2\gamma}\ltwos{x-x^\star}^2 + g(x)
 = \spf(x, y^\star) + \frac{R^2}{2\gamma}\ltwos{x-x^\star}^2 ,
\end{align*}
which is the first desired inequality.

Similarly, the second inequality can be shown by first writing
$D(y)=-\frac{1}{n}\sum_{i=1}^n \phi_i^*(y_i) - G^*(y)$, where
\begin{align*}
  G^*(y) = g^*\biggl(-\frac{1}{n} A^T y\biggr)
  = \max_{x\in\R^d} \biggl\{ -\frac{1}{n} x^T A^T y - g(x) \biggr\}.
\end{align*}
In this case, $\nabla G^*(y)$ is Lipschitz continuous with
constant $\|A\|_2^2/(n^2\lambda)\leq n R^2/(n^2\lambda)=R^2/(n\lambda)$.
Again by definition of the saddle-point, we have 
$\nabla G^*(y^\star) = - (1/n)A x^\star$. Therefore,
\begin{align*}
 G^*(y) & \leq G^*(y^\star) + \langle \nabla G^*(y^\star), y-y^\star\rangle
 + \frac{R^2}{2n\lambda}\ltwos{y-y^\star}^2  \\
 &= \max_{x\in\R^d} \biggl\{ -\frac{1}{n} x^T A^T y^\star - g(x) \biggr\}
 - \frac{1}{n} (y-y^\star)^T A x^\star + \frac{R^2}{2n\lambda}\ltwos{y-y^\star}^2  \\
 &= \biggl\{ -\frac{1}{n} (x^\star)^T A^T y^\star - g(x^\star) \biggr\}
 - \frac{1}{n} (y-y^\star)^T A x^\star + \frac{R^2}{2n\lambda}\ltwos{y-y^\star}^2  \\
 &= - \frac{1}{n} y^T A x^\star - g(x^\star) + \frac{R^2}{2n\lambda}\ltwos{y-y^\star}^2  .
\end{align*}
Recalling that $D(y)=-\frac{1}{n}\sum_{i=1}^n\phi_i^*(y_i) - G^*(y)$, we conclude with
\begin{align*}
 D(y) \geq  -\frac{1}{n}\phi_i^*(y_i) + \frac{1}{n} y^T A x^\star + g(x^\star) - \frac{R^2}{2n\lambda}\ltwos{y-y^\star}^2  
  = \spf(x^\star, y) - \frac{R^2}{2n\lambda}\ltwos{y-y^\star}^2 .
\end{align*}
This finishes the proof.

\section{Proof of Theorem~\ref{thm:nonuniform-spdc-convergence}}

The proof of Theorem~\ref{thm:nonuniform-spdc-convergence} follows similar 
steps for proving Theorem~\ref{thm:spdc-convergence}. 
We start by establishing relation between $(\yvec^\supt,\yvec^\suptp)$ 
and between $(x^\supt,x^\suptp)$. 
Suppose that the quantity $\ytilde_i$ minimizes the function
$\phistar_{i}(\beta) - \beta \langle a_i, \xbar^\supt \rangle + \frac{p_i n}{2\sigma} (\beta-\yvec_i^\supt)^2$. Also notice that
$\phi_i^*(\beta) - \beta\langle a_i, x^* \rangle $ is a $\gamma$-strongly convex function minimized by $y_i^*$, which implies
\begin{align}
	\phi^*_i(\ytilde_i) - \ytilde_i \langle a_i, x^* \rangle
	\geq \phi^*_i(y_i^*) - y_i^* \langle a_i, x^* \rangle
	+ \frac{\gamma}{2} (\ytilde_i - y_i^*)^2.\label{eqn:prf-2-aux-1}
\end{align}
Then, following the same argument for establishing inequality~\eqref{eqn:y-update-sc} and plugging in inequality~\eqref{eqn:prf-2-aux-1},
we obtain
\begin{align}\label{eqn:nonuniform-alpha-update-inequality-basic}
 \frac{p_i n}{2\sigma}(\yvec_i^\supt - \yhat_i)^2 \geq \left( \frac{p_i n}{2\sigma} + \gamma\right)(\ytilde_i - \yhat_i)^2 + \frac{p_i n(\ytilde_i - \yvec_i^\supt)^2}{2\sigma} + \langle a_i, \xhat - \xbar^\supt \rangle (\ytilde_i - \yhat_i ).
\end{align}
Note that $i=k$ with probability $p_i$. Therefore, we have
\begin{align*}
	(\ytilde_i - \yhat_i)^2 &= \frac{1}{p_i} \E[(\yvec_i^\suptp - \yhat_i)^2 | \field_t] - \frac{1-p_i}{p_i} (\yvec_i^\supt - \yhat_i)^2,\\
	(\ytilde_i  - \yvec_i^\supt)^2 &= \frac{1}{p_i} \E[(\yvec_i^\suptp - \yvec_i^\supt)^2 | \field_t],\\
	\ytilde_i &= \frac{1}{p_i} \E[\yvec_i^\suptp | \field_t] - \frac{1-p_i}{p_i} \yvec_i^\supt,
\end{align*}
where $\field_t$ represents the sigma field generated by all random variables defined before iteration $t$.
Substituting the above equations into inequality~\eqref{eqn:nonuniform-alpha-update-inequality-basic}, and averaging over
$i=1,2,\dots,n$, we have
\begin{align}
	 \sum_{i=1}^n \left( \frac{1}{2\sigma} + \frac{(1-p_i)\gamma}{p_i n} \right)(\yvec_i^\supt - \yhat_i)^2 & \geq \sum_{i=1}^n
	\left( \frac{1}{2\sigma} + \frac{\gamma}{p_i n}\right)
	\E[(\yvec_i^\suptp - \yhat_i)^2 | \field_t] + \frac{\E[(\yvec_k^\suptp - \yvec_k^\supt)^2| \field_t]}{2\sigma}\nonumber\\
    &\quad + \E\Bigl[ \Bigl\langle (\uvec^\supt  - \uhat) + \frac{1}{p_k}(u^\suptp-u^\supt), ~ \xhat - \xbar^\supt \Bigr\rangle \Big| \field_t \Bigr],\label{eqn:nonuniform-alpha-expected-inequality-final}
\end{align}
where $\uhat =\frac{1}{n}\sum_{i=1}^n \yhat_i a_i$ and $\uvec^\supt = \frac{1}{n}\sum_{i=1}^n \yvec_i^\supt a_i$
have the same definition as in the proof of Theorem~\ref{thm:spdc-convergence}. 

For the relation between $x^\supt$ and $x^\suptp$, 
we first notice that
$\langle u^*, x\rangle + g(x)$ is a $\lambda$-strongly convex function minimized by $x^*$, which implies
\begin{align}\label{eqn:prf-2-aux-2}
	\langle u^*, x^\suptp\rangle + g(x^\suptp) \geq
	 \langle u^*, x^*\rangle + g(x^*) _+ \frac{\lambda}{2}(x^\suptp - x^*)^2 .
\end{align}
Following the same argument for establishing inequality~\eqref{eqn:thm-1-aux-1} and plugging in inequality~\eqref{eqn:prf-2-aux-2}, we obtain
\begin{align}\label{eqn:nonuniform-x-update-inequality-basic}
 \frac{\ltwos{x^\supt - \xhat}^2}{2\tau} &\geq \left( \frac{1}{2\tau} + \lambda\right)\ltwos{x^\suptp - \xhat}^2 + \frac{\ltwos{x^\suptp - x^\supt}^2}{2\tau}\nonumber\\
&\qquad + \Big\langle (\uvec^\supt - \uhat )+ \frac{1}{p_k}(u^\suptp-u^\supt),~ x^\suptp - \xhat \Big\rangle.
\end{align}
Taking expectation over both sides of inequality~\eqref{eqn:nonuniform-x-update-inequality-basic} and adding it to inequality~\eqref{eqn:nonuniform-alpha-expected-inequality-final}
yields
\begin{align}
	&\frac{\ltwos{x^\supt - \xhat}^2}{2\tau} + \sum_{i=1}^n \left( \frac{1}{2\sigma} + \frac{(1-p_i)\gamma}{p_i n} \right)(\yvec_i^\supt - \yhat_i)^2 \geq \left( \frac{1}{2\tau} + \lambda\right)\E[ \ltwos{x^\suptp - \xhat}^2 | \field_t] \nonumber\\
	&\qquad +  \sum_{i=1}^n \left( \frac{1}{2\sigma} + \frac{\gamma}{p_i n}\right)
	\E[(\yvec_i^\suptp - \yhat_i)^2 | \field_t] + \frac{\ltwos{x^\suptp - x^\supt}^2}{2\tau} + \frac{\E[(\yvec_k^\suptp - \yvec_k^\supt)^2| \field_t]}{2\sigma} \nonumber\\
	&\qquad  + \E\Big[ \underbrace{\Big( \frac{(\yvec^\supt  - \yhat)^{T}A}{n} + \frac{(\yvec_k^\suptp - \yvec_k^\supt)a_k^T}{p_k n}\Big) ( (x^\suptp - x^\supt)
	- \theta (x^\supt - x^\suptm))}_{v}  \Big| \field_t\Big] ,
    \label{eqn:nonuniform-pd-inequality-basic}
\end{align}
where the matrix $A$ is a $n$-by-$d$ matrix, whose $i$-th row is equal to the vector $a_i^{T}$.

\bigskip

Next, we lower bound the last term on the right-hand side of inequality~\eqref{eqn:nonuniform-pd-inequality-basic}.
Indeed, it can be expanded as
\begin{align}
 v &= \frac{(\yvec^\suptp - \yhat)^T A (x^\suptp - x^\supt)}{n} \nonumber - \frac{\theta(\yvec^\supt - \yhat)^T A (x^\supt - x^\suptm)}{n} \\
 & \qquad + \frac{1 - p_k}{p_k n}(\yvec_k^\suptp - \yvec_k^\supt)a_k^{T} (x^\suptp - x^\supt)
	- \frac{\theta}{p_k n}(\yvec_k^\suptp - \yvec_k^\supt)a_k^{T} (x^\supt - x^\suptm)\label{eqn:nonuniform-last-term-expansion}.
\end{align}
Note that the probability $p_k$ given in~\eqref{eqn:nonuniform-prob} satisfies 
\[
    p_k \geq \alpha\, \frac{\ltwos{a_k}}{ \sum_{i=1}^n \ltwos{a_i}} 
    = \alpha\, \frac{\ltwos{a_k}}{n \barR}, \qquad k=1, \dots, n.
\]
Since the parameters $\tau$ and $\sigma$ satisfies 
$\sigma\tau \barR^2=\alpha^2/4$, 
we have $p_k^2 n^2 / \tau \geq 4\sigma \|a_k\|_2^2$ and consequently
\begin{align*}
\frac{|(\yvec_k^\suptp - \yvec_k^\supt)a_k^{T} (x^\suptp - x^\supt)|}{p_k n} &\leq \frac{\ltwos{x^\suptp - x^\supt}^2}{4\tau} + \frac{\ltwos{(\yvec_k^\suptp - \yvec_k^\supt)a_k}^2}{p_k^2 n^2 / \tau} \\
&\leq \frac{\ltwos{x^\suptp - x^\supt}^2}{4\tau} + \frac{(\yvec_k^\suptp - \yvec_k^\supt)^2}{4\sigma}.
\end{align*}
Similarly, we have
\begin{align*}
	\frac{|(\yvec_k^\suptp - \yvec_k^\supt)a_k^{T} (x^\supt - x^\suptm)|}{p_k n} \leq \frac{\ltwos{x^\supt - x^\suptm}^2}{4\tau} + \frac{(\yvec_k^\suptp - \yvec_k^\supt)^2}{4\sigma}.
\end{align*}
Combining the above two inequalities with lower bounds~\eqref{eqn:nonuniform-pd-inequality-basic} and~\eqref{eqn:nonuniform-last-term-expansion}, we obtain
\begin{align}
   & \frac{\ltwos{x^\supt - \xhat}^2}{2\tau} + \sum_{i=1}^n \left( \frac{1}{2\sigma} + \frac{(1-p_i)\gamma}{p_i n} \right)(\yvec_i^\supt - \yhat_i)^2 \geq
	\left( \frac{1}{2\tau} + \lambda\right)\E[\ltwos{x^\suptp - \xhat}^2 | \field_t] \nonumber\\
	& \qquad + \sum_{i=1}^n \left( \frac{1}{2\sigma} + \frac{\gamma}{p_i n}\right)
	\E[(\yvec_i^\suptp - \yhat_i)^2 | \field_t]  + \frac{\E[\ltwos{x^\suptp - x^\supt}^2 | \field_t] - \theta \ltwos{x^\supt - x^\suptm}^2}{4\tau}\nonumber\\
	& \qquad + \frac{\E[ (\yvec^\suptp - \yhat)^T A (x^\suptp - x^\supt) | \field_t] - \theta (\yvec^\supt - \yhat)A (x^\supt - x^\suptm)}{n}.
	\label{eqn:nonuniform-pd-inequality-final}
\end{align}
Recall that the parameters $\tau$, $\sigma$, and $\theta$ are chosen to be
\begin{align*}
	\tau = \frac{\alpha}{2\barR} \sqrt{\frac{\gamma}{n\lambda}},\quad 
    \sigma =\frac{\alpha}{2\barR}\sqrt{\frac{n\lambda}{\gamma}},\quad\mbox{and}\quad
    \theta = 1 - \left(\frac{n}{1-\alpha} + \frac{\barR}{\alpha}\sqrt{\frac{n}{\lambda\gamma}}\right)^{-1}.
\end{align*}
Plugging in these assignments and using the fact that $p_i \geq \frac{1-\alpha}{n}$, we find that
\begin{align*}
	&\frac{1/(2\tau)}{1/(2\tau) + \lambda} = 1 - \Big( 1 + \frac{1}{2\tau\lambda}\Big)^{-1} = 
	1 - \Big( 1 + \frac{\barR}{\alpha}\sqrt{\frac{n}{\lambda\gamma}}\Big)^{-1}
	\leq \theta \quad \mbox{and}\\
	& \frac{1/(2\sigma) + (1-p_i)\gamma/(p_i n)}{1/(2\sigma) + \gamma/ (p_i n)} =
	1 - \Big( \frac{1}{p_i} + \frac{n}{2\sigma} \Big)^{-1}
	\\
	&\qquad\qquad\qquad\leq 1 - \Big( \frac{n}{1-\alpha} + \frac{n}{2\sigma\gamma} \Big)^{-1} = \theta  \quad \mbox{for $i=1,2,\dots,n$} .
\end{align*}
Therefore, if we define a sequence $\Delta^\supt$ such that
\begin{align*}
	\Delta^\supt &= \left( \frac{1}{2\tau} + \lambda\right)\E[\ltwos{x^\supt - \xhat}^2] + \sum_{i=1}^n\left( \frac{1}{2\sigma} + \frac{\gamma}{p_i n} \right)
	\E[(\yvec_i^\supt - \yhat_i)^2]\\
	&\qquad + \frac{\E[\ltwos{x^\supt - x^\suptm}^2]}{4\tau} + \frac{\E[ (\yvec^\supt - \yhat)^T A (x^\supt - x^\suptm)]}{n},
\end{align*}
then inequality~\eqref{eqn:nonuniform-pd-inequality-final} implies the recursive relation $\Delta^\suptp \leq \theta \cdot \Delta^\supt$,
which implies
\begin{align}
&\left(\frac{1}{2\tau}+\lambda\right) \E[\ltwos{x^\supt - \xhat}^2] 
+ \left(\frac{1}{2\sigma} + \frac{\gamma}{n}\right)
	\E[\ltwos{\yvec^\supt - \yhat}^2] \nonumber\\
&\quad + \frac{\E[\ltwos{x^\supt - x^\suptm}^2]}{4\tau} 
  + \frac{\E[(\yvec^\supt - \yhat)^T A (x^\supt - x^\suptm)]}{n} 
~\leq~ \theta^{t} \Delta^{(0)} ,
    \label{eqn:nonuniform-almost-there}
\end{align}
where
\begin{align*}
    \Delta^{(0)} &= \left( \frac{1}{2\tau} + \lambda\right)\ltwos{x^{(0)}-\xhat}^2 + \sum_{i=1}^n\left( \frac{1}{2\sigma} + \frac{\gamma}{p_i n} \right)
    (\yvec_i^{(0)} - \yhat_i)^2 \\
    &\leq  \left( \frac{1}{2\tau} + \lambda\right)\ltwos{x^{(0)}-\xhat}^2 + \left( \frac{1}{2\sigma} + \frac{\gamma}{1-\alpha} \right)\ltwos{y^{(0)}-\yhat}^2.
\end{align*}
To eliminate the last two terms on the left-hand side of inequality~\eqref{eqn:nonuniform-almost-there}, we notice that
\begin{align*}
	\frac{|(\yvec^\supt - \yhat)^T A (x^\supt - x^\suptm)|}{n} &
	\leq \frac{\ltwos{x^\supt - x^\suptm}^2}{4\tau}
	+ \frac{\ltwos{\yvec^\supt - \yhat}^2\,\ltwos{A}^2}{n^2/\tau}\\
	&\leq  \frac{\ltwos{x^\supt - x^\suptm}^2}{4\tau}
	+ \frac{\ltwos{\yvec^\supt - \yhat}^2 \,\| A\|_F^2}{n^2/\tau}\\
	&=  \frac{\ltwos{x^\supt - x^\suptm}^2}{4\tau}
	+ \frac{\ltwos{\yvec^\supt - \yhat}^2 \, \sum_{i=1}^n \ltwos{a_i}^2}{(4/\alpha^2) \sigma (\sum_{i=1}^n \ltwos{a_i})^2}
	\\
	&\leq \frac{\ltwos{x^\supt - x^\suptm}^2}{4\tau} + \frac{\ltwos{\yvec^\supt - \yhat}^2}{4\sigma},
\end{align*}
where in the equality we used 
$n^2/\tau=(4/\alpha^2)\sigma n^2\barR^2 = (4/\alpha^2) \sigma\left(\sum_{i=1}^n\|a_i\|_2\right)^2$.
This implies
\[
	\frac{(\yvec^\supt - \yhat)^T A (x^\supt - x^\suptm)}{n} 
	~\geq~ -\frac{\ltwos{x^\supt - x^\suptm}^2}{4\tau} 
    - \frac{\ltwos{\yvec^\supt - \yhat}^2}{4\sigma}.
\]
Substituting the above inequality into 
inequality~\eqref{eqn:nonuniform-almost-there} completes the proof.


\section{Efficient update for $(\ell_1+\ell_2)$-norm penalty}
\label{sec:l1l2-update-detail}

From Section~\ref{sec:l1+l2-penalty}, 
we have the following recursive formula for $t\in [t_0+1, t_1]$, 
\begin{align}\label{eqn:l1l2-norm-x-update-zero}
x_j^\suptp &= \left\{ 
	\begin{array}{ll}
        \frac{1}{1 + \lambda_2 \tau} (x_j^\supt - \tau \uvec^{(t_0+1)} - \tau\lambda_1) & \mbox{if } x_j^\supt - \tau \uvec^{(t_0+1)}_j > \tau\lambda_1 ,\\
        \frac{1}{1 + \lambda_2 \tau} (x_j^\supt - \tau \uvec^{(t_0+1)} + \tau\lambda_1) & \mbox{if } x_j^\supt - \tau \uvec^{(t_0+1)}_j < -\tau\lambda_1 ,\\
		0 & \mbox{otherwise.}
	\end{array}
\right.
\end{align}
Given $x_j^{(t_0+1)}$ at iteration $t_0$, we present an efficient algorithm for calculating $x_j^{(t_1)}$. 
We begin by examining the sign of $x_j^{(t_0+1)}$.
\vspace{-10pt}

\paragraph{Case I ($x_j^{(t_0+1)} = 0$):} If $- \uvec_j^{(t_0+1)} > \lambda_1$, then equation~\eqref{eqn:l1l2-norm-x-update-zero}
implies $x_j^{(t)} > 0$ for all $t > t_0+1$. Consequently, we have a closed-form formula for $x_j^{(t_1)}$:
\begin{align}\label{eqn:l1l2-case1-1}
	x^{(t_1)}_j = \frac{1}{(1+\lambda_2 \tau)^{t_1-t_0-1}} \Big( x^{(t_0+1)}_j + \frac{\uvec_j^{(t_0+1)}+\lambda_1}{\lambda_2}\Big) - \frac{\uvec_j^{(t_0+1)}+\lambda_1}{\lambda_2}.
\end{align}
If $- \uvec_j^{(t_0+1)} < -\lambda_1$, then equation~\eqref{eqn:l1l2-norm-x-update-zero}
implies $x_j^{(t)} < 0$ for all $t > t_0+1$. 
Therefore, we have the closed-form formula:
\begin{align}\label{eqn:l1l2-case1-2}
	x^{(t_1)}_j = \frac{1}{(1+\lambda_2 \tau)^{t_1-t_0-1}} \Big( x^{(t_0+1)}_j + \frac{\uvec_j^{(t_0+1)}-\lambda_1}{\lambda_2}\Big) - \frac{\uvec_j^{(t_0+1)}-\lambda_1}{\lambda_2}.
\end{align}
Finally, if $- \uvec_j^{(t_0+1)} \in [-\lambda_1,\lambda_1]$, then equation~\eqref{eqn:l1l2-norm-x-update-zero} implies $x^{(t_1)}_j = 0$.

\paragraph{Case II ($x_j^{(t_0+1)} > 0$):} If $- \uvec_j^{(t_0+1)} \geq \lambda_1$, then it is easy to verify that 
$x_j^{(t_1)}$ is obtained by equation~\eqref{eqn:l1l2-case1-1}. Otherwise,
We use the recursive formula~\eqref{eqn:l1l2-norm-x-update-zero} to derive the latest time $t^+ \in [t_0+1,t_1]$
such that $x_j^{t^+}>0$ is true. Indeed, since $x_j^{(t)} > 0$ for all $t\in [t_0+1,t^+]$, we have a closed-form formula for $x_j^{t^+}$:
\begin{align}\label{eqn:l1l2-case2-1}
	x^{t^+}_j = \frac{1}{(1+\lambda_2 \tau)^{t^+-t_0-1}} \Big( x^{(t_0+1)}_j + \frac{\uvec_j^{(t_0+1)} + \lambda_1}{\lambda_2}\Big) - \frac{\uvec_j^{(t_0+1)}+\lambda_1}{\lambda_2}.
\end{align}
We look for the largest $t^+$ such that the right-hand  side of equation~\eqref{eqn:l1l2-case2-1}
is positive, which is equivalent of 
\begin{align}\label{eqn:l1l2-case2-2}
	t^+ - t_0 - 1 < {\log\Big( 1 + \frac{\lambda_2 x^{(t_0+1)}_j }{\uvec_j^{(t_0+1)} + \lambda_1} \Big)} / {\log(1 + \lambda_2 \tau)}.
\end{align}
Thus, $t^+$ is the largest integer in $[t_0+1,t_1]$ such that inequality~\eqref{eqn:l1l2-case2-2} holds.
If $t^+ = t_1$, then $x_j^{(t_1)}$ is obtained by~\eqref{eqn:l1l2-case2-1}.
Otherwise, we can calculate $x_j^{t^+ + 1}$ by formula~\eqref{eqn:l1l2-norm-x-update-zero},
then resort to Case I or Case III, treating $t^+$ as $t_0$.

\paragraph{Case III ($x_j^{(t_0+1)} < 0$):} If $- \uvec_j^{(t_0+1)} \leq -\lambda_1$, then
$x_j^{(t_1)}$ is obtained by equation~\eqref{eqn:l1l2-case1-2}. Otherwise, we calculate the largest
integer $t^-\in [t_0+1,t_1]$ such that $x_j^{t^-} < 0$ is true. Using the same argument as for Case II, we have
the closed-form expression
\begin{align}\label{eqn:l1l2-case3-1}
	x^{t^-}_j = \frac{1}{(1+\lambda_2 \tau)^{t^- -t_0-1}} \Big( x^{(t_0+1)}_j + \frac{\uvec_j^{(t_0+1)} - \lambda_1}{\lambda_2}\Big) - \frac{\uvec_j^{(t_0+1)} - \lambda_1}{\lambda_2}.
\end{align}
where $t^-$ is the largest integer in $[t_0+1,t_1]$ such that the following
inequality holds:
\begin{align}\label{eqn:l1l2-case3-2}
	t^- - t_0 - 1 < {\log\Big( 1 + \frac{\lambda_2 x^{(t_0+1)}_j }{\uvec_j^{(t_0+1)} - \lambda_1} \Big)} / {\log(1 + \lambda_2 \tau)}.
\end{align}
If $t^- = t_1$, then $x_j^{(t_1)}$ is obtained by~\eqref{eqn:l1l2-case3-1}.
Otherwise, we can calculate $x_j^{t^- + 1}$ by formula~\eqref{eqn:l1l2-norm-x-update-zero},
then resort to Case I or Case II, treating $t^-$ as $t_0$.

\bigskip

Finally, we note that formula~\eqref{eqn:l1l2-norm-x-update-zero} implies the monotonicity of $x_j^{(t)}~(t=t_0+1,t_0+2,\dots)$.
As a consequence, the procedure of either Case I, Case II or Case III is executed for at most once.
Hence, the algorithm for calculating $x_j^{(t_1)}$ has $\order(1)$ time complexity.

\bibliographystyle{abbrv}
\bibliography{spdc_paper2}

\end{document}